\documentclass[a4paper]{amsart}

\usepackage[all]{xy}
\usepackage{amsmath}
\usepackage{amssymb}
\usepackage{latexsym}
\usepackage{amsthm}
\usepackage{mathrsfs}
\usepackage{float}
\usepackage{mathdots} 
\usepackage{geometry}
\usepackage{enumerate}
\usepackage{mathtools}

\newcommand{\cA}{\mathcal{A}}
\newcommand{\fa}{\mathfrak{a}}

\newcommand{\cB}{\mathcal{B}}
\newcommand{\fb}{\mathfrak{b}}

\newcommand{\C}{\mathbb{C}}

\newcommand{\fc}{\mathfrak{c}}

\newcommand{\fD}{\mathfrak{D}}
\newcommand{\fd}{\mathfrak{d}}

\newcommand{\be}{\mathbf{e}}
\newcommand{\cE}{\mathcal{E}}

\renewcommand{\bf}{\mathbf{f}}
\newcommand{\cF}{\mathcal{F}}

\newcommand{\fg}{\mathfrak{g}}

\renewcommand{\H}{\mathbb{H}}
\newcommand{\cH}{\mathcal{H}}

\newcommand{\K}{\mathbb{K}}

\renewcommand{\L}{\mathbb{L}}

\newcommand{\rM}{\operatorname{M}}
\newcommand{\fm}{\mathfrak{m}}

\newcommand{\fp}{\mathfrak{p}}

\newcommand{\R}{\mathbb{R}}

\newcommand{\cS}{\mathcal{S}}
\newcommand{\fS}{\mathfrak{S}}
\newcommand{\fs}{\mathfrak{s}}

\newcommand{\cT}{\mathcal{T}}

\newcommand{\cV}{\mathcal{V}}

\newcommand{\W}{\mathbb{W}}
\newcommand{\cW}{\mathcal{W}}
\newcommand{\fw}{\mathfrak{w}}

\newcommand{\X}{\mathbb{X}}
\newcommand{\cX}{\mathcal{X}}

\newcommand{\Y}{\mathbb{Y}}
\newcommand{\cY}{\mathcal{Y}}

\newcommand{\Z}{\mathbb{Z}}
\newcommand{\cZ}{\mathcal{Z}}

\newcommand{\Ad}{\operatorname{Ad}}

\newcommand{\Aut}{\operatorname{Aut}}

\newcommand{\Id}{\operatorname{Id}}
\newcommand{\Hom}{\operatorname{Hom}}
\newcommand{\GL}{\operatorname{GL}}

\renewcommand{\Im}{\operatorname{Im}}

\newcommand{\Ind}{\operatorname{Ind}}

\newcommand{\inv}{\operatorname{inv}}
\newcommand{\sgn}{\operatorname{sgn}}
\newcommand{\SL}{\operatorname{SL}}

\newcommand{\diag}{\operatorname{diag}}
\newcommand{\Mp}{\operatorname{Mp}}

\newcommand{\Sym}{\operatorname{Sym}}
\newcommand{\SO}{\operatorname{SO}}

\newcommand{\Sp}{\operatorname{Sp}}
\newcommand{\fsp}{\mathfrak{sp}}
\newcommand{\spl}{\operatorname{spl}}
\newcommand{\std}{\operatorname{std}}

\newcommand{\Tr}{\operatorname{Tr}}

\newcommand{\End}{\operatorname{End}}
\newcommand{\op}{\operatorname{op}}

\newcommand{\btru}{\blacktriangle}
\newcommand{\btrd}{\blacktriangledown}
\newcommand{\la}{\langle}
\newcommand{\ra}{\rangle}
\newcommand{\RIT}{\mathcal{RIT}}
\newcommand{\wPi}{\widetilde{\Pi}}

\newtheorem{df}{Definition}[section]
\newtheorem{thm}[df]{Theorem}
\newtheorem{prop}[df]{Proposition}
\newtheorem{lem}[df]{Lemma}
\newtheorem{cor}[df]{Corollary}
\newtheorem{conj}[df]{Conjecture}

\newtheorem{rem}[df]{Remark}

\newtheorem{fact}[df]{Fact}
\newtheorem{hyp}[df]{Hypothesis}

\newtheorem*{prop*}{Proposition}

\makeatletter
    
    \@addtoreset{equation}{section}
\makeatother

\title{Local theta correspondences and Langlands parameters for rigid inner twists}

\author{
	Hirotaka KAKUHAMA
	}
\email{
	kakuhama@math.sci.hokudai.ac.jp
	}
\address{
	Faculty of Science, Hokkaido University, Kita-10 Nishi-8, Kita-ku, Sapporo, JAPAN
	}

\date{
\today
}

\begin{document}

\maketitle

\begin{abstract}
In this paper, we formulate a conjecture that describes the local theta correspondences in terms of the local Langland correspondences for rigid inner twists, which contain the correspondences for quaternionic dual pairs. Moreover, we verify the conjecture holds in some specific cases.
\end{abstract}

\setcounter{tocdepth}{1}
\tableofcontents



\section{
		Introduction
		}\label{intro}

Since a certain unitary representation of the metaplectic group, called the Weil representation at present, was organized by Weil \cite{Wei64}, it has been playing important roles in the representation theory. In particular, the theta correspondence, the correspondence of representations defined by using Weil representation, has become one of the main tools in the theory of automorphic representations. Besides, the local Langlands conjecture, a classification theory of representations, has been developed steadily. Therefore, it is natural to ask how the local theta correspondence is described in terms of Langlands parameters. For symplectic-orthogonal dual pairs and unitary-unitary dual pairs with certain conditions of ranks, Prasad conjectured the formula of the description \cite{Pra93}\cite{Pra00}. For the part of the behavior of $L$-parameters, he assembled and generalized some known works \cite{Ada89}\cite{HKS96}. Moreover, he also conjectured the behavior of the internal structure of $L$-packets. We remark that the work of Adams \cite{Ada89} is a conjecture for Arthur packets over $\R$ (see also  \cite{Moe11}, \cite[\S15.1]{GI14}). The Prasad conjecture over a $p$-adic field is proved by Atobe \cite{Ato18}, Atobe-Gan \cite{AG17}, Gan-Ichino \cite{GI16}, and extended by Atobe-Gan \cite{AG17b} to the description over a $p$-adic field without the rank conditions. In the Archimedean case, the local theta correspondence is described in terms of parameters generalizing Harish-Chandra parameters by many researchers \cite{KV78}\cite{Moe89}\cite{Li89}\cite{Pau98}\cite{Pau00}\cite{Pau05}\cite{LPTZ03}\cite{Ich22}.

The formulation of the local Langlands correspondence for the rigid inner twists given by Kaletha \cite{Kal16} allows us to discuss the description of the local theta correspondence for quaternionic dual pairs in terms of Langlands parameters. This is the main theme of this paper. Here, we briefly summarize the local Langlands correspondence. Let $F$ be a local field of characteristic $0$, let $G^\#$ be a connected reductive group over $F$, and let $Z$ be a finite central subgroup of $G^\#$. In \cite{Kal16}, Kaletha defined the set $Z^1(u \rightarrow \cW, Z \rightarrow G^\#)$ which surjects on $Z^1(\Gamma, G^\#/Z)$ if $Z$ is sufficiently large. A rigid inner twist is a pair $(z, \varphi)$ where $z \in Z^1(u \rightarrow \cW, Z \rightarrow G^\#)$ and $\varphi$ is an isomorphism from $G^\#$ onto $G$ over $\overline{F}$ such that $\varphi^{-1}\circ \sigma \circ \varphi \circ \sigma^{-1} = \overline{z}(\sigma)$ for $\sigma \in \Gamma$ where $\overline{z}$ denotes the image of $z$ in $Z^1(\Gamma, G^\#/Z)$. We fix a Whittaker data $\fw$. For a tempered $L$-parameter $\phi$ of $G$, the local Langlands conjecture claims that there is a set $\Pi_\phi(G)$ of irreducible representations of $G(F)$ and that there is an injective map
\[
\iota_\phi[\fw, z, \varphi] \colon \Pi_\phi(G) \rightarrow {\rm Irr}(\cS_\phi^+)
\]
characterized by certain character relations formulated in the theory of endoscopy. Here, $\cS_\phi^+$ is the $S$-group of $\phi$.
If an irreducible tempered representation $\pi$ of $G(\R)$ is contained in $\Pi_\phi(G)$, then we call the pair $(\phi, \iota_\phi[\fw, z, \varphi](\pi))$ the Langlands parameter of $\pi$.

As the notation indicates, the Langlands parameter of $\pi$ depends on the choice of the Whittaker data $\fw$ and the rigid inner twist $(z, \varphi)$ except for the irreducible tempered representation $\pi$ of $G(F)$. On the other hand, the local theta correspondence for the reductive dual pair $(G,G')$ depends on a fixed non-trivial additive character of $F$ and an equivalent class of the embeddings of $G(F)\times G'(F)$ into a Metaplectic group that is strictly finer than the isomorphism class of $G \times G'$. We are required to discuss these dependencies comprehensively.

For example, we focus on the orthogonal-symplectic dual pairs discussed by Prasad \cite{Pra93}. Let $Q$ be a $2n$-dimensional quadratic space over $F$, and let $U$ be a $2m$-dimensional symplectic space over $F$. Then, the local theta correspondence for ${\rm O}(a\cdot Q) \times \Sp(U)$ depends on the scalar $a \in F^\times$ in general in spite that the orthogonal group ${\rm O}(a\cdot Q)$ does not (see the second remark in \S5 of \cite{Pra93}). In this case, we can construct a pure inner twist $(t_Q, \varphi_Q)$ from $Q$ which behaves covariantly with the local theta correspondence as follows. Let $Q^\#$ be a $2n$-dimensional quadratic space so that ${\rm O}(Q^\#)$ is a quasi-split inner form of ${\rm O}(Q)$. For an isometry $f$ from $Q^\#\otimes\overline{F}$ onto $Q\otimes\overline{F}$,  we define $t_f(\sigma) = f^{-1}\circ\sigma\circ f \circ \sigma \in {\rm SO}(Q^\#)(\overline{F})$. We denote by $\varphi_f$ the isomorphism from ${\rm O}(Q^\#)$ onto ${\rm O}(Q)$ satisfying $f(gx) = \varphi_f(g)f(x)$ for $g \in {\rm O}(Q^\#)(\overline{F})$ and $x \in Q^\#$. Since a change of $f$ does not affect the Langlands parameter $\iota_\phi[\fw, t_f, \varphi_f]$ (c.f. Proposition \ref{parameter orth}), we may denote it by $(t_Q, \varphi_Q)$, which is the pure inner twist that we want. The same framework is available for the unitary-unitary dual pairs. However, it seems to be difficult for the quaternionic dual pairs.

In this paper, we will construct a more general framework to control the dependencies of the local theta correspondences and the Langlands parameters. We explain it for quaternionic dual pairs, for example. This is done in two steps. Let $D$ be a division quaternion algebra over $F$, let $V$ be a right $D$-vector space equipped with an $\epsilon$-Hermitian form $( \ , \ )$, and let $W$ be a left $D$-vector space equipped with a $(-\epsilon)$-Hermitian form $\la \ , \ \ra$. Moreover, we consider a $2m$-dimensional symplectic space $V^\#$, and a $2n$-dimensional quadratic space $W^\#$ so that ${\rm O}(W^\#)$ is quasi-split, and the discriminant of $W^\#$ coincides with that of $W$. We denote by $G(V)$ (resp. $G(W)$) the unitary group of $V$ (resp. $W$). The first step is to define a set 
\[
\RIT^\star(V^\#, V)
\]
of the rigid inner twists $(z_+, \varphi_+) \colon \Sp(V^\#) \rightarrow G(V)$, which is an analogue of the set of $(t_f, \varphi_f)$ for various $f$. To define it precisely, we use the $2m$-dimensional symplectic space $(V\otimes\overline{F})^\natural$ over $\overline{F}$ defined by using the Morita equivalence (\S\ref{morita}). It provides us a certain isomorphism $\fm_V$ from $G(V)(\overline{F})$ onto $\Sp((V\otimes\overline{F})^\natural)(\overline{F})$. Then, we define $\RIT^\star(V^\#, V)$ as the set of rigid inner twists of the form $(z_+, \fm_V^{-1}\circ \varphi_A)$ for an isometry $A \colon V^\#\otimes \overline{F} \rightarrow (V\otimes \overline{F})^\natural$ over $\overline{F}$. Here, $\varphi_A$ denotes the isomorphism induced by $A$ (see \S\ref{notations}). We can also define the set
\[
\RIT^\star(W^\#, W)
\]
in a similar way. The second step is to construct a link between $\RIT^\star(V^\#, V)$ and $\RIT^\star(W^\#, W)$. More precisely, by $(z_+, \varphi_+) \leftrightarrow (z_-, \varphi_-)$ we mean that there exists an isometry $\Omega\colon \W^\#\otimes_F\overline{F} \rightarrow \W\otimes_F\overline{F}$ over $\overline{F}$ such that 
\[
\Omega^{-1}\circ w\circ \Omega \circ w^{-1} = \iota^\#(z_+(w), z_-(w))
\]
for all $w \in \cW$ and the following diagram is commutative.
\begin{align}
\label{corr_rig}
\xymatrix{
\Sp(\W^\#) \ar[rr]^-{\varphi_\Omega} & & \Sp(\W) \\
\Sp(V_c^\#) \times {\rm O}(W_c^\#) \ar[u]^-{\iota^\#} \ar[rr]_-{(\varphi_+, \varphi_-)} & & G(V) \times G(W) \ar[u]_-{\iota}
}
\end{align}
Here, $\varphi_\Omega$ denotes the isomorphism induced by $\Omega$ (see \S\ref{notations}). In \S\ref{main sec}, we verify that this framework works well.

Now, we state the main conjecture in this paper. Let $D$ be a division quaternion algebra over $F$, let $V$ be a Hermitian space over $D$, let $W$ be a skew-Hermitian space over $D$, let $\psi\colon F \rightarrow \C^1$ be a non-trivial character, let $(z_+, \varphi_+) \in \RIT^\star(V^\#, V)$, $(z_-, \varphi_-) \in \RIT^\star(W^\#, W)$ with $(z_+, \varphi_+) \leftrightarrow (z_-, \varphi_-)$. We denote by $G_0(W)$ the Zariski connected component of $G(V)$ containing $1$. Assume that $\dim W - \dim V$ is $0$ or $1$.
Then, as in \cite[\S15.1]{GI14}, we have an embedding $\xi\colon {}^LG_0(W) \rightarrow {}^LG(V)$ (resp. $\xi \colon {}^LG(V) \rightarrow {}^LG_0(W)$) of $L$-groups if $\dim V = \dim W$ (resp. $\dim V = \dim W - 1$). Let $\phi, \phi'$ be tempered $L$-parameters of  $G(V), G_0(W)$ such that $\phi = \xi\circ\phi'$ (resp. $\phi = \xi\circ\phi'$) if $n=m$ (resp. $n=m+1$). In this case, it is known that $\theta_\psi(\pi, W)$ is non-zero for $\pi \in \Pi_\phi(G(V))$ (\cite[Proposition 20.4]{Kak22}). Note that we use the slightly adjusted version of Langlands parameters for $G_0(W)$ in this paper (see \S\ref{LLC}). Then, the conjecture is stated as follows.

\begin{conj}\label{main_intro}
Let $s, s'$ be elements of $S_\phi^+, S_{\phi'}^+$ so that they are associate with each other via $\xi$, and let $\pi \in \Pi_\phi(G(V))$. Then, $\theta_\psi(\pi, W)$ has $L$-parameter $\phi'$ and we have
\[
\iota_\phi[\fw_+, z_+, \varphi_+](\pi)(s) = \overline{\iota_{\phi'}[\fw_-, z_-, \varphi_-](\theta_\psi(\pi, W))(s')}.
\]
\end{conj}

We will verify Conjecture \ref{main_intro} in the cases where either $F = \R$ (\S\S\ref{Arch_computation}--\ref{Arch_SpO}) or $F$ is non-Archimedean with $n = m =1$ (\S\ref{nonArch_m=n=1}). In the case $F=\R$, we prove Conjecture \ref{main_intro} by translating the results of Li \cite{Li89} and Li-Paul-Tan-Zhu \cite{LPTZ03} in terms of Langlands parameters.  The real local Langlands correspondence is completed by Mezo by verifying the endoscopic character relations \cite{Mez13}\cite{Mez16}. We use his computation in the proof in order to translate Harish-Chandra parameters into Langlands parameters. In the case where $F$ is non-Archimedean and $n=m=1$, Ikematsu described the local theta correspondence in terms of characters of representations via the accidental isomorphism from quaternionic unitary groups of low ranks with the subgroups of unitary groups \cite{Ike19}. Using this result, we will compute the Langlands parameters of irreducible tempered representations to verify Conjecture \ref{main_intro} in this case. 

The descriptions of local theta correspondences using the sets $\RIT^\star(V^\#, V), \RIT^\star(W^\#, W)$ and the link ``$\leftrightarrow$'' between them are also available to symplectic-orthogonal dual pairs and unitary dual pairs. Hence, we will discuss them in the body of this paper. One can show that they are equivalent to the conjectures in \cite{Pra93} and \cite{Pra00}.  Moreover, we prove the ``weak Prasad conjecture'' for symplectic-orthogonal dual pairs over $\R$ (\S\ref{Arch_SpO}) in the sense of  \cite{AG17}. 

We mention the strong Prasad conjecture here, which uses the Langlands parameter for orthogonal groups instead of that for special orthogonal groups. The formulation of the local Langlands correspondence for disconnected reductive groups (containing orthogonal groups) has appeared in \cite{Kal22}. In the preprint, Kaletha suggested the canonical normalizations of twisted geometric transfer factors, and formulated the endoscopic character relation using twisted spectral transfer factors. Moreover, in the Archimedean case, Mezo's computation \cite{Mez13} also provides the formula of the twisted spectral transfer factor using twisted geometric transfer factors. Hence, in principle, it is possible to formulate the strong Prasad conjecture in the framework of rigid inner twists and prove it in the Archimedean cases. However, we do not discuss it in this paper since it will require careful calculations and is considered to take a lot of time.

Finally, we explain the structure of this paper. 
In \S\S\ref{settings}--\ref{LLC}, we prepare for the later sections. In \S\ref{main sec}, we state the conjecture. The main theorem (Theorem \ref{welldefness}) which controls the dependencies is also stated in this section. In \S\S\ref{Arch_computation}--\ref{Arch_SpO}, we prove the weak Prasad conjecture over the field of real numbers.
In \S\ref{nonArch_m=n=1}, we prove Conjecture \ref{main_intro} when $n=m=1$. This paper also contains five appendices. In \S\ref{cospin}, we prove an elementary result on the centers of spin groups. In \S\ref{app_op}, we discuss a different convention of the local theta correspondence, which is adopted in some previous results. In \S\ref{app theta HC 1}, we discuss a convention problem of the oscillator representation. In \S\S\ref{notes on psi} -- \ref{spec_correction}, we comment on some references. The Archimedean part of this paper is based on the results on the Archimedean local Langlands correspondence and on the Fock model of the oscillator representations. Moreover, the proof of Conjecture \ref{main_intro} in the case $n=m=1$ with $F$ non-Archimedean is obtained by explicit discussions of local theta correspondences for unitary groups of low ranks.  They are attained by a huge amount of calculations, and there are a few small errors. In these appendices, we will point them out.

\subsection*{Acknowledgements}

The author would like to thank A.Ichino and W.T.Gan for suggesting this theme, and thank H.Atobe for useful comments. The contents in \S\S\ref{notes on psi}--\ref{spec_correction} are discovered during discussions with Jialiang Zou and Rui Chen. The author would like to thank them for their help. 
This research is partially supported by JSPS KAKENHI Grant Numbers 20J11509, 23KJ0001.

\section{
		Settings
		}\label{settings}

\subsection{Notations}\label{notations}

First, we list the notations around the algebras.
Throughout this paper, $F$ denotes a field of characteristic $0$, $D$ denotes either a quadratic extension field over $F$ or a quaternion algebra over $F$, and $E$ denotes the center of $D$. The multiplicative groups of $F, D, E$ are denoted by $F^\times, D^\times, E^\times$ respectively. The main involution of $D$ over $F$ is denote by $x \mapsto x^*$ for $x \in D$. Using the main involution, we define the two maps $T_D \colon D \rightarrow F$ and $N_D\colon D\rightarrow F$ by 
\[
T_{D/F}(x) = x + x^*, \quad N_{D/F}(x) = x\cdot x^* 
\]
for $x \in D$. The restrictions of $T_{D/F}$ and $N_{D/F}$ to $E$ are denoted by $T_{E/F}$ and $N_{E/F}$ respectively. We write $D^1 = \{ x \in D \mid N_D(x) = 1\}$ and $E^1 = E^\times \cap D^1$. For an additive character $\psi \colon F \rightarrow \C^1$ and $t \in F^\times$, we denote by $\psi_t$ the additive character of $F$ given by $\psi_t(x) = \psi(tx)$ for $x \in F$.

Then, we prepare the notation of isomorphisms of linear algebraic groups. Let $X, Y$ be right (resp. left) $D$-vector spaces, and let $h \colon X \rightarrow Y$ be a right (resp. left) $D$-linear isomorphism. Then, we denote by $\varphi_h$ the isomorphism from $\GL(X)$ onto $\GL(Y)$ satisfying
\[
\varphi_h(g)h(x) = h(gx) \quad (\mbox{resp. } h(x)\varphi_h(g) = h(xg) )
\]
for $x \in X$ and $g \in \GL(X)$. Restrictions of $\varphi_h$ to subgroups of $\GL(X)$ are also denoted by $\varphi_h$.

Finally, we will list other important notations. If $G$ is a group and $\delta \in G$, then we denote by $S_G(\delta)$ the centralizer of $\delta$ in $G$. If there is no fear of confusion, we denote it by $S(\delta)$. If $k,l$ are positive integers, $a,b$ are positive integers satisfying $a \leq k$ and $b \leq l$, and $x \in D$, then we denote by $e_{a,b}(x)$ the $k \times l$ matrix whose $(a,b)$-component is $x$ and the other components are $0$. For a positive integer $r$, we denote by $J_r$ the anti-diagonal matrix whose anti-diagonal components are $1$, that is, we have
\[
J_r = e_{1,r}(1) + e_{2,r-1}(1) + \cdots + e_{r,1}(1).
\]
If $G$ is a reductive group, $T$ is a maximal torus of $G$, and $B$ is a Borel subgroup containing $T$, we denote by $R(G, T)$ the root system of the roots of $T$ in $G$, and by $\Delta_B$ the positive system of $R(G, T)$ associated with $B$.  

\subsection{Spaces and groups}\label{groups}
Let $\epsilon = \pm1$, let $V$ be a right vector space over $D$ with a non-degenerate $F$-bilinear form $( \ , \ )$ satisfying
\[
(y,x)\cdot a = (y, xa) = \epsilon(xa, y)^* 
\]
for $a \in D$, $x, y \in V$, and let $W$ be a left vector space over $D$ with a non-degenerate $F$ bilinear form $\langle \ , \ \rangle$ satisfying
\[
a\cdot \la y, x \ra = \la ax, y \ra = - \epsilon \la y, ax \ra^*
\]
for $a \in D$, $x, y \in W$.   We call such a form $( \ , \ )$ an $\epsilon$-Hermitian form, and call such a $D$ vector space $V$ equipped with $( \ , \ )$ an right $\epsilon$-Hermitian space. We put $\dim_DV = m$ and $\dim_DW = n$. In this paper, we consider the following cases:
\begin{enumerate}[(I)]
\item $D$ is the matrix algebra $\rM_2(F)$ over $F$,  \label{vsp I}
\item $D$ is a quadratic extension field of $F$, \label{vsp II}
\item $D$ is a division quaternion algebra over $F$, \label{vsp III}
\end{enumerate}
We denote by $G(V)$ (resp. $G(W)$) the unitary group of $V$ (resp. $W$), and by $G_0(V)$ (resp. $G_0(W)$) its Zariski connected component containing $1 \in G(V)$ (resp. $1 \in G(W)$). We denote by $\W$ the tensor product $V\otimes_D W$ of $V$ and $W$, and we consider the symplectic form $\la\la \ , \ \ra\ra$ on $\W$ given by
\[
\la\la x_1\otimes y_1, x_2\otimes y_2 \ra\ra = {\rm T}_{D/F}((x_1, x_2)\la y_1, y_2 \ra^*)
\]
for $x_1, x_2\in V$ and $y_1, y_2 \in W$.

We consider the new action of $D$ on $W$ by
\[
D \times W \rightarrow W, \quad (a, x) \mapsto a^*\cdot x
\]
which defines a structure of right $D$-vector space on $W$. Moreover, the $(-\epsilon)$-Hermitian form $\la \ , \ \ra$ is also $(-\epsilon)$-Hermitian  with respect to the new right action above. When we discuss the new action, we write for $W^{\op}$ instead of $W$, and for $\la \ , \ \ra^{\op}$ instead of $\la \ , \ \ra$ to distinguish the action. For $g \in G(W)$, the map $\fs_W(g) \colon W^{\op} \rightarrow W^{\op}$ given by
\[
\fs_W(g) (x) = x\cdot g^{-1} \quad (x \in W^{\op})
\]
is linear and isometric with respect to $\la \ , \ \ra^{\op}$. Hence, we have the isomorphism $\fs_W\colon G(W) \rightarrow G(W^{\op})$. Besides, we denote by $V^{\rm op}$ the left  $\epsilon$-Hermitian space over $D$ so that $(V^{\rm op})^{\rm op} = V$, and by $\fs_V$ the inverse map of $\fs_{V^{\op}} \colon G(V^{\op}) \rightarrow G(V)$. 

In the cases \eqref{vsp I} and \eqref{vsp III} with $\epsilon = 1$, we define the discriminant of $W$ by 
\[
(-1)^nN_{\End (W)}((x_k, x_l)_{k,l}) \in F^\times/F^{\times 2}
\]
where $x_1, \ldots, x_n$ is a basis of $W$ over $D$, and $N_{\End(W)}$ is the reduced norm of $\End(W)$. The definition does not depend on the choice of the basis $x_1, \ldots, x_n$, and we denote the discriminant by $\fd(W)$. On the other hand, we put $\fd(V) = 1 \in F^\times/F^{\times 2}$. When $\epsilon = -1$, we put $\fd(W) = 1 \in F^\times/F^{\times2}$ and $\fd(V) = \fd(V^{\rm op})$.

In the case \eqref{vsp II}, we fix an element $\daleth \in E^\times$ so that $\varsigma_E(\daleth) = -\daleth$. Then, the discriminant can also be defined (cf. \cite[p.~517]{GI14}), but we do not use it in this paper.

\subsection{Quasi-split inner forms}\label{q-spl grps}
To discuss the quasi-split inner forms of $G(V)$, we consider explicit vector spaces $V_c^\#$ with forms $( \ , \ )^\#$ given by as follows. Let $c \in F^\times$.
\begin{itemize}
\item In the cases \eqref{vsp I} and \eqref{vsp III}, $V_c^\#$ is the $2m$-dimensional $F$-vector space of column vectors, $( \ , \ )^\#$ is given by the matrix
\[
c^{-1}\begin{pmatrix} & J_m \\ -J_m & \end{pmatrix} \quad (\epsilon = 1), \quad c\begin{pmatrix} & & & J_{n-1} \\ & 2 & & \\ & & -2d & \\ J_{n-1} & & & \end{pmatrix} \quad (\epsilon = -1).
\]
Here, $d$ is an element of $F^\times$ so that $\fd(W) = d F^{\times2}$. 
\item In the case \eqref{vsp II}, $V_c^\#$ is the $m$-dimensional $E$-vector space of column vectors,  $( \ , \ )^\#$ is given by the matrices
\[
J_m \quad (\epsilon = 1),  \quad \daleth \cdot J_n \quad (\epsilon = -1).
\]
Note that $V_c^\#$ does not depend on $c$ in this case. However, we use it to unify the notations.
\end{itemize}

We also define $W_c^\#$ by the $2n$-dimensional $E$-vector space of row vectors equipped with the bilinear form $\la \ , \ \ra^\#$ on $W_c^\#$ satisfying
\[
 \la f_k, f_l \ra^\# = \la f^k, f^l \ra^{\op \#}
\]
for all $1 \leq k,l \leq 2n$ (see \S\ref{groups} for the meaning of ``${\op}$'').
Here, $f_1, \ldots, f_{2n}$ denote the canonical basis of $W_c^\#$ and $f^1, \ldots, f^{2n}$ denote the canonical basis of $W_c^{\op \#}$. One can show that 
\begin{align*}
(W_c^\#)^{\op} \rightarrow (W^{\op})_c^\# \colon x \mapsto {}^tx^*
\end{align*}
is isometric. 
In the cases \eqref{vsp I} and \eqref{vsp III} with $\epsilon = 1$, it is useful to put
\[
\varepsilon = \begin{pmatrix} 1_n & & \\ & -1 & \\ & & 1_{n-1} \end{pmatrix} \in G(W_c^\#)(F).
\] 

We denote by  $\W^\#$  the tensor product $V_c^\#\otimes_DW_c^\#$ of $V_c^\#$ and $W_c^\#$, and let $\la\la \ , \ \ra\ra^\#$ be the symplectic form on $\W^\#$ defined by
\[
\la\la x_1\otimes y_1, x_2\otimes y_2 \ra\ra^\# = T_{E/F} ((x_1, x_2)^\#\la y_1, y_2 \ra^{\#*})
\]
for $x_1, x_2\in V_c^\#$ and $y_1, y_2 \in W^\#$. This symplectic space does not depend on $c$.


\subsection{Maximal tori of quasi-split inner forms} 

We set some notations around maximal tori. First, we discuss $G(V_c^\#)$. 

\begin{itemize}
\item In the cases \eqref{vsp I} and \eqref{vsp III} with $\epsilon = 1$, we denote by $T_+^\#$ the maximal torus consisting of the diagonal matrices in $G(V_c^\#)$, and by $\alpha_k^\#$ the algebraic character of $T_+^\#$ projecting the $(k,k)$-component. Then, $\alpha_1^\#, \ldots, \alpha_m^\#$ consists a basis of $X^*(T_+)$.   
\item In the cases  \eqref{vsp I} and \eqref{vsp III} with $\epsilon = -1$, we denote by $A_+^\#$ the maximal split torus of $G_0(V_c^\#)$ consisting of diagonal matrices, and by $T_+^\#$ its centralizer in $G_0(V_c^\#)$. For $k=1, \ldots, m-1$, we denote by $\alpha_k^\#$ the algebraic character of $T_+^\#$ projecting the $(k,k)$-component. Moreover, we define $\alpha_m^\#\colon T_+^\#\rightarrow \GL_1$ by
\[
\alpha_m(\begin{pmatrix} a &  &  &  \\ & x & y & \\ & dy & x & \\ & & & a^{-1} \end{pmatrix}) = x + \sqrt{d} y 
\]
for a diagonal matrix $a$ and $x,y \in \overline{F}$ satisfying $x^2 - dy^2 = 1$. 
\item Consider the case \eqref{vsp II}. We fix an identification ${\rm Res}_{E/F}\GL_1 = \GL_1 \times \GL_1$ over $E$, and we denote by $p_1$ (resp. $p_2$) the projection to the left $\GL_1$ factor. Then, we denote by $T_+^\#$ the maximal torus consisting of the diagonal matrices in $G(V_c^\#)$, by $\alpha_k'$ the algebraic homomorphism from $T_+^\#$ onto ${\rm Res}_{E/F}\GL_1$ projecting the $(k,k)$-component. Moreover, we define the algebraic characters $\alpha_1, \ldots, \alpha_m$ by
\[
\alpha_k = \begin{cases} p_1\circ \alpha_k' & (1 \leq k \leq \lceil m/2 \rceil), \\ p_2\circ \alpha_{m+1-k}' & (1 \leq k \leq \lfloor m/2 \rfloor). \end{cases}
\]
\end{itemize}
Finally, we define the maximal torus $T_-^{\#\op}$ of $G_0((W^{\op})_c^\#)$ and a basis $\beta_1^{\#\op}, \ldots, \beta^{\#\op}$ of $X^*(T_-^{\#\op})$ in the same way as for $G(V_c^\#)$, and put 
\[
T_-^\# = (\fs_{W_c^\#}^{-1}\circ t^{-1})(T_-^{\#\op}), \quad \beta_k^\# = \beta_k^{\#\op}\circ t \circ \fs_{W_c^\#} \ (k = 1, \ldots, m)
\] 
where $t$ denotes the isomorphism from $G((W_c^\#)^{\op})$ onto $G((W^{\op})_c^\#)$ given by $t(g) = {}^tg^{*-1}$ for $G((W_c^\#)^{\op})$.
\subsection{Extensions by extension fields}\label{morita}

In this subsection, we define the $F'$-algebra ${(D\otimes F')}^\natural$, the vector spaces $(V\otimes F')^\natural, (W\otimes F')^\natural$ and forms $( \  , \ )^\natural, \la \ , \  \ra^\natural$ on them for a certain extension field $F'$ of $F$. In the case \eqref{vsp II}, for all extension field $F'$ of $F$, we put ${(E\otimes F')}^\natural = E\otimes F'$, $(V\otimes F')^\natural = V \otimes F'$, $(W\otimes F')^\natural = W \otimes F'$, $( \ , \ )^\natural = ( \ , \ )$, and $\la \ , \ \ra^\natural = \la \ , \ \ra$. 
In the cases \eqref{vsp I} and \eqref{vsp III}, we define them by using the Morita equivalence \cite[p.~362]{Sch85} as follows. Let $F'$ be an extension field of $F$ which splits $D$. Then, we put $(D\otimes F')^\natural = F'$. We fix an identification $D \otimes_FF' \rightarrow {\rm M}_2(F')$. Put
\[
e_{11} = \begin{pmatrix} 1 & 0 \\ 0 & 0 \end{pmatrix}, e_{12} = \begin{pmatrix} 0 & 1 \\ 0 & 0 \end{pmatrix}, 
\]
\[
e_{21} = \begin{pmatrix} 0 & 0 \\ 1 & 0 \end{pmatrix}, \mbox{and } e_{22} = \begin{pmatrix} 0 & 0 \\ 0 & 1 \end{pmatrix}.
\]
We define $V^\natural = V\otimes F'e_{11}$ and the bilinear form $( \ , \ )^\natural$ on $V^\natural$ by
\[
(x,y)^\natural = {\rm Tr}(e_{12} \cdot (x,y))
\]
for $x, y \in V^\natural$. We also define $W^\natural = e_{11}W\otimes F'$ and the bilinear form $\la \ , \ \ra^\natural$ on $W^\natural$ by
\[
\la x,y \ra^\natural = - {\rm Tr}(\la x,y \ra \cdot e_{21})
\]
for $x,y \in W^\natural$. If $\epsilon = 1$ then $( \ , \ )^\natural$ is symplectic and $\la \ , \ \ra^\natural$ is symmetric, and if $\epsilon = -1$ then $( \ , \ )^\natural$ is symmetric and $\la \ , \ \ra^\natural$ is symplectic. 

\begin{rem}
By a technical reason, we adopted the definitions of $\natural$ those do not commute with ``$\op$'', that is, 
$(W\otimes F')^\natural \not= (W\otimes F')^{\op\natural\op}$ as subsets of $W \otimes F'$. However,  one can show that $(W \otimes F')^\natural$ is isometric to $(W\otimes F')^{\op\natural\op}$. 
\end{rem}

The functor $\natural$ gives a categorical equivalence between the category of the $\epsilon$-Hermitian spaces over $D\otimes F'$ and that of the $(-\epsilon)$-Hermitian space over $(D\otimes F')^\natural$ (c.f. \cite[Chapeter 10, \S3]{Sch85}). Namely, we have:

\begin{fact}
An element $g \in G(V)(F')$ preserves the subspace $V^\natural$ of $V \otimes F'$. Moreover, this restriction induces the isomorphism $\mathfrak{m}_V\colon G(V) \rightarrow G(V^\natural)$ over $F'$. Similarly, we have the isomorphism $\mathfrak{m}_W\colon G(W) \rightarrow G(W^\natural)$ over $F'$.
\end{fact}

Put $\W^\natural = (V\otimes F')^\natural \otimes_{(D\otimes F')^\natural} (W\otimes F')^\natural$, and define the symplectic form $\la\la \ , \ \ra\ra^\natural$ on $\W^\natural$ by
\[
\la\la x_1\otimes y_1 , x_2\otimes y_2  \ra\ra^\natural = (x_1, x_2)^\natural \la y_1, y_2 \ra^\natural
\]
for $x_1, x_2 \in (V\otimes F')^\natural$ and $y_1, y_2 \in (W\otimes F')^\natural$. 

\begin{lem}\label{Morita_W}
The natural linear map
\[
\W^\natural \rightarrow \W\otimes_FF'
\]
is bijective and isometric. Moreover, the following diagram is commutative.
\[
\xymatrix{
\Sp(\W)  \ar[rr] & & \Sp(\W^\natural) \\ 
G(V) \times G(W) \ar[u]^-{\iota_{V,W}} \ar[rr]_-{(\mathfrak{m}_V, \mathfrak{m}_W)} 
& & G(V^\natural)\times G(W^\natural) \ar[u]_-{\iota_{V^\natural,W^\natural}}
}
\]
\end{lem}
 
\begin{proof}
In the case \eqref{vsp II}, the claim is obvious. In the rest of the proof, we consider the cases \eqref{vsp I} and \eqref{vsp III}.  Since $\dim_{F'}\W^\natural = \dim_F\W$, it suffices to show that it commutes with the symplectic forms. But we have
\begin{align*}
\la\la x_1\otimes y_1, x_2\otimes y_2 \ra\ra &= {\rm Tr} ((x_1, x_2)\cdot\la y_1, y_2\ra^*) \\
&= {\rm Tr}((x_1, x_2)^\natural e_{21} \cdot e_{12} \la y_1, y_2 \ra^\natural) \\
&= {\rm Tr}((x_1, x_2)^\natural \la y_1, y_2 \ra^\natural e_{22}) \\
&= \la\la x_1\otimes y_1, x_2\otimes y_2 \ra\ra^\natural
\end{align*}
for $x_1, x_2 \in (V\otimes F')^\natural$ and $y_1, y_2 \in (W\otimes F')^\natural$. Hence we have the first assertion.
The second assertion is obvious by the construction.
\end{proof}


\subsection{Whittaker data}\label{WD}

In this subsection, we explain the choice of Whittaker data (c.f. \cite[\S5.3]{KS99}). Fix a non-trivial additive character $\psi \colon F \rightarrow \C^1$. 

First, we consider the case \eqref{vsp II}. In this case, we choose the Whittaker data being compatible with that of \cite{GI16}. More precisely, 
\begin{itemize}
\item if $V_c^\#$ has odd dimension, then we denote by $\fw_+$ the unique Whittaker data of $G(V_c^\#)$, 
\item if $\epsilon = -1$ (resp. $\epsilon = 1$) and $V_c^\#$ has even dimension, denoting ${}_\#V_c$ the left-linear $\epsilon$-Hermitian space satisfying $({}_\#V_c)^\varrho = V_c^\#$ (see \S\ref{conv_theta} below), then we define $\fw_+$ to be the Whittaker data of $G(V_c^\#) = G({}_\#V_c)$ associated with $\psi$ (resp. $x \mapsto \psi_{1/2}(\Tr_{E/F}(x \cdot \daleth)$) via the correspondence of \cite[Proposition 12.1]{GGP12}.
\item if $W_c^\#$ has odd dimension, then we denote by $\fw_-$ the unique Whittaker data of $G(W_c^\#)$,
\item if $\epsilon = 1$ (resp. $\epsilon = -1$) and $W_c^\#$ has even dimension, denoting ${}_\#W_c^{\op}$ the left-linear $(-\epsilon)$-Hermitian space satisfying $({}_\#W_c^{\op})^\varrho = W_c^{\# \op}$ (see \S\ref{conv_theta} below), then we denote by $\fw_-$ the Whittaker data of $G(W_c^\#)$ transferred via $\fs_{W_c^\#}$ from that of $G(W_c^{\# \op}) = G({}_\#W_c^{\op})$ associated with $\psi$ (resp. $x \mapsto \psi_{1/2}(\Tr_{E/F}(x \cdot \daleth)$) via the correspondence of \cite[Proposition 12.1]{GGP12}.
\end{itemize}

Then, we consider the cases \eqref{vsp I} with $\epsilon = 1$ and \eqref{vsp III} with $\epsilon = 1$. In this case, we choose the Whittaker data in the essentially same way as in \cite{Ato18}. More precisely, we define
\begin{itemize}
\item the Whittaker data $\fw_+$ of $G_0(V_c^\#)$  as a conjugacy class represented by the pair $(B_+^\#, \lambda_+^{(c)})$ where $B_+^\#$ is the Borel subgroup consisting of the upper triangle matrices in $G(V_c^\#)$, and $\lambda_+^{(c)}$ is a generic character of the group of $F$-valued points $N_+^\#(F)$ of the nilpotent radical $N_+^\#$ of $B_+^\#$ given by
\[
\lambda_+^{(c)}(u) = \psi(\sum_{k=1}^{m-1} (e_{k+1} \cdot u, e_k)^\# + (e_m\cdot u, e_m)^\#)
\]
for $u \in N_+^\#(F)$,
\item the Whittaker data  $\fw_-$ of $G_0(W_c^\#)$ as a conjugacy class represented by the pair $(B_-^\#, \lambda_-^{(c)})$ where $B_-^\#$ is the Borel subgroup consisting of the upper triangle matrices in $G_0(W_c^\#)$, and $\lambda_-^{(c)}$ is the generic character of the group of $F$-valued points $N_-^\#(F)$ of the nilpotent radical $N_-^\#$ of $B_-^\#$ given by 
\[
\lambda_-^{(c)}(u) = \psi(\sum_{k=1}^{n-2}\la f_k \cdot u, f_{k+1} \ra^\# + \la f_n \cdot u, f_n \ra^\#)
\]
for $u \in N_-^\#(F)$.
\end{itemize}

\begin{rem}
We make an additional explanation of the construction of Whittaker data of $G(W_c^\#)$ in the cases \eqref{vsp I} with $\epsilon = 1$ and \eqref{vsp III} with $\epsilon = 1$. Suppose that $\chi_W(c) = 1$. Then $W_c^\#$ is isomorphic to $W_1^\#$.  Take an isometry $I[c] \colon W_c^\# \rightarrow W_1^\#$, which induces the isomorphism $\varphi_{I[c]}^{-1} \colon G_0(W_1^\#) \rightarrow G_0(W_c^\#)$ of the special orthogonal groups. Denote by $L \subset W_c^\#$ the anisotropic line spanned by $f_n + f_{n+1}$. If we denote by $\fw'$ the Whittaker data associated with $I[c](L) \subset W_1^\#$ via the correspondence of \cite[Proposition 12.1]{GGP12}, then the Whittaker data $(\varphi_{I[c]})^{-1}(\fw')$ of $G_0(W_c^\#)$ transferred by $\fw'$ coincides with $\fw_-$. Here, we applied \cite[Proposition 12.1]{GGP12} for $W_1^\#$ by using ``$\op$'' as in the case \eqref{vsp II}.
\end{rem}


\section{
		Rigid inner twists
		}\label{rig inn twi}
		
In this section, we recall the rigid inner twists of Kaletha. Then, we introduce the class $\RIT^\star(- , -)$ of rigid inner twists, and observe a basic property (Proposition \ref{RIT_str}).

\subsection{
		Settings
		}\label{ritset}

Denote by $\Gamma$ the absolute Galois group of $F$, and by $u$ the ``multiplicative pro-algebraic group'' introduced by Kaletha \cite[\S 3.1]{Kal16}. Then he showed that $H^1(\Gamma, u) = 1$ and 
\[
H^2(\Gamma, u) = \begin{cases}\Z/2\Z & \mbox{ if $F$ is Archimedean}, \\ \widehat{\Z} & \mbox{ if $F$ is non-Archimedean}. \end{cases}
\]
We define the group $\cW$ so that the exact sequence
\[
1 \rightarrow u(\overline{F}) \rightarrow \cW \rightarrow \Gamma \rightarrow 1
\]
is associated with $-1 \in H^2(\Gamma, u)$. The readers should be careful that it is different from the Weil group $W_F$. For an connected reductive group $G$ over $F$ and a finite central subgroup $Z$ of $G$, he also defined the sets $Z^1(u \rightarrow \cW, Z \rightarrow G)$ and $H^1(u \rightarrow \cW, Z \rightarrow G)$ in \cite[\S3.2]{Kal16}.

Let $G'$ be another reductive group over $F$, let $\varphi\colon G \rightarrow G'$ be an isomorphism of algebraic groups defined over $\overline{F}$, and let $z \in Z^1(u \rightarrow \cW, Z \rightarrow G)$. Then, the pair $(z, \varphi)$ is said to be a \emph{rigid inner twist} if they satisfy 
\[
\varphi^{-1}\circ w \circ \varphi \circ w^{-1} = \Ad z(w)
\]
for $w \in \cW$. The following fact (\cite[Corollary 3.8]{Kal16}) is fundamental.

\begin{fact}\label{rig can surj}
If $Z$ contains the center of the derived subgroup of $G$, then the natural homomorphism 
\[
Z^1(u \rightarrow \cW, Z \rightarrow G) \rightarrow Z^1(\Gamma, G/Z(G))
\]
is surjective. Here, $Z(G)$ denotes the center of $G$.
\end{fact}

Moreover, in the case $F = \R$, the following lemma is useful.

\begin{lem}\label{str_real}
Assume that $F = \R$. Fix $w_1 \in \cW$ so that the image of $w_1$ in $\Gamma$ is the non-trivial element. If $h \in G(\C)$ satisfies $h^2 \in Z$ and $(h\cdot w_1(h))^N = 1$ for some positive integer $N$, then there exists unique $z \in Z^1(u \rightarrow \cW, Z \rightarrow  G_0(V_c^\#))$ such that $z(w_1) = h$.
\end{lem}

\begin{proof}
This is just a part of \cite[Theorem 5.2]{Kal16}.
\end{proof}
 
Let $Z$ be a central subgroup of $G$, which is not required to be a finite group. Then, following \cite{Kal18b}, we define
\[
Z^1(u \rightarrow \cW, Z \rightarrow G) = \bigcup_{Z'} Z^1(u \rightarrow \cW, Z' \rightarrow G)
\]
where $Z'$ runs over the finite subgroup of $Z$ defined over $F$. 

\subsection{
		Special classes of rigid inner twists
		}\label{special class RIT}

Denote by $\RIT^\star(V^\#, V)$ the set of the rigid inner twists of the form 
\[
(z, \fm_V^{-1}\circ\varphi_P)
\]
where $z$ is a rigid inner form in $Z^1(u \rightarrow \cW, Z_{V_c^\#} \rightarrow G_0(V^\#))$,  and $P$ is an isometry from $V^\#\otimes\overline{F}$ onto $(V\otimes \overline{F})^\natural$. 
Now we discuss about the structure of the set $\RIT^\star(V^\#, V)$. Denote by $Z_{V_c^\#}$ be the center of $G(V_c^\#)$, and by $Z_V$ the center of $G(V)$. Moreover, to simplify the notation, we put
\[
\cZ^1[V_c^\#] = Z^1(u \rightarrow \cW, Z_{V_c^\#} \rightarrow Z_{V_c^\#}).
\]
The product of the three groups $\cZ^1[V_c^\#] \times (G(V)/Z_V)(F) \times G(V_c^\#)(\overline{F})$ acts on $\RIT^\star(V_c^\#, V)$ by
\begin{align}\label{def_ac}
(\lambda, h, g)\cdot (z, \varphi) = (\lambda \cdot z_g, (\Ad h)\circ \varphi \circ (\Ad g))
\end{align}
for $(\lambda, h, g) \in \cZ^1[V_c^\#] \times (G(V)/Z_V)(F) \times G(V_c^\#)(\overline{F})$ and $(z, \varphi) \in \RIT^\star$. Here, $z_g$ denotes the cocycle in $Z^1(u \rightarrow \cW, Z_{V_c^\#} \rightarrow G_0(V_c^\#))$ given by $z_g(w) = g^{-1}z(w)w(g)$ for $w \in \cW$. 

\begin{prop}\label{RIT_str}
\begin{enumerate}
\item $\RIT^\star(V^\#, V) \not= \varnothing$. \label{nonempty}
\item The action of $\cZ^1[V_c^\#] \times (G(V)/Z_V)(F) \times G(V_c^\#)(\overline{F})$ on $\RIT^\star(V_c^\#, V)$ defined in \eqref{def_ac} is transitive. \label{transitive}
\end{enumerate}
\end{prop}

The assertion \eqref{nonempty} will be proved in \S\ref{main sec} (see Remark \ref{rm_exist} below). The rest of this subsection is devoted to proving \eqref{transitive}. 
First, we study the set $Z^1(u \rightarrow \cW, Z_{V_c^\#} \rightarrow G_0(V_c^\#))$. 

\begin{lem}\label{actn_lem1}
The following sequence of homomorphisms is exact.
\[
\cZ^1[V_c^\#] \rightarrow H^1(u \rightarrow \cW, Z_{V_c^\#} \rightarrow G_0(V_c^\#)) \rightarrow H^1(\Gamma, G_0(V_c^\#)/Z_{V_c^\#}) \rightarrow 1.
\]
\end{lem}

\begin{proof}
In the cases \eqref{vsp I} and \eqref{vsp III}, the claim is obvious. We consider the case \eqref{vsp II}.
It suffices to show the second map is surjective. In this case, $G_0(V_c^\#)$ possesses an anisotropic maximal torus isomorphic to $(E^1)^m$. We denote it by $S$. Then, it is known that $H^1(\Gamma, S/Z_{V_c^\#})) \rightarrow H^1(\Gamma, G_0(V_c^\#)/Z_{V_c^\#})$ is surjective (c.f. \cite[Lemma 10.2]{Kot86} and \cite[Theorem 6.18]{PR94}). 
Take a finite central subgroup $Z$ of $G_0(V_c^\#)$. Since the natural morphism $[Z \rightarrow S] \rightarrow [1 \rightarrow S/Z_{V_c^\#}]$ splits in the category $\mathcal{A}$ of \cite[\S3.2]{Kal16}, we have the natural homomorphism 
\begin{align}\label{tori_surj} 
H^1(u \rightarrow \cW, Z \rightarrow S) \rightarrow H^1(\Gamma, S/Z_{V_c^\#})
\end{align}
is surjective. Hence, we have the map 
\[
H^1(u \rightarrow \cW, Z \rightarrow G_0(V_c^\#)) \rightarrow H^1(\Gamma, G_0(V_c^\#)/Z_{V_c^\#})
\]
is also surjective. Hence, we have the claim. 
\end{proof}

For $z \in Z^1(u \rightarrow \cW, Z_{V_c^\#} \rightarrow G_0(V_c^\#))$ and $g \in G(V_c^\#)(\overline{F})$, we denote by $z_g$ the cocycle in $Z^1(u\rightarrow \cW, Z_{V_c^\#} \rightarrow G_0(V_c^\#))$ given by $z_g(w) = g^{-1}z(w)w(g)$ for $w \in \cW$.

\begin{lem}\label{actn_lem2}
Let $z, z' \in Z^1(u \rightarrow \cW, Z \rightarrow G_0(V_c^\#))$. If the two groups $G_0(V_c^\#)_z$ and $G_0(V_c^\#)_{z'}$ are isomorphic, then there exists $g \in G(V_c^\#)(\overline{F})$ and $\lambda \in Z^1(u \rightarrow \cW, Z \rightarrow Z_{V_c^\#})$ such that $z' = \lambda \cdot z_g$. Here, $G_0(V_c^\#)_z$ (resp. $G_0(V_c^\#)_{z'}$) denotes an inner form of $G_0(V_c^\#)$ associated with $z$ (resp. $z'$).
\end{lem}

\begin{proof}
By Lemma \ref{actn_lem1}, it suffices to show that the number of the $\la \varepsilon \ra$-orbits of $H^1(\Gamma, G_0(V_c^\#)/Z_{V_c^\#})$ coincides with the number of the isomorphism classes of the inner forms of $G_0(V_c^\#)$. 
Assume that $F$ is non-Archimedean. Then, we have the bijection
\[
H^1(\Gamma, G_0(V_c^\#)/Z_{V_c^\#}) \rightarrow \Hom(Z((G_0(V_c^\#)/Z_{V_c^\#})^\wedge)^\Gamma, \C^\times)
\]
(\cite[Theorem 4.1]{Kal16} and \cite[Proposition 5.3]{Kal16}). By construction, this isomorphism is ${\rm Out}_F(G_0(V_c^\#))$-equivariant. The number of the $\la \varepsilon \ra$-orbits of $\Hom(Z((G_0(V_c^\#)/Z_{V_c^\#})^\wedge)^+, \C^\times)$ is $3$ (in the cases \eqref{vsp I} and \eqref{vsp III} with $\epsilon = -1$) or $2$ (otherwise). On the other hand, the number of the isomorphism classes of the inner forms of $G_0(V_c^\#)$ is also $3$ (in the cases \eqref{vsp I} and \eqref{vsp III} with $\epsilon = -1$) or $2$ (otherwise). Hence, for two cocycles $z,  z' \in Z^1(u \rightarrow \cW, Z \rightarrow G_0(V_c^\#))$ satisfying $G_0(V_c^\#)_{z} \cong G_0(V_c^\#)_{z'}$, there exists $g \in G_0(V_c^\#)(\overline{F})$ such that $z' = z_g$. 

Then, we assume that $F = \R$.
Put
\[
G = \begin{cases} {\rm O}(1, 2m-1) & \mbox{ if $G(V_c^\#)$ is an inner form of ${\rm O}(1, 2m-1)$}, \\
 \mbox{ anisotropic inner form of $G(V_c^\#)$} & \mbox{ otherwise}. \end{cases}
\]
Then, one can show that $\#H^1(\Gamma, \overline{G}_0(V_c^\#))/\la \varepsilon \ra = \# H^1(\Gamma, G^\circ/Z)/\la \varepsilon \ra$ where $G^\circ$ denotes the Zariski connected component, $Z$ denotes the central subgroup of order $2$.
We compute it case by case using results in \cite[\S6]{PR94}.
\begin{itemize}
\item First, we assume that $G = {\rm O}(1, 2m-1)$. If $m=1$, then the claim is obvious. Thus, we may assume $m>1$. Denote by $G'$ the anisotropic subgroup of $G$ which is isomorphic to ${\rm O}(0, 2m-2)$, and by $S'$ a maximal torus of $G'$, and by $S$ the neutral connected component of the centralizer of $S'$ in $G$. Then, we have $S \cong S' \times \mathbb{G}_m$. Consider the exact sequence
\[
1 \rightarrow S' \rightarrow S/Z(G) \rightarrow \mathbb{G}_m \rightarrow 1
\]
where the second homomorphism is given by the square of the projection. Taking the long exact sequence, we obtain the isomorphism
\[
H^1(\Gamma, S')/\{\pm 1\} \cong H^1(\Gamma, S/Z(G)).
\]
Since the left-hand side is isomorphic to $\{\pm1\}^{m-1}/\Delta\{\pm1\}$ (c.f. \cite[Theorem 6.17]{PR94}), its quotient by the Weyl group $W(S', G')$ has order $\lfloor (m-1)/2 \rfloor + 1$. Hence, by \cite[Theorem 6.18]{PR94} we have
\begin{align*}
\lfloor (m-1)/2 \rfloor + 1 &= \# H^1(\Gamma, S/Z(G)) / W(S', G') \\
&\geq \# H^1(\Gamma, S/Z(G))/ W(S, G) \\
&= \# H^1(\Gamma, \overline{G}_0(V_c^\#))/ \la \varepsilon \ra \\
&\geq \# \{\mbox{ isomorphism classes of inner forms of $G_0(V_c^\#)$ } \}.
\end{align*}
However, it is known that the last term is also $\lfloor (m-1)/2 \rfloor + 1$, which implies that all inequalities above are indeed equalities. 
\item Then, we assume that $V$ is of the type \eqref{vsp II}. Since the homomorphism \eqref{tori_surj} is surjective, we have $H^1(\Gamma, S/Z(G))$ is isomorphic to $\{ \pm 1\}^m/\Delta\{\pm 1\}$ where $\Delta$ denotes the diagonal embedding. Using this expression, one can obtain
\begin{align*}
\# H^1(\Gamma, \overline{G}_0(V_c^\#)) &= \# H^1(\Gamma, S/Z(G)) / W(S, G) = \lfloor m/2 \rfloor + 1 \\
& = \#\{ \mbox{ isomorphism classes of inner forms of $G_0(V_c^\#)$ } \}.
\end{align*}
\item Finally, we assume that $V$ is of the type \eqref{vsp I} and \eqref{vsp III}, and assume that $G_0(V^\#)$ possesses a anisotropic inner form $G$. Denote by $S$ a maximal torus of $G$. Then, we have
\[
H^1(\Gamma, S/Z) \cong 
 \{ (\zeta_1, \ldots, \zeta_m) \mid \zeta_1^4 =\cdots = \zeta_m^4 = 1, \zeta_1^2 = \cdots = \zeta_m^2 \}/\{\pm 1\}.
\]
Using this expression, one can obtain
\begin{align*}
\# H^1(\Gamma, \overline{G}_0(V_c^\#))/\la \varepsilon \ra &= \# H^1(\Gamma, S/Z(G))/W(S,G) = \lfloor m/2 \rfloor + 2 \\
&= \#\{ \mbox{ isomorphism classes of inner forms of $G_0(V_c^\#)$ } \}. 
\end{align*}
\end{itemize}
These computations complete the proof of Lemma \ref{actn_lem2}.
\end{proof}

Now we complete the proof of Proposition \ref{RIT_str}. Let $(z_1, \varphi_1), (z_2, \varphi_2) \in \RIT^\star(V^\#, V)$. Then, by Lemma \ref{actn_lem2}, there exists $\lambda \in H^1(u \rightarrow \cW, Z_{V_c^\#} \rightarrow Z_{V_c^\#})$ and $g \in G(V_c^\#)(\overline{F})$ so that $z_2 = \lambda\cdot {z_1}_g$. Put $(\lambda, 1, g)\cdot (z_1, \varphi_1) = (z_2, \varphi_1')$. Take isometries $P_1, P_2 \colon V^\#\otimes\overline{F} \rightarrow V\otimes\overline{F}$ so that $\varphi_1' = \fm_V^{-1}\circ \varphi_{P_1}$, $\varphi_2 = \fm_V^{-1}\circ \varphi_{P_2}$. Then, putting $h = \varphi_1(P_1^{-1}\circ P_2)$, we have $\varphi_2 = (\Ad h) \circ \varphi_1'$. Moreover, we have
\begin{align*}
\Ad w(h) &= w\circ \varphi_2 \circ {\varphi_1'}^{-1} \circ w^{-1} \\
&= \varphi_2\circ (\Ad z_2(w)) \circ (\Ad z_2(w)^{-1}) \circ {\varphi_1'}^{-1} \\
&=\Ad h
\end{align*}
for $w \in \cW$, which implies that $h \in (G(V)/Z_V)(F)$. Hence we have $(\lambda, h, g) \cdot (z_1, \varphi_1) = (z_2, \varphi_2)$. This completes the proof of Proposition \ref{RIT_str}.

\subsection{Rigid inner twists for Levi subgroups}\label{RIT_Levi}

First, consider the cases \eqref{vsp I} and \eqref{vsp II}. Denote by $\RIT^\star(V^\#, V^\natural)$ the set of rigid inner twists of the form $(z, \varphi_P)$ where $z$ is a rigid inner form in $Z^1(u \rightarrow \cW, Z_{V_c^\#} \rightarrow G_0(V^\#))$, and $P$ is an isometry from $V^\#\otimes\overline{F}$ onto $V^\natural\otimes\overline{F}$. Then, we identify $\RIT^\star(V^\#, V^\natural)$ with $\RIT^\star(V^\#, V)$ by the isomorphism $\fm_V$. Consider the decomposition
\[
V^\natural = X_1 \oplus \cdots \oplus X_r \oplus V_0 \oplus Y_r \oplus \cdots \oplus Y_1
\]
over $F$ so that both $X_1\oplus \cdots \oplus X_r$ and $Y_1 \oplus\cdots \oplus Y_r$ are isotropic subspace, $V_0$ is a non-degenerate subspace, and $X_k \oplus Y_k$ are non-degenerate and orthogonal to $V_0$ for all $k$ with respect to the bilinear form $( - , - )^\natural$. We define $\RIT_M^\star(V^\#, V^\natural)$ by the set of the rigid inner twists $(z, \varphi_P) \in \RIT^\star(V^\#, V^\natural)$ such that 
\begin{itemize}
\item the subspaces $P^{-1}(V_0)$, $P^{-1}(X_1), \ldots, P^{-1}(X_r), P^{-1}(Y_1), \ldots, P^{-1}(Y_r)$ are defined over $F$, 
\item $z(w)$ preserves the subspaces $P^{-1}(V_0)$, $P^{-1}(X_1), \ldots, P^{-1}(X_r), P^{-1}(Y_1), \ldots, P^{-1}(Y_r)$ for all $w \in \cW$. 
\end{itemize}
We also denote it by $\RIT_M^\star(V^\#, V)$.

Then, consider the case \eqref{vsp III}. Consider the decomposition
\[
V = X_1 \oplus \cdots \oplus X_r \oplus V_0 \oplus Y_r \oplus \cdots \oplus Y_1
\]
over $D$ so that both $X_1\oplus \cdots \oplus X_r$ and $Y_1 \oplus\cdots \oplus Y_r$ are isotropic subspace, $V_0$ is a non-degenerate subspace, and $X_k \oplus Y_k$ are non-degenerate and orthogonal to $V_0$ for all $k$ with respect to the $\epsilon$-Hermitian form $( - , - )$. We define $\RIT_M^\star(V^\#, V)$ by the set of the rigid inner twists $(z, \fm_V^{-1}\circ \varphi_P) \in \RIT^\star(V^\#, V)$ such that
\begin{itemize}
\item the subspaces $P^{-1}((V_0\otimes\overline{F})^\natural), P^{-1}((X_1\otimes\overline{F})^\natural), \ldots, P^{-1}((X_r\otimes\overline{F})^\natural), P^{-1}((Y_1\otimes\overline{F})^\natural), \ldots, P^{-1}((Y_1\otimes\overline{F})^\natural)$ are defined over $F$,
\item  $z(w)$ preserves the subspaces $P^{-1}((V_0\otimes\overline{F})^\natural), P^{-1}((X_1\otimes\overline{F})^\natural), \ldots, P^{-1}((X_r\otimes\overline{F})^\natural), P^{-1}((Y_1\otimes\overline{F})^\natural), \ldots, P^{-1}((Y_r\otimes\overline{F})^\natural)$ for all $w \in \cW$. 
\end{itemize}

	
\section{
		Local theta correspondences
		}\label{LTC}

In this section, we clarify the setting in the definition of the local theta correspondence.


Fix a non-trivial additive character $\psi \colon F \rightarrow \C^1$, and an isotropic subspaces $\X, \Y$  so that $\W = \X + \Y$. 
Then, we denote by $r_{\psi, \Y}$ the Siegel-Shale-Weil projective representation of $\Sp(\W)$ given by
\[
[r_{\psi, \Y}(g)\phi](x) = \int_{\ker c \backslash \Y} \phi(xa + yc)\psi(\la\la xa, xb \ra\ra + 2 \la \la yc, xb \ra\ra + \la\la yc, yd\ra\ra) \: dy
\]
for 
\[
 g = \begin{pmatrix} a & b \\ c & d \end{pmatrix} \in \Sp(\W),
 \]
$F \in \cS(\X)$, and $x \in \X$. Moreover, for $g_1, g_2 \in \Sp(\W)(F)$, we put
\[
c_{\psi, \Y}(g_1, g_2) = \gamma_F(\psi \circ L(\Y, \Y g_2^{-1}, \Y g_1))
\]
where $\gamma_F( \ )$ is the Weil index and $L( \  , \  , \  )$ is the Leray invariant. Then, by \cite{Rao93}, we have 
\[
r_{\psi, \Y}(g_1)r_{\psi, \Y}(g_2) = c_{\psi, \Y}(g_1, g_2)\cdot r_{\psi, \Y}(g_1g_2)
 \]
for $g_1, g_2 \in \Sp(\W)$. To specify that the symplectic space $\W$ is considered, we also write $r_{\psi, \Y}^{(\W)}$ (resp. $c_{\psi, \Y}^{(\W)}$) for $r_{\psi, \Y}$ (resp. $c_{\psi, \Y}$). The metaplectic group $\Mp(\W, c_{\psi, \Y})$ is the group $\Sp(\W)(F) \times \C^1$ together with the binary operation
\[
(g_1, z_1)\cdot (g_2, z_2) = (g_1g_2, z_1z_2c_{\psi, \Y}(g_1, g_2))
\]
for $g_1, g_2 \in \Sp(\W)(F)$ and $z_1, z_2 \in \C^1$, and the Weil representation $\omega[\W, c_{\psi, \Y}]$ of $\Mp(\W, c_{\psi, \Y})$ on $\cS(\X)$ is defined by
\[
(\omega[\W, c_{\psi, \Y}](g,z) \cF)(x) = z \cdot [r_{\psi,\Y}(g)\cF](x)
\]
for $(g,z) \in \Mp(\W, c_{\psi, \Y})$, $\cF \in \cS(\X)$, and $x \in \X$. If there is no fear of confusion, then we denote by $\omega_\psi$ instead of $\omega[\W, c_{\psi, \Y}]$. We take characters $\chi_V$ and $\chi_W$ of $E^\times$ as follows.
\begin{itemize}
\item In the cases \eqref{vsp I} and \eqref{vsp III} with $\epsilon = 1$, $\chi_V$ is the trivial character on $F^\times$ and $\chi_W$ is the character on $F^\times$ given by $\chi_W(a) = (a, \fd(W))_F$ for $a \in F^\times$. 
\item  In the cases \eqref{vsp I} and \eqref{vsp III} with $\epsilon = -1$, we put $\chi_V = \chi_{V^{\op}}$ and $\chi_W = \chi_{W^{\op}}$. 
\item In the case \eqref{vsp II}, we fix a character $\chi_V$ and $\chi_W$ on $E^\times$ so that $\chi_V|_{F^\times} = \omega_{E/F}^{\dim V}$ and $\chi_W|_{F^\times} = \omega_{E/F}^{\dim W}$.
\end{itemize}
Then, following Kudla \cite{Kud94}, we define the embedding
\[
\widetilde{\iota}_{V,W}\colon G(V)\times G(W) \rightarrow \Mp(\W, c_{\psi, \Y})
\]
which is a lift of $\iota_{V,W}\colon G(V)\times G(W)\rightarrow \Sp(\W)$. Note that the two different characters $\psi$ and $\eta$ are discussed in \cite{Kud94}. If $W$ is split in the sense of \cite{Kud94}, taking a basis $\fb = (w_1, \ldots, w_n)$ of $W$ satisfying 
\begin{align}\label{split_herm_mat}
(\la w_k, w_l \ra)_{k,l} = \begin{pmatrix} & I_{n/2} \\ -\epsilon I_{n/2} & \end{pmatrix},
\end{align}
then we denote the function $\beta_V$ of \cite[Theorem 3.1]{Kud94} by $\beta_V[W, \fb, \eta]$ to emphasize that the basis $\fb$ is used in order to apply the setting of \cite{Kud94} and that its definition is given by the formula in $\eta$. For example, in the case \eqref{vsp II} with $\epsilon = 1$, we have 
\[
\beta_{V}[W, \fb, \eta](g) = \chi_V(x(g))\cdot \gamma_F(\eta\circ RV)^{-j}
\]
for $g \in G(W)(F)$. Here, we used the notations $x( \ )$, and $RV$ of \cite{Kud94}. 

First, we assume that $W$ possesses a basis $\fb$ so that the $(-\epsilon)$-Hermitian form $\la \ , \ \ra$ satisfy the equation \eqref{split_herm_mat}. In the case \eqref{vsp I}, we denote by $\fb^\natural$ the basis $(w_1^\natural, \ldots, w_{2n}^\natural)$ of $W^\natural$ given by $w_{2k-1}^\natural = w_ke_{11}$, $w_{2k}^\natural = w_ke_{21}$ for $k=1, \ldots, n$. Then, we define $\widetilde{\iota}_{V, \chi_V}^W\colon G(W)(F) \rightarrow \Mp(\W, c_{\psi, \Y})$ by
\[
\widetilde{\iota}_{V, \chi_V}^{W}(g) = 
\begin{cases}
(\iota_{V,W}(1,g), \beta_{V^\natural}[W^\natural, \fb^\natural, \psi](g)) & (\mbox{ in the case \eqref{vsp I}}), \\
(\iota_{V,W}(1,g), \beta_V[W, \fb, \psi](g)) & (\mbox{ in the cases \eqref{vsp II}, \eqref{vsp III}})
\end{cases}
\]
for $g \in G(W)(F)$.

Second, we define the embedding $\widetilde{\iota}_{V, \chi_V}^W$ for arbitrary $W$. Let $W^\Box$ be the $(-\epsilon)$-Hermitian space $W \times W$ equipped with the $(-\epsilon)$-Hermitian form given by
\[
\la (x_1, x_2), (y_1, y_2) \ra^\Box = \la x_1, y_1 \ra - \la x_2, y_2 \ra
\]
for $x_1, x_2, y_1, y_2 \in W$. Then, the space $W^\Box$ possesses a basis $\fb = (w_1, \ldots, w_{2n})$ satisfying \eqref{split_herm_mat}. Denote by $\mathcal{X}$ (resp. $\mathcal{Y}$) the subspace of $W^\Box$ spanned by $w_1, \ldots, w_n$ (resp. $w_{n+1}, \ldots, w_{2n}$). Put $\W^\Box = V \otimes W^\Box$. Recall that $\X$ and $\Y$ are isotropic subspaces of $\W$. Thus, we have the isotropic subspace $\X^\Box$ (resp. $\Y^\Box$) consisting of the elements $(x,x')$ of $W^\Box$ for $x, x' \in \X$ (resp. $x,x' \in \Y$). Choose an element $\alpha \in \Sp(\W)(F)$ so that $\X^\Box = (V\otimes \mathcal{X})\alpha$ and $\Y^\Box = (V\otimes \mathcal{Y})\alpha$. We denote by $i_1$ the embedding of $G(W)$ into $G(W^\Box)$ so that $(x,y) \cdot i_1(g) = (xg, y)$ for $g \in G(W)$ and $x, y \in W$, by $j_1$ the embedding of $\Sp(\W)$ into $\Sp(\W^\Box)$ so that $(x,y) j_1(g) = (xg, y)$ for $g \in \Sp(\W)$ and $x,y \in \W$. Then, we define $\widetilde{\iota}_{V, \chi_V}^W$ so that the following diagram is commutative.
\[
\xymatrix{
\Mp(\W, c_{\psi, V\otimes \mathcal{Y}}) && \Mp(\W, c_{\psi, \Y^\Box}) \ar[ll]_-{(\Ad \alpha, \Id)}\ar[rr]^-{\Ad (\alpha, 1)}&& \Mp(\W^\Box, c_{\psi, \Y^\Box}) \\
G(W^\Box) \ar[u]^-{\widetilde{\iota}_{V, \chi_V}^{W^\Box}} && G(W)\ar[ll]^-{\widetilde{i_1}} \ar[rr]_-{\widetilde{\iota}_{V, \chi_V}^W} && \Mp(\W, c_{\psi, \Y}) \ar[u]_{\widetilde{j_1}}
}
\]
Here, the isomorphisms $(\Ad\alpha, \Id)$ are $\Ad (\alpha, 1)$ are given by $(\Ad\alpha, \Id)(g,z) = (\alpha g \alpha^{-1}, z)$ and $\Ad (\alpha, 1) = (\alpha, 1) (g,z) (\alpha, 1)^{-1}$ for $g \in \Sp(\W, c_{\psi, \Y^\Box}), z \in \C^1$, the embedding $\widetilde{i_1}$ is given by $\widetilde{i_1}(g,z) = (i_1(g), z)$ for $g \in G(W)(F), z \in \C^1$, and the embedding $\widetilde{j_1}$ is given by $\widetilde{j_1}(g,z) = (j_1(g), z)$ for $g \in \Sp(\W)(F), z \in \C^1$. 

Finally, we define the embedding $\widetilde{\iota}_{W, \chi_W}^V \colon G(V) \rightarrow \Mp(\W, c_{\psi, \Y})$. We use the opposite spaces $V^{\op}$ and $W^{\op}$ (see \S\ref{groups}). Then, the linear map
\begin{align}\label{op_symp}
W^{\op}\otimes V^{\op} \rightarrow V \otimes W \colon x\otimes y \mapsto y \otimes x
\end{align}
is isometric, and the following diagram is commutative.
\[
\xymatrix{G(V)\times G(W) \ar[d] \ar[rr]^-{\iota_{V,W}} & &  \Sp(V\otimes W) \ar[d] \\ G(W^{\op}) \times G(V^{\op}) \ar[rr]_-{\iota_{W^{\op}, V^{\op}}} & & \Sp(W^{\op}\otimes V^{\op})}
\]
Here, the left column map is given by $(h, g) \mapsto (\fs_W(g), \fs_V(h))$ and the right column map is the isomorphism induced by the isometry \eqref{op_symp}.  Then, we also obtain the embedding
\[
\widetilde{\iota}_{W, \chi_W}^V = \widetilde{\iota}_{W^{\op}, \chi_W}^{V^{\op}} \circ \fs_V \colon G(V) \rightarrow \Mp(\W, c_{\psi, \Y}).
\]
We define $\widetilde{\iota}_{V, W, \chi_V, \chi_W}$ by $\widetilde{\iota}_{V, \chi_V}^W$ and $\widetilde{\iota}_{W,\chi_W}$. If there is no fear of confusion, we write $\widetilde{\iota}_{V,W}$ for $\widetilde{\iota}_{V,W, \chi_V, \chi_W}$.

\begin{rem}\label{natural_Weil}
In the case \eqref{vsp I}, via the identification $\Sp(\W^\natural) = \Sp(\W)$ of Lemma \ref{Morita_W}, we have $r_{\psi, \Y}^{(\W)} = r_{\psi, \Y^\natural}^{(\W^\natural)}$ where $\Y^\natural$ denotes the image of $\Y$ in $\W^\natural$. Thus we can identify $\Mp(\W, c_{\psi, \Y})$ with $\Mp(\W^\natural, c_{\psi, \Y^\natural})$, and we have $\omega_{\psi, \Y} = \omega_{\psi, \Y^\natural}$.
\end{rem}


For an irreducible representation $\pi$ of $G(W)(F)$, we define
\[
\Theta_\psi(\pi, V) = ((\omega_{\psi, \Y}\circ\widetilde{\iota}_{V,W})\otimes\pi^\vee)_{G(V)}.
\]
If $\Theta_\psi(\pi, V)=0$, we put $\theta_\psi(\pi, V) = 0$. Otherwise, by the Howe duality (\cite{How89}, \cite{Wal90}, \cite{GT16}, \cite{GS17}), we have that $\Theta_\psi(\pi, V)$ has the unique irreducible quotient if it is non-zero. We denote the irreducible quotient by $\theta_\psi(\pi, V)$. To emphasize $\chi_V$ and $\chi_W$, we also denote it by $\theta_\psi^{\chi_V, \chi_W}(\pi, W)$.


\section{
		Local Langlands correspondence
		}\label{LLC}
		
In this section, we explain the formulation of the Langlands parameters, which we use in the later sections. 

\subsection{
		The L-groups
		}\label{Lgrp}

Put
\[
G_0(V^\#_c)^\wedge = \begin{cases} \SO_M(\C) & \mbox{ in the cases \eqref{vsp I}, \eqref{vsp III}}, \\ \GL_m(\C) & \mbox{in the case \eqref{vsp II}}. \end{cases}
\]
where $M = 2m + (1 + \epsilon)/2$ and $\SO_M(\C)$ is the set of $g \in \SL_M(\C)$ satisfying ${}^tg\cdot J_M\cdot g = J_M$. Then, $G_0(V_c^\#)^\wedge$ is the Langlands dual group of $G_0(V^\#)$. We denote by $\cT_+$ the maximal torus of $G_0(V_c^\#)^\wedge$ consisting of diagonal matrices, and by $\cB_+$ the Borel subgroup of $G_0(V_c^\#)^\wedge$ consisting of the upper triangle matrices.
Denote by $\widehat{\alpha}_k$ the algebraic character of $\cT$ projecting the $(k,k)$-component. Then, we identify $X^*(T^\#)$ with $X_*(\cT)$ via the isomorphism $\fD\colon X^*(T^\#) \rightarrow X_*(\mathcal{T})$ characterized by 
\[
(\widehat{\alpha}_k\circ \fD(\alpha_l))(z) = z^{\delta_{k,l}}  \quad (z \in \C^\times, \ 1 \leq k,l \leq m)
\]
where $\delta_{k,l}$ is the Kronecker's delta.

In the cases \eqref{vsp I} and \eqref{vsp III} with $\epsilon = -1$, we choose an automorphism $\widehat{\varepsilon}$ of $G_0(V_c^\#)^\wedge$ such that $G_0(V)^\wedge \rtimes \langle \widehat{\varepsilon} \rangle$ is isomorphic to an orthogonal group, $\widehat{\varepsilon}(\cT) = \cT$, $\widehat{\varepsilon}(\cB) = \cB$, and $\widehat{\varepsilon}(\widehat{\Delta}_\cB^\circ) = \widehat{\Delta}_\cB^\circ$. To unify the notation, we put $\widehat{\varepsilon} = {\rm Id}_{G_0(V)^\wedge}$ in the other cases. 

The Weil group $W_F$ act on $G_0(V_c^\#)^\wedge$ by
\[
w \cdot g = \begin{cases} g & (\chi_V(w) = 1), \\ \widehat{\varepsilon} g \widehat{\varepsilon}^{-1} & (\chi_V(w) = -1). \end{cases}
\]
for $w \in W_F, g \in G_0(V_c^\#)^\wedge$ in the case \eqref{vsp I} and \eqref{vsp III}, and by
\[
w \cdot g = \begin{cases} g & (w \in W_E), \\  \Phi_m {}^tg^{-1} \Phi_m^{-1} & (w \not\in W_E). \end{cases}
\]
for $w \in W_F, g \in G_0(V_c^\#)^\wedge$ in the case \eqref{vsp II}. Here, 
\[
\Phi_m = \sum_{k=1}^m e_{k, m+1-k}((-1)^{k-1}) \in \GL_m(\C). 
\]
Then, we define the L-group of $G(V_c^\#)$ to be
\[
{}^LG_0(V_c^\#) = G_0(V_c^\#)^\wedge\rtimes W_F.
\]

Finally, we define the Langlands dual group and L-group of $G_0(W_c^\#)$ via the isomorphism 
\[
t\circ\fs_{W_c^\#}\colon G_0(W_c^\#) \rightarrow G_0((W^{\op})_c^\#)
\]
where $t$ is an isomorphism from $G((W_c^\#)^{\op})$ onto $G((W^{\op})_c^\#)$ given by $t(g) = {}^tg^{* -1}$ for $g \in G((W_c^\#)^{\op})$. We also choose an automorphism $\widehat{\varepsilon}$ of $G_0(W_c^\#)^\wedge$ in the same way as that for $G_0(V_c^\#)^\wedge$.

\subsection{
		The L-parameters
		}\label{Lpar}

We define the local Langlands group by
\[
L_F = \begin{cases} W_F\times \SL_2(\C) & \mbox{ when $F$ is non-Archimedean}, \\ W_F & \mbox{ when $F$ is Archimedean}.\end{cases}
\]
In this paper, by an L-parameter of $G_0(V)$ we mean a homomorphism $\phi \colon L_F \rightarrow {}^L\!G_0(V)$ satisfying
\begin{itemize}
\item $\phi$ is relevant to $G_0(V)$,
\item $\phi|_{W_F}$ is continuous, 
\item  for $w \in W_F$, $\phi(w) = (w, a(w))$ for some semi-simple element of $G_0(V)^\wedge$, and
\item $\phi|_{\SL_2(\C)}$ is algebraical if $F$ is non-Archimedean.
\end{itemize}
Then, we put
\[
C_\phi = {\rm Cent}_{G_0(V)^\wedge}(\Im \phi)
\]
and
\[
S_\phi^+ = p^{-1}(C_\phi)
\]
where $p$ is the covering homomorphism from $\overline{G_0}(V)^\wedge$ onto $G_0(V)^\wedge$. 

We denote by $\Phi(G_0(V))$ the set of the $L$-parameters for $G_0(V)$, by $\Phi_t(G_0(V))$ the set of the tempered $L$-parameters for $G_0(V)$, and by $\Phi_2(G_0(V))$ the set of the discrete series $L$-parameters for $G_0(V)$.

\subsection{
		(Unions of) L-packets
		}\label{Lpkt}

In this section, we define (unions of) L-packets using the Plancherel measures. 
 Let $\sigma$ be an irreducible representation of $G_0(V)(F)$, let $r$ be a positive integer, and let $\tau$ be an irreducible representation of $\GL_r(D)$. We define the Plancherel measure, a rational function on $s \in \C$, as follows. Denote by $H_{2r}$ the $\epsilon$-Hermitian space given by a pair consisting of the space $D^{2r}$ of column vectors and the Hermitian form $( \ , \ )_{2r}$ defined by
\[
( x, y)_{2r} = \sum_{k=1}^r(x_k^*\cdot y_{2r+1 - k} + \epsilon \cdot x_{2r + 1 - k}^*\cdot y_k)
\]
for $x= {}^t(x_1, \ldots, x_{2r}), y = {}^t(y_1, \ldots, y_{2r}) \in D^{2r}$. We denote by $X_r$ (resp. $\overline{X}_r$) the $r$-dimensional isotropic subspace of $H_{2r}$ generated by $\mathbf{e}_1, \ldots, \mathbf{e}_r$ (resp. $\mathbf{e}_{r+1}, \ldots, \mathbf{e}_{2r}$) where
\[
\mathbf{e}_1 = {}^t(1, 0,\ldots, 0), \ldots, \mathbf{e}_{2r} = {}^t(0,\ldots, 0,1).
\]
Let $V' = V \oplus H_{2r}$, and let $P_{X_r}$ (resp. $P_{\overline{X}_r}$) be the maximal parabolic subgroup of $G_0(V')$ stabilizing $X$ (resp. $\overline{X}$). Then, the Levi-subgroup $M_{X_r}$ can be identified with $G_0(V)\times\GL_r(D)$ in the natural way. 
Then, for an irreducible representation $\sigma$ of $G_0(V)(F)$ and an irreducible representation $\tau$ of $\GL_r(D)$, we define the Prancherel measure $\mu(s, \sigma\boxtimes\tau)$ by the same manner as in \cite[\S16.1]{Kak22}. On the other hand, for an $L$-parameter $\phi$ for $G_0(V_c^\#)$ and an irreducible tempered representation $\tau$ of $\GL_r(D)$, we define
 \[
 \mu(s, \phi\boxtimes\phi_\tau) = \frac{\gamma(s, \phi^\vee\boxtimes\phi_\tau, \psi)}{\gamma(1+s, \phi^\vee\boxtimes\phi_\tau, \psi)}\cdot\frac{\gamma(2s, \phi_\tau, R, \psi)}{\gamma(1+2s, \phi_\tau, R, \psi)}
 \]
 where $\phi_\tau$ is the $L$-parameter of $\tau$, and $R$ denotes $\wedge^2$ in the cases \eqref{vsp I} and \eqref{vsp III} or one of the Asai's representations in the case \eqref{vsp II} (c.f. \cite[\S7]{GGP12}).

Let $(z, \varphi) \in \RIT^\star(V^\#, V)$, and let $\phi \in \Phi_2(G_0(V))$. Then, we associate the set $\widetilde{\Pi}_\phi(G_0(V))$ with $\phi$ consisting of the square-integrable irreducible representations of $G_0(V)(F)$ such that 
\[
  \mu(s, \pi\boxtimes\tau) = \mu(s, \phi\boxtimes \phi_\tau)
\]
for all square-integrable irreducible representations $\tau$ of $\GL_k(D)$ for all $k$. By Proposition \ref{RIT_str} \eqref{transitive}, we have the set $\widetilde{\Pi}_\phi(G_0(V))$ does not depend on the choice of $(z, \varphi) \in \RIT^\star(V_c^\#, V)$.

Let $\phi \in \Phi_t(G_0(V))$. Take the minimal Levi-subgroupb $M$ so that $g\phi(L_F)g^{-1} \subset {}^L(M)$ for some $g \in G_0(V)^\wedge$. Then, one can show that $(\Ad g)\circ\phi \in \Phi_2(M)$. Moreover, there exists $h \in G_0(V)(\overline{F})$ such that $h^{-1}z(w)w(h) \in M(\overline{F})$ for all $w \in \cW$. We denote by $(z, \varphi)^{M}$ the pair $(w \mapsto h^{-1}z(w)w(h), \varphi\circ (\Ad h))$. Since $(z, \varphi) \in \RIT^\star(V^\#, V)$, we have $(z, \varphi)^{M} \in \RIT_M^\star(V^\#, V)$. Then, we define $\widetilde{\Pi}_\phi(G_0(V))$ by
\[
\{ \mbox{ irreducible components of $\Ind_P^{G_0(V)}\pi$ } \mid \pi \in \widetilde{\Pi}_{(\Ad g)\circ\phi}(M) \}.
\]

\begin{lem}\label{GS}
the two sets $\widetilde{\Pi}_\phi(G_0(V))$ and $\widetilde{\Pi}_{\phi'}(G_0(V))$ are disjoint if $\phi, \phi' \in \Phi_t(G_0(V))$ are not conjugate under $G_0(V)^\wedge\rtimes \la\widehat{\varepsilon}\ra$.
\end{lem}

\begin{proof}
This lemma is proved by a similar argument to \cite[Lemma12.3]{GS12}. We prove it here only in case \eqref{vsp III} for simplicity. Recall that $\phi$ and $\phi'$ are conjugate under $G_0(V)^\wedge \rtimes\la \widehat{\varepsilon}\ra$ if and only if $\std \circ \phi = \std \circ \phi'$ as representations of $L_F$ (c.f. \cite[Lemma 3.1]{GGP12}). 
According to \cite[Proposition III.4.1 (ii)]{Wal03}, it suffices to show Lemma \ref{GS} for discrete series parameters. Denote by $R_k$ the unique $k+1$-dimensional irreducible representation of $\SL_2(\C)$. Let $\phi, \phi'$ be discrete series parameters of $G(V)$. The representation $\std \circ \phi$ of $W_F \times \SL_2(\C)$ decomposes into 
\[
\std\circ\phi = \bigoplus_{k=0}^\infty X_k \boxtimes R_k
\]
for some $W_F$-modules $X_k \ (k=0,1, \ldots)$. Let $k_0$ be a non-negative integer, and suppose $X_k\boxtimes R_k$ is contained in $\std\circ\phi'$ for all $k < k_0$. Take an irreducible component $\rho$ of $X_{k_0}$. Then, there exists $0$ or an irreducible representation $\tau$ of $W_F$ such that $\tau \not\cong\rho$, $\rho \oplus \tau$ has an even dimension, and ${\rm Hom}_{W_F}(\tau, \std\circ\phi') =  {\rm Hom}_{W_F}(\tau, \std\circ\phi) = 0$.
Then, $\rho \oplus \tau$ defines a discrete series $L$-parameter of a general linear group over $D$, and 
\[
\frac{\mu(s, (\std\circ\phi) \otimes (\rho\oplus \tau)^\vee)}{\prod_{k<k_0}\mu(s, (X_k\boxtimes R_k)\otimes (\rho\oplus\tau)^\vee))} 
\]
has a pole at $s = k_0/2$. Hence, if 
\[
\mu(s, (\std\circ\phi') \otimes (\rho\oplus \tau)^\vee) = \mu(s, (\std\circ\phi) \otimes (\rho\oplus \tau)^\vee),
\]
then we have that $\rho\boxtimes R_{k_0}$ is contained in $\std \circ \phi'$. Hence, by using the induction, we have $\std\circ\phi = \std\circ\phi'$. Thus, we have the claim.
\end{proof}

\begin{df}
We call two irreducible representations $\pi$ and $\pi'$ of $G_0(V)(F)$ are $G(V)(F)$ equivalent if there exists $g \in G(V)(F)$ such that $\pi' = \pi\circ\Ad g$ as representations of $G_0(V)(F)$. We denote by $\widetilde{\Pi}_\phi(G_0(V))_{weak}$ the set of the $G(V)(F)$-equivalent classes of $\widetilde{\Pi}_\phi(G_0(V))$.
\end{df}

\begin{rem}
It is natural to expect that  the set $\widetilde{\Pi}_\phi(G_0(V))$ is the union
\[
\Pi_\phi((z, \varphi)) \cup \Pi_{\Ad \widehat{\varepsilon}\circ\phi}((z, \varphi))
\]
of usual $L$-packets. This can be a larger set than a usual $L$-packet in the cases \eqref{vsp I} and \eqref{vsp III} with $\epsilon = -1$.
\end{rem}

\begin{rem}
In the case \eqref{vsp I} with $\epsilon = -1$, the natural map $\widetilde{\Pi}_\phi(G_0(V)) \rightarrow \widetilde{\Pi}_\phi(G_0(V))_{weak}$ can possess non-trivial fibers. Otherwise, we have $\widetilde{\Pi}_\phi(G_0(V)) = \widetilde{\Pi}_\phi(G_0(V))_{weak}$.
\end{rem}

Finally, for a tempered L-parameter $\phi$ for $G(W_c^\#)$, we define
\[
\widetilde{\Pi}_\phi(G_0(W)) = \{ \pi \circ \fs_W \mid \pi \in  \widetilde{\Pi}_\phi(G_0(W^{\op})) \}.
\]

\subsection{Langlands parameters}\label{Langpa}

Then, we recall how the internal structure of a tempered $L$-packet is described. A refined local endoscopic data introduced by Kaletha \cite{Kal16} is a tuple $(H, \cH, \dot{t}, \eta)$ where
\begin{itemize}
\item $H$ is a quasi-split connected reductive group over $F$,
\item $\cH$ is a split extension of $\widehat{H}$ by $W_F$ so that the homomorphism $W_F \rightarrow {\rm Out}(\widehat{H})$ given by the extension coincides with the composition of $W_F \rightarrow {\rm Out}(H)$ and ${\rm Out}(H) \rightarrow {\rm Out}(\widehat{H})$,
\item $\dot{t}$ is an element of the component group $\pi_0(Z(\widehat{\overline{H}})^+)$ of $Z(\widehat{\overline{H}})^+$,
\item $\eta$ is an injective $L$-homomorphism $\cH \rightarrow {}^L\!G$ so that ${\rm Im}(\eta) = {\rm Cent}_{{}^LG}(\eta(t))$ where $t$ is the image of $\dot{t}$ in $\widehat{H}$.
\end{itemize}
Let $\phi$ be a tempered L-parameter for $G_0(V)$, and let $\dot{s} \in S_\phi^+$. 
We denote by $\cE(\dot{s})$ the set of the refined endoscopic data $(H, \cH, \dot{t}, \eta)$ so that $\overline{\eta}(\dot{t}) = \dot{s}$. Here, $\overline{\eta}\colon \widehat{\overline{H}}\rightarrow\widehat{\overline{G}}$ is the unique lift of $\eta$. Note that all elements of $\cE(\dot{s})$ are isomorphic to each other in the sense of Kaletha \cite[pp.~599]{Kal16}. 
Let $H_1$ be a $z$-extension of $H$ (see \cite[\S2.2]{KS99}). Then, there exists an injection $\cH \rightarrow {}^LH_1$ which extends ${\rm Id}\colon \widehat{H} \rightarrow \widehat{H}$. For $\delta \in G_0(V)(F)$ and $\delta^\# \in G_0(V^\#)(F) \cap (\Ad G_0(V^\#)(\overline{F}))(\varphi^{-1}(\delta))$, we denote by $\inv_z^\varphi(\delta, \delta^\#)$ the cocycle in $Z^1(u \rightarrow \cW, Z \rightarrow S^\#)$ given by
\[
\inv_z^\varphi(\delta, \delta^\#)(w) = g^{-1}z(w) w(g) \quad (w \in \cW)
\] 
where $g$ is an element of $G_0(V^\#)(\overline{F})$ satisfying $\varphi(g\delta^\# g^{-1}) = \delta$. If there exists a norm $\gamma_1 \in H_1(F)$ of $\delta$ (\cite[\S3]{KS99}) and if its image $\gamma$ is semi-simple and strongly $G_0(V)$-regular, then we denote by $u_{\gamma, \delta^\#}$ the embedding $S_H(\gamma) \rightarrow S_{G_0(V^\#)}(\delta^\#)$ so that $u_{\gamma, \delta^\#}(\gamma) = \delta^\#$.
Moreover, we put
\[
\Delta'(\gamma_1, \delta) = \varepsilon(\cV, \psi)(\Delta_{I}^{-1}\Delta_{II}\Delta_{III_1}^{-1}\Delta_{III_2}\Delta_{IV})(\gamma_1, \delta^\#)\la \inv_z^\varphi(\delta, \delta^\#), \dot{s}_{\gamma_1, \delta^\#} \ra
\]
where $\varepsilon(\mathcal{V}, \psi)$ is the normalization factor of \cite[\S5.2]{KS99}, $\Delta_{I}( - , - ), \ldots, \Delta_{IV}( - , - )$ are the factors of \cite{LS87}, $\dot{s}_{\gamma, \delta^\#}$ is the image of $\dot{s}$ in $\widehat{\overline{S}}$ via the composition of $Z(\widehat{\overline{H}}) \rightarrow  \overline{C_H(\gamma)}^\wedge$ and $(u_{\gamma, \delta^\#}^{-1})^\wedge$, and $\la - , - \ra$ is the pairing of \cite[Corollary 5.4]{Kal16}. It is known that $\Delta'(\gamma_1, \delta)$ does not depend on the choice of $\delta^\#$.

\begin{hyp}\label{Flem}
For $f \in C_c^\infty(G_0(V)(F))$, there exists $f^{\phi, \cE(\dot{s})} \in C^\infty(H_1(F))$ such that its support is comact modulo the center of $H_1$, and 
\[
\sum_{\gamma' \sim \gamma_1}\int_{C_{H_1(F)}(\gamma')\backslash H_1(F)} f^{\phi, \cE(\dot{s})}(h^{-1}\gamma' h) \: dh = \sum_{\delta} \Delta'(\gamma_1, \delta) \int_{C_{G_0(V)(F)}(\delta)\backslash G_0(V)(F)}f(g^{-1} \delta g) \: dg
\]
for all semisimple strongly $G_0(V)$-regular elements $\gamma$ in $H(F)$.
Here, $\gamma_1$ denotes a representative of $\gamma$ in $H_1(F)$, the summation of the left hand side is taken over the representatives of the elements of $H(F)$ which are conjugate to $\gamma$ under $H(\overline{F})$, and the summation of the right side hand is taken over the elements $\delta$ in $G_0(V)(F)$ having a norm $\gamma$ in $H(F)$. 
\end{hyp}

For $z \in Z^1(u \rightarrow \cW, Z \rightarrow G_0(V))$, we denote by $\zeta_z$ the character of $Z((\overline{G_0}(V)^\wedge))^+$ associated with $z$ via the pairing of \cite[Corollary 5.4]{Kal16}. For an $L$-parameter of $G_0(V)$, we denote by ${\rm Irr}(S_\phi^+, V)$ the set of the irreducible representations of $S_\phi^+$ whose restrictions to $Z(\overline{G_0}(V)^\wedge)^+$ meet with $\zeta_{z'}$ for some $(z', \varphi') \in \RIT^\star(V_c^\#, V)$.

\begin{hyp}\label{LLCclg}
Let $(z, \varphi) \in \RIT^\star(V^\#, V)$. Then, any tempered irreducible representation of $G_0(V)(F)$ is contained in $\wPi_\phi(G_0(V))$ for some $\phi \in \Phi_t(G_0(V))$.
Moreover, for $\phi \in \Phi_t(G_0(V))$, $\dot{s} \in S_\phi^+$, and $f \in C_c(G_0(V)(F))$, there exists a map $\widetilde{\Pi}_\phi(G_0(V)) \rightarrow {\rm Irr}(S_\phi^+, V)\colon \pi \mapsto \rho_\pi$ such that 
\begin{align}\label{ECRweak}
\sum_{\sigma \in \Pi_\phi(H_1)} \Tr_\sigma(f^{\phi, \cE(\dot{s})} + f^{\Ad\widehat{\varepsilon}\circ\phi, \cE((\Ad\widehat{\varepsilon})\dot{s})})
= e(G_0(V)) \cdot n_\phi \sum_{\pi \in \wPi_\phi(G_0(V))} \Tr_{\rho_\pi}(\dot{s}) \cdot \Tr_\pi(f)
\end{align}
where $n_\phi=2$ if $(\Ad \widehat{\varepsilon})\circ \phi$ is conjugate to $\phi$ under $G_0(V_c^\#)^\wedge$, and $n_\phi = 1$ otherwise. 

\end{hyp}

We denote by $\iota[\fw, z, \varphi]_{\phi} $ the map $\pi \mapsto \rho_\pi$ of Hypothesis \ref{LLCclg}. 
For an irreducible tempered representation $\pi$ of $G_0(V)(F)$, by the \textbf{Langlands parameter} (with respect to $\fw, z, \varphi$) of $\pi$ we mean a pair $(\phi, \iota[\fw, z, \varphi]_\phi(\pi))$ so that $\pi \in \wPi_\phi(G_0(V))$. By the characterization \eqref{ECRweak}, we have
\begin{align*}
\iota[\fw, z, \varphi]_{\Ad \widehat{\varepsilon}\circ\phi}(\pi)((\Ad \widehat{\varepsilon}) s) = \iota[\fw, z, \varphi]_\phi(\pi)(s)
\end{align*}
for $\phi \in \Phi_t(G_0(V))$, $\pi \in \wPi_\phi(G_0(V))$, and $s \in S_\phi^+$.

\subsection{Some Properties}\label{LLC prop}

In this section, we summarize the results on the behaviors of the Langlands parameters under some changes of $\fw, z, \varphi$, which are essentially due to Kaletha.

Denote by $\mathcal{K}_V$ the kernel of the covering map $\overline{G_0}(V)^\wedge \rightarrow G_0(V)^\wedge$.
Take a maximal torus $S$ of $G_0(V)$. We denote by $\overline{S}$ the quotient $S/Z$. Then, the cokernel of the natural homomorphism $X^*(\overline{S}) \rightarrow X^*(S)$ is $\Hom(Z, \mu_N)$ where $N = \# Z$, and the kernel of the covering map $\overline{S}^\wedge \rightarrow S^\wedge$ is $\mathcal{K}_V$. For $s \in S_\phi^+$ and $\lambda \in H^1(\Gamma, Z)$ we put
\begin{align}\label{paircenter}
\la \lambda, s \ra_Z \coloneqq \la \lambda, \fc(d(s^{-1})) \ra_S
\end{align}
where $\la \ , \  \ra$ is the pairing given by the Tate-Nakayama duality for $S$, $d$ is the connecting homomorphism from $S_\phi^+ = H^0(\phi(L_F), \overline{G_0}(V)^\wedge)$ to $H^1(\Gamma, \mathcal{K}_V)$, and $\fc$ is the connecting homomorphism of the following diagram:
\[
\xymatrix{
0 \ar[r] & X^*(\overline{S}) \ar[r]\ar[d] &{\rm Lie}(\widehat{\overline{S}})\ar@{=}[d] \ar[r]^-{\rm exp} & \widehat{\overline{S}} \ar[r]\ar[d] & 1 \\
0 \ar[r] & X^*(S) \ar[r] &{\rm Lie}(\widehat{S}) \ar[r]^-{\rm exp} & \widehat{S} \ar[r] & 1}.
\]
One can show that $\la \ , \  \ra$ does not depend on the choice of $S$. In this paper, we need the following lemma. 
\begin{lem}\label{zurashi1}
Let $(z, \varphi)\colon G_0(V^\#) \rightarrow G_0(V)$ be a rigid inner twist, and let $\lambda \in H^1(\Gamma, Z)$. 
Then, $(z\cdot \lambda, \varphi)$ is also a rigid inner twist, and we have
\[
\iota[\fw, z\cdot\lambda, \varphi]_\phi(\pi) = \iota[\fw, z, \varphi]_\phi(\pi)\otimes\la\lambda, - \ra
\]
for $\pi \in \Pi_\phi(G_0(V))$.
\end{lem}

\begin{proof}
 \cite[Lemma 6.3]{Kal18b}
\end{proof}

\begin{cor}\label{zu1-2}
Let $V$ and $V'$ be Hermitian spaces over $D$, let $(z,\varphi)\colon G_0(V^\#)\rightarrow G_0(V)$ be rigid inner twist, and let $A\colon V\rightarrow V'$ be an isometry over $\overline{F}$ so that $\varphi_A \colon G_0(V)\rightarrow G_0(V')$ is an isomorphism over $F$. Then, we have
\[
\iota[\fw, z, (\Ad g)\circ\varphi]_\phi(\pi) = \iota[\fw, z, \varphi]_\phi(\pi)\otimes\la \lambda, - \ra
\]
for $\pi \in \Pi_\phi(G(V))$. Here, $\lambda \in Z^1(\Gamma, Z)$ is the 1-cocycle satisfying $\varphi(\lambda(\tau)) = \varphi^{-1}(A\tau(A)^{-1})$ for $\tau \in \Gamma$.
\end{cor}

\begin{proof}
 Take $\dot{s} \in S_f^+$. It suffices to show 
\[
\Delta'[\fw, z, (\Ad g)\circ \varphi]_\phi(\gamma, \dot{\delta}) =\la\lambda, \dot{s}\ra \Delta'[\fw,z,\varphi]_\phi(\gamma, \dot{\delta})
\]
for a semisimple strongly $G_0(V)$-regular element $\gamma \in H(F)$ and an element $\dot{\delta} \in G(V)(F)$ having a norm $\gamma$ in $H(F)$. By definition, only the coincidence of the normalization factors of both sides is non-trivial.
Put $h = \varphi^{-1}(g)$. Then, the definition of rigid inner twists implies that
\[
g\tau(g)^{-1} = \varphi(hz(\tau)\tau(h)^{-1}z(\tau)^{-1})
\]
for $\tau \in \Gamma$. Thus, we have $hz(\tau)\tau(h)^{-1} = \lambda(\tau)z(\tau)$ for $\tau \in \Gamma$. Let $\delta$ be an element in $G_0(V^\#)(F)$ having a norm $\gamma$, and let $g_1$ be an element of $G_0(V^\#)(\overline{F})$ so that $\varphi(g_1\delta g_1^{-1}) = \dot{\delta}$. Then, we have $\dot{\delta} = ((\Ad g) \circ \varphi)(h^{-1}g_1\delta g_1^{-1}h)$. Hence, we have
\begin{align*}
\inv_z^{(\Ad g)\circ\varphi}(\delta, \dot{\delta})(w) &= g_1^{-1}h\cdot z(w) \cdot w(h)^{-1}w(g_1) \\
&=g_1^{-1}\cdot\lambda(w) z(w) \cdot w(g_1) \\
&=\lambda(w)\cdot\inv_z^\varphi(\delta, \dot{\delta})(w)
\end{align*} 
for $w \in \cW$. This proves Lemma \ref{zurashi1}.\end{proof}

We remark that in the case \eqref{vsp II}, the natural homomorphism
\[
H^1(\Gamma, Z) \rightarrow H^1(\Gamma, Z(G(V^\#)))
\]
is surjective although $Z \not= Z(G(V^\#))$.

\begin{prop}\label{parameter orth}
Let $(z, \varphi)$ and $(z', \varphi')$ be rigid inner twists from $G(V_c^\#)$ onto $G(V)$. Assume that there exists an element $\gamma_0 \in G(V_c^\#)(\overline{F})$ so that $\varphi' = \varphi\circ\Ad \gamma_0$ and $z'(w) = \gamma_0^{-1}z(w)w(\gamma_0)$ for $w \in \cW$.
Then we have 
\[
\iota[\fw, z, \varphi]_\phi = \iota[\fw, z', \varphi']_\phi. 
\]
\end{prop}

\begin{proof}
We may assume that $\gamma_0 = \varepsilon$. First, we consider the case \eqref{vsp I} with $\epsilon = -1$. In this case, one can take a rigid inner form $z_0$ so that $\gamma_0 z_0(w) \gamma_0^{-1} = z_0(w)$ for all $w \in \cW$ and so that there exists $h \in G(V_c^\#)(\overline{F})$ such that $z(w) = h^{-1}z_0(w)w(h)$ for all $w \in \cW$. Then, consider the following diagram.
\[
\xymatrix{
 & G(V_c^\#) \ar[rr]^-{\Ad h'} \ar[dl]_-{\Ad \gamma_0} & & G(V_c^\#) \ar[d]^-{\varphi_0} \ar[dl]_-{\Ad \gamma_0} \\
 G(V_c^\#) \ar[rr]^-{\Ad h} \ar[d]_-{\varphi} & & G(V_c^\#) \ar[d]_-{\varphi_0} & G(V) \ar[dl]^-{\Ad \varphi_0(\gamma_0)} \\ 
 G(V) \ar@{=}[rr] & & G(V) & 
}
\]
Here we put $h' \coloneqq \gamma_0^{-1}h\gamma_0$ and $\varphi_0 \coloneqq (\Ad h)^{-1}\circ \varphi$. Then, we have $z'(w) = {h'}^{-1}z_0(w)w(h')$ all $w \in \cW$. Since $\varphi_0(\gamma_0) \in G(V)(F)$, we have
\begin{align*}
\iota[\fw, z', \varphi'](\pi) 
& = \iota[\fw, z_0, \varphi_0\circ\Ad\gamma_0](\pi) \\
& = \iota[\fw, z_0, (\Ad \varphi_0(\gamma_0))\circ\varphi_0](\pi) \\
& = \iota[\fw, z_0, \varphi_0](\pi\circ \Ad \varphi(\gamma_0)) \\
& = \iota[\fw, z, \varphi](\pi\circ \Ad \varphi(\gamma_0))\\
& = \iota[\fw, z, \varphi](\pi).
\end{align*}

Then, we consider the case \eqref{vsp III} with $\epsilon = -1$. If $\gamma_0 \in G_0(V_c^\#)(F)$, then the claim follows from \cite[Proposition 5.6]{Kal16}. Thus we may assume that $\det(\gamma_0) = -1$. Moreover, by using \cite[Proposition 5.6]{Kal16} again, we may assume that $\gamma_0 = \varepsilon$. To prove Proposition \ref{parameter orth} in this case, we return to the definition of the transfer factor.
Take $\dot{s} \in S_\phi^+$ and an endoscopic data $(H, \cH, \dot{t}, \eta) \in \cE(\dot{s})$. Then we have $(\Ad \widehat{\varepsilon}) \dot{s} \in S_{\Ad \widehat{\varepsilon}\circ\phi}^+$ and $(H, \cH, \dot{t}, \Ad\widehat{\varepsilon}\circ \eta) \in \cE((\Ad\widehat{\varepsilon})\dot{s})$. Take a semisimple strongly $G_0(V)$-regular element $\gamma \in H(F)$, and an element $\delta \in G_0(V_c^\#)(F)$ having a norm $\gamma$ via $\eta$, and a norm $\dot{\delta} \in G_0(V)(F)$ of $\delta$ via the inner twist $(z, \varphi)$, that is, there exists $\delta \in G_0(V^\#)(F)$ and $g_1 \in G(V^\#)_0(\overline{F})$ so that $\varphi(g_1\delta g_1^{-1}) = \dot{\delta}$. Put $g_1' \coloneqq \varepsilon g_1\varepsilon$ and $\delta' \coloneqq \varepsilon \delta \varepsilon^{-1}$. Then we have $\gamma$ is a norm of $\delta'$ via $(\Ad \widehat{\varepsilon})\circ \eta$, and $\dot{\delta}$ is a norm of $\delta'$ via the inner twist $(z', \varphi')$. More precisely, we have $\dot{\delta} = \varphi'(g_1' \delta' {g_1'}^{-1})$. Then, to prove Proposition \ref{parameter orth} in this case, it suffices to show that 
\begin{align}\label{eq transfer factor}
{\Delta'}^{(\phi, \dot{s})}[\fw, z, \varphi](\gamma, \dot{\delta}) = {\Delta'}^{((\Ad\widehat{\varepsilon})\circ\phi, (\Ad\widehat{\varepsilon}) \dot{s})}[\fw, z', \varphi'](\gamma, \dot{\delta}).
\end{align}
Here, we inserted the superscripts $(\phi, \dot{s})$ and $((\Ad\widehat{\varepsilon})\circ\phi, (\Ad\widehat{\varepsilon}) \dot{s})$ to specify the implicit data in the definitions. Suppose that the left-hand side of \eqref{eq transfer factor} is computed by using the splitting $(T^\#, B^\#, \{X_\alpha\}_\alpha)$ which defines $\fw$, the splitting $(\cT, \cB, \{\cX_{\widehat{\alpha}} \}_{\widehat{\alpha}})$ of $G_0(V_c^\#)^\wedge$, the $a$-data $\{a_\alpha\}_\alpha$, the $\chi$-data $\{\chi_\alpha\}_\alpha$, and the toral data (c.f. \cite{She08}) $u = u_{\gamma, \delta}\colon S_H(\gamma) \rightarrow S_G(\delta)$ (see \S\ref{Langpa}). Then, to compute the right-hand side of \eqref{eq transfer factor}, we put
\begin{itemize}
\item $X_\alpha' = \varepsilon (X_{\alpha\circ\Ad \varepsilon})\varepsilon^{-1}$ for $\alpha \in \Delta_-^\circ$,
\item $\cX_{\widehat{\alpha}}' = \widehat{\varepsilon} (\cX_{\widehat{\alpha}\circ\Ad \widehat{\varepsilon}}) \widehat{\varepsilon}^{-1}$ for $\alpha \in \Delta_-^\circ$,
\item $a_\alpha' = a_{\alpha \circ (\Ad\varepsilon)}$ for $\alpha \in R(G_0(V_c^\#), T^\#)$,
\item $\chi_\alpha' = \chi_{\alpha \circ (\Ad\varepsilon)}$ for $\alpha \in R(G_0(V_c^\#), T^\#)$,
\item and $u'(x) = \varepsilon u(x) \varepsilon^{-1}$ for $x \in S_H(\gamma)$.
\end{itemize}
Then, we have the splitting $(T^\#, B^\#, \{X_\alpha'\}_\alpha)$ which defines $\fw$, the splitting $(\cT, \cB, \{\cX_{\widehat{\alpha}}' \}_{\widehat{\alpha}})$ of $G_0(V_c^\#)^\wedge$, the $a$-data $\{a_\alpha'\}_\alpha$, the $\chi$-data $\{\chi_\alpha'\}_\alpha$, and the toral data $u'\colon S_H(\gamma) \rightarrow S_G(\delta)$ such that $u(\gamma) = \delta'$. Moreover, one can show that 
\begin{align*}
\Delta_\bullet^{(\phi, \dot{s})}[\fw, z, \varphi](\gamma, \delta) = \Delta_\bullet^{((\Ad\widehat{\varepsilon})\circ\phi, (\Ad \widehat{\varepsilon})\dot{s})}[\fw, z', \varphi'](\gamma, \delta')
\end{align*}
for $\bullet = {I}, {II}, {III_1}, {III_2}, {IV}$, and
\[
\inv_{z'}^{\varphi'}(\delta', \dot{\delta})(w)= \varepsilon \inv_z^{\varphi}(\delta, \dot{\delta})(w) \varepsilon^{-1}.
\]
for $w \in \cW$. Hence, we obtain \eqref{eq transfer factor}, and we complete the proof of Proposition \ref{parameter orth}.
\end{proof}

\begin{rem}
The equation \eqref{eq transfer factor} verifies \cite[Conjecture 2.12]{Kal23} for the automorphism $\Ad \varepsilon$ and the rigid inner twists $(z, \varphi)\colon G_0(V^\#)\rightarrow G_0(V)$.
\end{rem}

		
\section{
		The conjecture
		}\label{main sec}


Let $V$ be a right Hermitian space over $D$, let $W$ be a left skew Hermitian space over $D$, and let $c \in F^\times$. Define $\W$, $V_c^\#$, $W_c^\#$, $\W_c^\#$ as in \S\ref{groups}. Moreover, we use the terminologies $b_E$ and $\mathscr{J}_\W$ as in Lemma \ref{Morita_W}. 
By $(z_+, \varphi_+) \leftrightarrow (z_-, \varphi_-)$ we mean that there exist an isometry $\Omega\colon \W^\#\otimes_F\overline{F} \rightarrow \W\otimes_F\overline{F}$ over $\overline{F}$ such that 
\[
\Omega^{-1}\circ w\circ \Omega \circ w^{-1} = \iota^\#(z_+(w), z_-(w))
\]
for all $w \in \cW$ and the following diagram is commutative.
\begin{align}
\label{corr_rig}
\xymatrix{
\Sp(\W^\#) \ar[rr]^-{\varphi_\Omega} & & \Sp(\W) \\
G(V_c^\#) \times G(W_c^\#) \ar[u]^-{\iota^\#} \ar[rr]_-{(\varphi_+, \varphi_-)} & & G(V) \times G(W) \ar[u]_-{\iota}
}
\end{align}
Here, $\varphi_\Omega$ denotes the isomorphism induced by $\Omega$ (see \S\ref{notations}). 
Before stating the conjecture, we discuss some fundamental properties. We identify $Z_{V_c^\#}$ and $Z_{W_c^\#}$ by the isomorphism $a\cdot 1_{V_c^\#} \mapsto a \cdot 1_{W_c^\#}$ for $a \in Z(D) \cap D^1$. Then, for $\lambda_+ \in \cZ^1[V_c^\#]$ and $\lambda_- \in \cZ^1[W_c^\#]$, we write $\lambda_+ \leftrightarrow \lambda_-$ if $\lambda_-$ coincides with the image of $\lambda_+$ via the identification $Z_{V_c^\#} \rightarrow Z_{W_c^\#}$. For $H^1(\Gamma, Z_V)$ and $H^1(\Gamma, Z_W)$ we also define the correspondence $\leftrightarrow$ in the same way. Moreover, for $h_0 \in (G(V)/Z_V)(F)$ and $h_- \in (G(W)/Z_W)(F)$, we write $h_0 \leftrightarrow h_-$ if $\lambda_{h_0} \leftrightarrow \lambda_{h_-}$ where $\lambda_{h_0}$ (resp. $\lambda_{h_-}$) is the image of the connecting homomorphism $(G(V)/Z_V)(F) \rightarrow H^1(\Gamma, Z_V)$ (resp. $(G(W)/Z_W)(F) \rightarrow H^1(\Gamma, Z_W)$).

\begin{prop}\label{corr_funda}
\begin{enumerate}
\item Consider the cases \eqref{vsp I} and \eqref{vsp II}. Assume that there are isomorphisms $f_+\colon V^\natural \rightarrow V_c^\#$ over $F$ and $f_- \colon W^\natural \rightarrow W_c^\#$ over $F$, we have $(1_+, \fm_V^{-1}\circ\varphi_{f_+}^{-1}) \leftrightarrow (1_-, \fm_W^{-1}\circ \varphi_{f_-}^{-1})$. Here, $1_+$ (resp. $1_-$) denotes the constant function whose value is $1 \in G(V_c^\#)$ (resp. $1 \in G(W_c^\#)$). \label{corr_1}
\item Let $(z_+, \varphi_+)\in \RIT^\star(V_c^\#, V)$ and $(z_-, \varphi_-)\in \RIT^\star(W_c^\#, W)$ be rigid inner twists satisfying $(z_+, \varphi_+)\leftrightarrow (z_-, \varphi_-)$, let $(\lambda_+, h_0, g_+) \in \cZ^1[V_c^\#]\times (G(V)/Z_V)(F) \times G(V_c^\#)(\overline{F})$, and let $(\lambda_-, h_-, g_-) \in \cZ^1[W_c^\#]\times (G(W)/Z_W)(F)\times G(W_c^\#)(\overline{F})$. If $\lambda_+ \leftrightarrow \lambda_-$ and $h_0 \leftrightarrow h_-$ then we have
\[
(\lambda_+, h_0, g_+)\cdot(z_+, \varphi_+) \leftrightarrow (\lambda_-, h_-, g_-)\cdot(z_-, \varphi_-).
\] \label{corr_2}
\item Let $(z_+, \varphi_+)\in \RIT^\star(V_c^\#, V)$ and $(z_-, \varphi_-), (z_-', \varphi_-') \in \RIT^\star(W_c^\#, W)$ be rigid inner twists satisfying $(z_+, \varphi_+)\leftrightarrow (z_-, \varphi_-)$ and $(z_+, \varphi_+)\leftrightarrow (z_-', \varphi_-')$. Then, there exists $g \in G_0(W_c^\#)(\overline{F})$ such that $(1,1,g)\cdot(z_-, \varphi_-) = (z_-', \varphi_-')$. \label{corr_2.5}
\item There exist rigid inner twists $(z_+, \varphi_+)\in\RIT^\star(V_c^\#, V)$ and $(z_-, \varphi_-)\in\RIT^\star(W_c^\#, W)$ satisfying $(z_+, \varphi_+)\leftrightarrow (z_-, \varphi_-)$. \label{corr_3}
\end{enumerate}
\end{prop}

\begin{proof}
The assertions \eqref{corr_1} and \eqref{corr_2} are obviously. We prove \eqref{corr_2.5}. 
Let $\Omega, \Omega' \colon \W^\#\otimes\overline{F} \rightarrow \W\otimes\overline{F}$ be isometries over $\overline{F}$  such that 
\begin{align*}
&\Omega^{-1}\circ w \circ \Omega \circ w^{-1} = \iota^\#(z_+(w), z_-(w)), \\
&{\Omega'}^{-1}\circ w \circ {\Omega'} \circ w^{-1} = \iota^\#(z_+(w), z_-'(w))
\end{align*}
for $w \in \cW$ and 
\begin{align*}
& \varphi_\Omega\circ  \iota^\# = \iota\circ(\varphi_+, \varphi_-), \\
& \varphi_{\Omega'}\circ\iota^\#= \iota\circ(\varphi_+, \varphi_-').
\end{align*}
Put $g_0 = \Omega^{-1}\circ \Omega' \in \Sp(\W^\#)$. Then, for all $h \in G(V^\#)(\overline{F})$ we have 
\begin{align*}
g_0\iota^\#(h)g_0^{-1} &= (\varphi_\Omega^{-1}\circ\varphi_{\Omega'})(\iota^\#(h)) \\
& =(\varphi_+^{-1} \circ \varphi_+)(\iota^\#(h)) = \iota^\#(h).
\end{align*}
Hence we have $g_0 \in \iota^\#(1 \times G(W)(\overline{F}))$. Then, putting $g = \iota^{\#-1}(g_0) \in G(W^\#)(\overline{F})$, we have $(1,1,g)\cdot(z_-, \varphi_-) = (z_-', \varphi_-')$.
Finally, we prove \eqref{corr_3}. We denote by $L$ the natural limnear map $(V\otimes\overline{F})^\natural\otimes (W\otimes\overline{F})^\natural \rightarrow \W\otimes\overline{F}$ of \S\ref{morita}. Take isometries $A_+\colon V_c^\#\otimes\overline{F} \rightarrow (V\otimes\overline{F})^\natural$ and $A_-\colon W_c^\#\otimes\overline{F}\rightarrow (W\otimes\overline{F})^\natural$, and put $\Omega \coloneqq L\circ (A_+\otimes A_-)$. Then, by Lemma \ref{Morita_W}, we have that $\Omega$ is a bijective isometry linear map and that
\begin{align*}
&(\varphi_\Omega)(\iota^\#(G(V_c^\#)(\overline{F}) \times 1) = \iota(G(V)(\overline{F}) \times 1),\\
&(\varphi_\Omega)(\iota^\#(1\times G(W_c^\#)(\overline{F}))) = \iota(1\times G(W)(\overline{F})).
\end{align*}
Hence, we obtain isomorphisms $\varphi_+\colon G(V_c^\#)\rightarrow G(V)$ and $\varphi_-\colon G(W_c^\#)\rightarrow G(W)$ over $\overline{F}$, which make the diagram \eqref{corr_rig} commutative.
For $w \in \cW$, we regard $\Omega^{-1}\circ w\circ \Omega \circ w^{-1}$ as an element of $\Sp(\W_c^\#)(\overline{F})$. Since $\Ad (\Omega^{-1}\circ w\circ \Omega \circ w^{-1})$ preserves $\iota^\#(G(V_c^\#)\times 1)$ and $\iota^\#(1\times G(W_c^\#))$, it defines cocycles $c_+ \in Z^1(\Gamma, \Aut(G(V_c^\#)))$ and $c_- \in Z^1(\Gamma, \Aut(G(W_c^\#))$ respectively. Since $G(V)$ and $G(W)$ are inner forms of $G(V_c^\#)$ and $G(W^\#)$ respectively, we have $c_+ \in Z^1(\Gamma, G(V_c^\#)/Z_{V_c^\#})$ and $c_- \in Z^1(\Gamma, G_0(W_c^\#)/Z_{W_c^\#})$. Then, by Fact \ref{rig can surj}, there exists $z_+ \in Z^1(u\rightarrow \cW, Z\rightarrow G(V_c^\#))$ whose image in $Z^1(\Gamma, G(V_c^\#)/Z_{V_c^\#})$ coincides with $c_+$. Put
\[
z_-'(w) = \iota^\#(z_+(w), 1)^{-1}\cdot (\Omega^{-1} \circ w \circ \Omega \circ w^{-1}) \quad (w \in \cW).
\]
Then, for each $w \in \cW$, the element $z_-'(w)$ commutes with all elements of $\iota^\#(G(V_c^\#)\times1)$. Hence, $z_-\coloneqq \iota^{\#-1}\circ z_-'$ defines a cocycle in $Z^1(u\rightarrow \cW, Z \rightarrow G(W_c^\#))$ whose image in $Z^1(\Gamma, G(W_c^\#)/Z_{W_c^\#})$ is $c_-$. Thus, we obtain the rigid inner twists $(z_+, \varphi_+)$ and $(z_-, \varphi_-)$ satisfying $(z_+, \varphi_+) \leftrightarrow (z_-, \varphi_-)$. Hence we have \eqref{corr_3}, and we finish the proof of Proposition \ref{corr_funda}.
\end{proof}

\begin{rem}\label{rm_exist}
The proof of Proposition \ref{corr_funda} \eqref{corr_3} contains that of Proposition \ref{RIT_str} \eqref{nonempty}.
\end{rem}

We define an L-embedding
\[
\begin{cases}
 \xi\colon {}^LG_0(V_c^\#) \rightarrow {}^LG_0(W_c^\#) & \mbox{ if $n=m+1$} \\
 \xi\colon{}^LG_0(W_c^\#) \rightarrow {}^LG_0(V_c^\#) & \mbox{ if $n=m$}
\end{cases}
\]
as follows. 
\begin{itemize}
\item Consider the cases \eqref{vsp I} and \eqref{vsp III}. For a positive integer $N$, we denote by $S_N$ the quadratic space $\C^N$ over $\C$ equipped with the symmetric bilinear form obtained by the matrix $J_N$. Then there exists a bijective isometry $S_{N+1} \cong S_N \bot S_1$, which induces an embedding $\xi_0\colon \SO_N(\C) \rightarrow \SO_{N+1}(\C)$. If $n = m+1$, then we define the L-embedding $\xi$ by
\[
\xi(h \rtimes w) = \chi_V(w)\xi_0(\chi_W(w)h) \rtimes w \quad (h\rtimes w \in {}^LG_0(V_c^\#)),
\]
and if $n=m$, then we define $\xi$ by
\[
\xi(g \rtimes w) = \chi_W(w)\xi_0(\chi_V(w)h) \rtimes w \quad (g\rtimes w \in {}^LG_0(W_c^\#)).
\]
\item Consider the case \eqref{vsp II}. We fix an element $w_c \in W_F \setminus W_E$. 
If $n=m+1$, then we define the embedding $\xi$ by
\begin{align*}
&\xi(h\rtimes w) = \chi_V(w)\begin{pmatrix} \chi_W(w)\cdot{}^th^{-1} & 0 \\ 0 & 1 \end{pmatrix} \rtimes w \quad (h\rtimes w \in \GL_m(\C)\rtimes W_E), \mbox{ and}\\
&\xi(1 \rtimes w_c) = \begin{pmatrix} \Phi_m & 0 \\ 0 & 1 \end{pmatrix}\Phi_n^{-1} \rtimes w_c.
\end{align*}
If $n=m$, then we define $\xi$ by
\begin{align*}
&\xi(g \rtimes w) = \chi_V(w)\chi_W(w)\cdot {}^tg^{-1}\rtimes w \quad (g \rtimes w \in \GL_n(\C) \rtimes W_E), \mbox{ and}\\
&\xi(1 \rtimes w_c) = w_c.
\end{align*}
\end{itemize}

Let $\phi$ be a tempered $L$-parameter of $G(V)$, let $\phi'$ be a tempered $L$-parameter of $G(W)$, and let $(z_+, \varphi_+) \in \RIT^\star(V_c^\#, V)$ and $(z_-, \varphi_-) \in \RIT^\star(W_c^\#, W)$ be rigid inner twists. We say that $\phi$ and $\phi'$ satisfy the condition \eqref{L-par theta} if 
\begin{align}
&\mbox{there exist } \widehat{h} \in G(V^\#_c)^\wedge \mbox{ and } \widehat{g} \in G_0(W^\#_c)^\wedge\rtimes \la \widehat{\varepsilon} \ra \mbox{ such that } \label{L-par theta} \\
&\begin{cases} (\Ad \widehat{h}) \circ \phi  = \xi\circ (\Ad \widehat{g})\circ \phi' & \mbox{ if } n=m, \\ (\Ad \widehat{g})\circ \phi' = \xi\circ (\Ad \widehat{h}) \circ \phi & \mbox{ if } n=m+1. 
\end{cases} \notag
\end{align}
Note that $\phi'$ may not exist.  Assume that $(z_+, \varphi_+) \leftrightarrow (z_-, \varphi_-)$ and $\phi, \phi'$ satisfy \eqref{L-par theta}. Then, we define the map
\[
\mathscr{T}_\psi[c, (z_+, \varphi_+), (z_-, \varphi_-)]\colon \wPi_\phi(G(V)) \rightarrow \wPi_{\phi'}(G(W))_{weak} \cup \{0\}
\]
as follows.
Let $\pi \in \Pi_\phi(G(V)))$, and let $(\phi, \eta)$ be the Langlands parameter of $\pi$.
\begin{itemize}
\item  In the case \eqref{vsp II}, we may assume that $\widehat{h} = 1$ and $\widehat{g} = 1$. 
If there exists an irreducible tempered representation having the Langlands parameter $(\theta(\phi), \theta(\eta))$ that is defined as in \cite[\S4]{GI16}, then we denote it by $\mathscr{T}_\psi[c, (z_+, \varphi_+), (z_-, \varphi_-)](\pi)$. Otherwise, we put $\mathscr{T}_\psi[c, (z_+, \varphi_+), (z_-, \varphi_-)](\pi) = 0$.
\item In the cases \eqref{vsp I} and \eqref{vsp III} with $n=m$, $(\Ad \widehat{h}^{-1})\circ \xi \circ (\Ad \widehat{g})$ induces an embedding $S_{\phi'}^+ \rightarrow S_\phi^+$.
Then, there exists the unique irreducible representation $\eta' \in {\rm Irr}(S_{\phi'}^+, W)$ such that $(\eta')^\vee \subset \eta\circ(\Ad \widehat{h}^{-1})\circ\xi \circ (\Ad \widehat{g})$. If there exists an irreducible tempered representation having the Langlands parameter $(\phi', \eta')$, we denote it by $\mathscr{T}_\psi[c, (z_+, \varphi_+), (z_-, \varphi_-)](\pi)$. Otherwise, we put $\mathscr{T}_\psi[c, (z_+, \varphi_+), (z_-, \varphi_-)](\pi)=0$.
\item In the cases \eqref{vsp I} and \eqref{vsp III} with $n=m+1$, then $(\Ad \widehat{g}^{-1})\circ\xi \circ (\Ad \widehat{h})$ induces an embedding $S_{\phi}^+ \rightarrow S_{\phi'}^+$. There is a unique $\eta' \in {\rm Irr}(S_{\phi'}, W)$ such that $(\eta')^\vee\circ(\Ad \widehat{g}^{-1})\circ\xi\circ(\Ad \widehat{h})$ contains $\eta$. If there exists an irreducible tempered representation having the Langlands parameter $(\phi', \eta')$, then we denote it by $\mathscr{T}_\psi[c, (z_+, \varphi_+), (z_-, \varphi_-)](\pi)$. Otherwise, we put $\mathscr{T}_\psi[c, (z_+, \varphi_+), (z_-, \varphi_-)](\pi) = 0$.
\end{itemize}

Here, we used a basic fact about centers of spin groups (see Corollary \ref{spin_fact} below).

\begin{thm}\label{welldefness}
The map $\mathscr{T}_\psi[c, (z_+, \varphi_+), (z_-, \varphi_-)]$ does not depend on the choice of $c, (z_+, \varphi_+)$, and $(z_-, \varphi_-)$ whenever $(z_+, \varphi_+)\leftrightarrow(z_-, \varphi_-)$.
\end{thm}

\begin{proof}
First, we fix $c$. Let $(z_+, \varphi_+), (z_+', \varphi_+') \in \RIT^\star(V_c^\#, V)$ and $(z_-, \varphi_-), (z_-', \varphi_-') \in \RIT^\star(W_c^\#, W)$ be rigid inner twists so that $(z_+, \varphi_+) \leftrightarrow (z_-, \varphi_-)$ and $(z_+', \varphi_+') \leftrightarrow (z_-', \varphi_-')$. Then, by Proposition \ref{RIT_str} and Proposition \ref{corr_funda}, there exist $(\lambda_+, h_0, g_+) \in \cZ^1[V_c^\#] \times (G(V)/Z_V)(F) \times G(V_c^\#)(\overline{F})$ and $(\lambda_-, h_-, g_-) \in \cZ^1[W_c^\#] \times (G(W)/Z_W)(F) \times G(W_c^\#)(\overline{F})$ such that $\lambda_+ \leftrightarrow \lambda_-$, $h_0 \leftrightarrow h_0$ and 
\begin{align*}
(z_+', \varphi_+')&= (\lambda_+, h_0, g_+)\cdot (z_+, \varphi_+), \\
(z_-', \varphi_-')&= (\lambda_-, h_-, g_-)\cdot (z_-, \varphi_-). 
\end{align*}
By Lemma \ref{zurashi1}, Corollary \ref{zu1-2}, and Proposition \ref{parameter orth}, we have
\[
\mathscr{T}_\psi[c, (z_+, \varphi_+), (z_-, \varphi_-)] = \mathscr{T}_\psi[c, (z_+', \varphi_+'), (z_-', \varphi_-')].
\]

Then, we prove the independence from $c$. This is clear in the case \eqref{vsp II}. Hence, we consider the cases \eqref{vsp I} and \eqref{vsp III}. Take another element $c' \in F^\times$. 
Since $W_c^\# = W_{c'}^\#$ as vector space, the groups $G(W_c^\#)$ and $G(W_{c'}^\#)$ coincide. We denote by $\mathscr{J}_-$ the identity map from $G(W_{c'}^\#)$ onto $G(W_c^\#)$. We also denote by $\mathscr{J}_+$ the identity map from $G(V_{c'}^\#)$ onto $G(V_c^\#)$.  Then, the following diagram is commutative.
\[
\xymatrix{
 & \Sp(\W^\#) \ar[rr]^-{\varphi_\Omega} && \Sp(\W) \\ G(V_{c'}^\#) \times G(W_{c'}^\#) \ar[ru]^{\iota_{c'}^\#} \ar[r]_{(\mathscr{J}_+, \mathscr{J}_-)} & G(V_c^\#)\times G(W_c^\#) \ar[u]^{\iota_c^\#} \ar[rr]_{(\varphi_+, \varphi_-)} && G(V) \times G(W) \ar[u]_{\iota}
}
\]
Hence, putting $\varphi_\pm' \coloneqq \mathscr{J}_\pm\circ\varphi_{\pm}$ and $z_\pm' \coloneqq \mathscr{J}_\pm^{-1}\circ z_\pm$, we have $(z_+', \varphi_+') \leftrightarrow (z_-', \varphi_-')$ with respect to $c'$. Then, since the splitting $\spl(G(V_{c'}^\#)$ (resp. $\spl(G(W_{c'}^\#))$) is transfered to the splitting $\spl(G(V_c^\#)$ (resp. $\spl(G(W_c^\#))$) via $\mathscr{J}_+$ (resp. $\mathscr{J}_-$), we have
\[
\iota[\fw_c, z_\pm, \varphi_\pm]_\phi\circ\mathscr{J}_\pm = \iota[\fw_{c'}, z_\pm', \varphi_\pm']_\phi.
\]
Therefore, we have
\[
\mathscr{T}_\psi[c, (z_+, \varphi_+), (z_-, \varphi_-)] = \mathscr{T}_\psi[c', (z_+', \varphi'), (z_-', \varphi')].
\]
This completes the proof of Theorem \ref{welldefness}.  
\end{proof}

In the rest of this paper, we write $\mathscr{T}_\psi$ instead of $\mathscr{T}_\psi[c, (z_+, \varphi_+), (z_-, \varphi_-)]$. 

\begin{conj}\label{conj_main}
Assume that $\epsilon = 1$. Let $\phi$ be a tempered $L$-parameter for $G_0(V)$. If there exists a tempered $L$-parameter $\phi'$ satisfying \eqref{L-par theta}, then we have $\theta_\psi(\pi, W) = \mathscr{T}_\psi(\pi)$. 
\end{conj}

It is not difficult to show that Conjecture \ref{conj_main} is equivalent to the weak version (in the sense of \cite{AG17}) of the Prasad conjecture which is already proved in the non-Archimedean cases \cite{Ato18}\cite{GI16} (see also \S\ref{conv_theta} below). Summarizing:

\begin{fact}\label{Pre_PC}
Assume that $F$ is a non-Archimedean local field. Then, Conjecture \ref{conj_main} holds in the cases \eqref{vsp I} and \eqref{vsp II}.
\end{fact}

If $F = \R$, Conjecture \ref{conj_main} will be verified in the cases \eqref{vsp I} and \eqref{vsp III} (Theorem \ref{sp-o real}) below. In addition, if $F$ is non-Archimedean, Conjecture \ref{conj_main} will be verified in the case \eqref{vsp III} with $m=n=1$ (Theorem \ref{m=n=1_conj}) below.


\section{
		Computations in Archimedean local Langlands correspondences
		}\label{Arch_computation}

\subsection{Settings}

In this section, we consider the cases \eqref{vsp I} and \eqref{vsp III} with $F = \R$ and $\epsilon =1$. We denote the quaternion algebra over $\R$ by
\[
\H = \R \oplus \R i \oplus \R j \oplus \R ij
\]
where $i, j$ are the symbols satisfying the relations
\[
i^2 = -1, \quad j^2 = e_\H, \quad  ij+ji = 0,
\]
where $e_\H = \pm 1$. If $e_\H=-1$ then $\H$ is called the skew-field of Hamilton quaternions. We denote by $\sigma$ the nontrivial element of the Galois group $\Gamma$. Then the Weil group is given by the formal disjoint union
\[
\W_\R = \C^\times \cup \C^\times \widetilde{\sigma}
\] 
where $\widetilde{\sigma}$ the symbol satisfying $\widetilde{\sigma}^2 = -1$ and $\widetilde{\sigma}\cdot z = \overline{z}\cdot \widetilde{\sigma}$ for $z \in \C^\times$.

For a non-negative integer integers $p,q$, we denote by $V_{p,q}$ the right $\H$-vector space of column vectors of degree $p+q$ equipped with the Hermitian form $( \ , \ )$ on $V_{p,q}$ given by
\[
(x, y) = (\sum_{k=1}^p x_k y_k^*) - (\sum_{k=p+1}^{p+q} x_k y_k^*)
\]
for $x, y \in V_{p,q}$. Here, we denote by $x_k$ (resp. $y_k$) the $k$-th component of $x$ (resp. $y$).
We also denote by $W_{p,q}$ the left $\H$-vector space of row vectors of degree $p + q$ equipped with the skew-Hermitian form $\la \ , \ \ra$ on $W_{p,q}$ given by
\[
\la x, y \ra = (\sum_{k=1}^p x_k i y_k^*) - (\sum_{k=p+1}^{p+q} x_k i y_k^*)
\]
for $x,y \in W_{p,q}$. Here, we denote by $x_k$ (resp. $y_k$) the $k$-th component of $x$ (resp. $y$). 

\subsection{Splittings}\label{splitting}

We denote by $T_+^\#$ the maximal torus of $G(V_c^\#)$ consisting of the diagonal matrices in $G(V_c^\#)$,  by $B_+^\#$ the Borel subgroup of $G(V_c^\#)$ containing all upper triangle matrices in $G(V_c^\#))$, and by $\alpha_k^\#$ the algebraic character of $T_+^\#$ projecting the $(k,k)$-component of $T_+^\#$. Then, 
\[
\Delta_+^\circ = \{ \alpha_1^\# - \alpha_2^\#, \ldots , \alpha_{m-1}^\# - \alpha_m^\# , 2\alpha_m^\# \}.
\]
is a basis of $\Delta_{B_+^\#}$. Then, we put
\[
X_{\alpha_k^\#-\alpha_{k+1}^\#} = e_{k, k+1}(1) + e_{2m+1-k, 2m-k}(-1)
\]
for $k = 1, \ldots, m-1$ and put
\[
X_{2\alpha_m^\#} = e_{m, m+1}(1).
\]
Then, we have the splitting $(T_+^\#, B_+^\#, \{X_\alpha\}_{\alpha \in \Delta_+^\circ})$ associated with $c$. One can show that  $(T_+^\#, B_+^\#, \{X_\alpha\}_{\alpha \in \Delta_+^\circ})$ defines the Whittaker data $\fw_+^{(c)}$.

We denote by $A_-^\#$ the maximal split torus consisting of diagonal matrices in $G_0(W_c^\#)$, by $T_-^\#$ the centralizer of $A_-^\#$ in $G_0(W_c^\#)$,  by $B_-^\#$ the Borel subgroup of $G(W_c^\#)$ containing all upper triangle matrices in $G(W_c^\#))$.
For $1 \leq k \leq n-1$, we denote by $\beta_k^\#$ the algebraic character of $T_-^\#$ projecting the $(k,k)$-component of $T_-^\#$. Moreover, we define
\[
\beta_n^\#(\begin{pmatrix} a & & & \\ & x & y & \\ & dy & x & \\ & & & J_{n-1} a^{-1} J_{n-1}\end{pmatrix}) = x + \sqrt{d} y
\]
for a diagonal matrix $a$ and $x,y \in \C$ so that $x^2 - d y^2= 1$. 
Then, 
\[
\Delta_-^\circ = \{ \beta_1^\# - \beta_2^\#, \ldots , \beta_{m-1}^\# - \beta_m^\# , \beta_{m-1}^\# + \beta_m^\# \}
\]
is a basis of $\Delta_{B_-^\#}$. Finally, we define
\[
X_{\beta_k^\#-\beta_{k+1}^\#} = e_{k, k+1}(1) + e_{2n+1-k, 2n-k}(-1)
\]
for $k = 1, \ldots, n-2$ and put
\begin{align*}
X_{\beta_{n-1}^\# - \beta_n^\#} &= e_{n-1,n}(\frac{1}{2}) + e_{n-1,n+1}(\frac{1}{2\sqrt{d}}) + e_{n,n+2}(-1) + e_{n+1,n+2}(\sqrt{d}), \\
X_{\beta_{n-1}^\#+\beta_n^\#} &= e_{n-1,n}(\frac{1}{2}) + e_{n-1,n+1}(-\frac{1}{2\sqrt{d}}) + e_{n,n+2}(-1) + e_{n+1,n+2}(-\sqrt{d}).
\end{align*}
Then, we have the splitting  $(T_-^\#, B_-^\#, \{Y_\beta\}_{\beta \in \Delta_-^\circ})$. One can show that  $(T_-^\#, B_-^\#, \{Y_\beta\}_{\beta \in \Delta_-^\circ})$ defines the Whittaker data $\fw_-^{(c)}$.

\subsection{Anisotropic tori}

Let $S_+$ be the maximal torus of $G(V_{p,q})$ of the form 
\[
\{\diag(x_1+ i y_1, \ldots, x_m+ i y_m) \in G(V_{p,q}) \mid x,y \in \R, x_k^2+y_k^2 = 1 \ (1 \leq k \leq m)\},
\]
We choose a basis $\alpha_1, \ldots, \alpha_m$ of $X^*(S_+)$ where $\alpha_k$ is given by
\[
\alpha_k(\diag(a_1 + i b_1, \ldots, a_m + i b_m)) = a_k + \sqrt{-1} b_k \in \C^\times.
\]
By this basis we identify $X^*(S_+)$ with $\Z^m$. Let $S_-$ be the maximal torus of $G(W_{p,q})$ of the form
\[
\{\diag(x_1+ i y_1, \ldots, x_n+ i y_n) \in G(W) \mid x,y \in \R, x_k^2+y_k^2 = 1 \ (1 \leq k \leq n)\}.
\]
We also chose a basis $\beta_1, \ldots, \beta_n$ of $X^*(S_-)$ where $\beta_k$ is given by
\[
\beta_k(\diag(x_1 + i y_1, \ldots, x_n + i y_n)) = x_k + \sqrt{-1} y_k \in \C^\times.
\]
By this basis we identify $X^*(S_-)$ with $\Z^n$.

We consider the embedding $\varsigma_+\colon (\C^1)^m \rightarrow G(V_{p,q})$ given by
\[
\varsigma_+(x_1 + \sqrt{-1}y_1, \ldots, x_m + \sqrt{-1} y_m) =\diag(x_1 + i y_1, \ldots, x_m + iy_m)
\]
for $x_1 + \sqrt{-1}y_1, \ldots,  x_m + \sqrt{-1} y_m \in \C^1$. We consider the embedding $\varsigma_+^\#\colon (\C^1)^m \rightarrow \Sp(V^\#)$ given by
\[
\varsigma_+^\#(x_1 + \sqrt{-1}y_1, \ldots, x_m + \sqrt{-1} y_m) = \begin{pmatrix} x_1 & & & & & y_1 \\ & \ddots & & & \iddots & \\ & & x_m & y_m & & 
\\ & & -y_m & x_m & & \\ & \iddots & & & \ddots & \\ -y_1 & & & & & x_1 \end{pmatrix}
\]
for $x_1 + \sqrt{-1}y_1, \ldots,  x_m + \sqrt{-1} y_m \in \C^1$. We denote by $S_+^\#$ the image of $\varsigma_+^\#$.

We define the $2n$-dimensional quadratic space $W_{\sim}^\#$ over $\R$ of the row vectors whose quadratic form is given by
\[
Q_n = \begin{pmatrix} 2I_{2t} & 0 \\ 0 & -2I_{2n - 2t} \end{pmatrix}\] 
where $t = \lceil n/2 \rceil$. Put 
\[
Q = \begin{pmatrix} I_{n-1} & & & \\  & 1 & 1 & \\ & 1 & -1 & \\ & & & I_{n-1} \end{pmatrix}, \quad P_1 = \begin{pmatrix} I_{2t} & & & \\ & J_2 & &  \\ & & \ddots & \\ & & & J_2 \end{pmatrix}
\]
and
\[
P_0 = \begin{cases} \begin{pmatrix} I_n & J_n \\ I_n & -J_n \end{pmatrix} & \mbox{ if $n$ is even}, \\ \begin{pmatrix} I_{n-1} & & J_{n-1} \\ & I_2 & \\ I_{n-1} & & -J_{n-1} \end{pmatrix} & \mbox{ if $n$ is odd}.\end{cases}
\]
Then, putting
\[
P = \begin{cases} P_1P_0Q^{-1} & \mbox{ if $n$ is even}, \\ P_1P_0 & \mbox{ if $n$ is odd}, \end{cases}
\]
we have $Q_n = {}^tP S_n P$ where $S_n$ is the matrix $( \langle \mathbf{e}_k, \mathbf{e}_l \rangle^\# )_{k,l}$.
We define $\varsigma_{\sim}^\#\colon (\C^1)^n \rightarrow \SO(W_{\sim}^\#)$ by
\[
\varsigma_{\sim}^\#(x_1 + \sqrt{-1}y_1, \ldots, x_n + \sqrt{-1}y_n) = \begin{pmatrix} x_1 & y_1 &  &  & \\ -y_1 & x_1 & & & \\ & & \ddots & & \\ & & & x_n & y_n \\ & & & -y_n & x_n \end{pmatrix}
\]
for $x_1 + \sqrt{-1}y_1, \ldots,  x_n + \sqrt{-1} y_n \in \C^1$. Then, we define $\varsigma_-^\# = \varphi_P^{-1} \circ \varsigma_{\sim}^\#$, and we denote by $S_-^\#$ the image of $\varsigma_-^\#$.


\subsection{Weyl groups}

It is useful to describe the actions of Weyl groups on tori. 
For a positive integer $k$, we denote by $\fS_k'$ the semi-direct product $\fS_k \ltimes \{\pm 1\}^k$ with respect to the action of $\fS_k$ on $\{\pm 1\}^k$ given by
\[
\gamma \cdot (\epsilon_1, \ldots, \epsilon_k) = (\epsilon_{\gamma^{-1}(1)}, \ldots, \epsilon_{\gamma^{-1}(k)}) 
\]
for $\gamma \in \fS_k$ and $\epsilon_1, \ldots, \epsilon_k \in \{\pm 1\}$. The group $\fS_k'$ acts on $\Z^k$ by
\begin{align*}
&\gamma\cdot \varsigma_+^\#(a_1, \ldots, a_k) = \varsigma_+^\#(a_{\gamma^{-1}(1)}, \ldots, a_{\gamma^{-1}(k)}), \\
&(\epsilon_1, \ldots, \epsilon_k)\cdot \varsigma_+^\#(z_1, \ldots, z_k) = \varsigma_+^\#(\epsilon_1\cdot a_1, \ldots, \epsilon_k\cdot a_k)
\end{align*}
for $a_1, \ldots, a_k \in \Z$, $\gamma \in \fS_k$, and $(\epsilon_1, \ldots, \epsilon_k) \in \{\pm 1\}^k$. Hence, $\fS_m'$ acts on $X^*(S_+^\#)$ and $X^*(S_+)$, and $\fS_n'$ acts on $X^*(S_-^\#)$ and $X^*(S_-)$. Moreover, they induces the algebraic actions of $\fS_m'$ on $S_+^\#$, $S_+$ and of $\fS_n'$ on $S_-^\#$, $S_-$. By these action, we identify $\fS_m'$ (resp. $\fS_n'$) with the Weyl groups $W(S_+^\#, G(V^\#))$, $W(S_+, G(V))$ (resp. $W(S_-^\#, G(W^\#))$, $W(S_-, G(W))$.

\subsection{Harish-Chandra parameters and Langlands parameters}\label{HCLpar}

In this subsection, we compute the Langlands parameter of a discrete series representation with the Harish-Chandra parameter using the transfer factor of Langlands-Shelstad \cite{LS87}. 
Let $G$ be a connected reductive group over $\R$, let $G^\#$ be the quasi-split inner form of $G$ equipped with the inner twist $\varphi\colon G^\# \rightarrow G$, let $(T^\#, B^\#)$ be a Borel pair in $G^\#$ defined over $\R$, and let $G^\wedge$ be the Langlands dual group of $G^\#$ equipped with the Borel pair $(\cT, \cB)$ of $G^\wedge$. We assume that $G^\#$ contains an anisotropic maximal torus $S^\#$ so that  $\varphi(S^\#)$ is an anisotropic maximal torus of $G$ defined over $\R$. As in \cite[p.~15]{Mez13}, we may assume that $\phi$ is consistent with $(\cT, \cB)$ (see \S\ref{Lgrp}) by taking a conjugacy by an element of $G^\wedge$.

Now, we will describe the $L$-packet of $\phi$ and determine the Langlands parameter for each element of the $L$-packet. Put $S = \varphi(S^\#)$ and put
\[
\cA(S^\#, T^\#) = \{ g \in G(\C) \mid gS^\#g^{-1} = T^\# \}.
\]
Following \cite{Mez13}, we use the $a$-data $\{a_\alpha\}_\alpha$ and $\chi$-data $\{\chi_\alpha\}_\alpha$  given by
\begin{align*}
a_\alpha &= \begin{cases} -\sqrt{-1} & \alpha \in \Delta_{B^\#}, \\ \sqrt{-1} & \alpha \not\in \Delta_{B^\#}\end{cases}\\
\chi_\alpha(z) &=  \begin{cases} |z|/z & \alpha \in \Delta_{B^\#}, \\ z/|z| & \alpha \not\in \Delta_{B^\#} \end{cases}
\end{align*}
for $z \in \C^\times$. Mezo proved the endoscopic character relation constructing the ``spectral transfer factor'' $\Delta_{spec}(\pi, s)$ whose appropriate normalization is $e(G)\cdot \iota_\phi[\fw, z, \varphi]$. We put $q_G = (1/2)(\dim G - \dim K)$ where $K$ is the maximal compact subgroup. Summarizing Mezo's computations (\cite[(115)--(117)]{Mez13}) in our setting (with the trivial twisting), we have the following.

\begin{fact}\label{mez_fml}
Let $\pi$ be an irreducible discrete series representation having its Harish-Chandra parameter $\mu \in X^*(S)$, let $s \in S_\phi^+$, and let $(H, \cH, \eta, \dot{t})$ an endoscopic data in $\cE(s)$. Assume that $\mu = \mu_\phi\circ (\Ad g) \circ (\Ad w) \circ \varphi^{-1}$ where $g \in \cA(S^\#, T^\#)$ and $w \in W(G_0(V_c^\#), S^\#)$. Let $\gamma_1$ be an regular element of $H_1$ so that the centralizer $C_H(\gamma_1)$ is an anisotropic torus, let $h_1$ be an element of $H_1(\overline{F})$ so that $h_1\gamma_1h_1^{-1} \in T_{H_1}^\#$, and let $\delta_g$ be the image of $\gamma_1$ in $S^\#(F)$ by the homomorphism $(\Ad g^{-1})\circ \underline{\eta} \circ (\Ad h_1)$ where $\underline{\eta}$ is the homomorphism $T_{H_1}^\# \rightarrow T^\#$ which commutes with $\eta$. Put $\delta_\mu = w \delta w^{-1}$. Then, we have $\pi \in \Pi_\phi(G)$ and 
\begin{align*}
&\iota_\phi[\fw, z, \varphi](\pi)(s) \\
&= (-1)^{q_{G_0(V_c^\#)} - q_{H}}\cdot (-\sqrt{-1})^{\#\Delta_B - \#\Delta_{B_H}}\cdot \epsilon(\cV_{G_0(V_c^\#),H}, \psi) \\
&\quad \times \la \inv_z(\delta_g, \delta_\mu), (\Ad g)^\wedge (s) \ra \cdot \Delta_I(\gamma_1, \delta_g).
\end{align*}
\end{fact}

\begin{rem} \label{c_mez}
Fact \ref{mez_fml} differs from the formula of Mezo \cite[(115)--(117)]{Mez13} slightly. More precisely, we use $(-\sqrt{-1})^{\#\Delta_B - \#\Delta_{B_H}}$ instead of $\sqrt{-1}^{\#\Delta_B - \#\Delta_{B_H}}$. This is necessary since there is an error in \cite[(75)]{Mez13} which expands the second factor $\Delta_{II}$. We explain the details in Appendix \ref{spec_correction} below.
\end{rem}

\begin{cor}\label{formula_LLC}
Let $\phi$ be a tempered $L$-parameter for $G$, let $\fw$ be a Whittaker data of $G^\#$, let $\mu_\fw$ the Harish-Chandra parameter for $G^\#$ so that $\pi(\mu_\fw)$ is the generic representation in $\Pi_\phi(G^\#)$, and let $\mu$ be a Harish-Chandra parameter so that $\pi(\mu) \in \Pi_\phi(G)$. Choose a rigid inner twist $(z, \varphi)\colon G^\# \rightarrow G$. Then, we have
\[
\iota_\phi[\fw, z, \varphi](\pi)(s) = \la \inv_z(\mu_{h_\fw}, \mu), (\Ad h_\fw)^\wedge(s) \ra
\]
\end{cor}

We return to the case where $G$ is $G(V_c^\#)$ or $G_0(W_c^\#)$. In this case, we have 
\[
C_\phi = \{ \ \widehat{t}(s_1, \ldots, s_N) \mid s_k \in \{\pm 1\} \ (k=1, \ldots, N) \}.
\] 
For $s = \widehat{t}(s_1, \ldots, s_N)$, we put $a(s) = \#\{k = 1, \ldots, N \mid s_k=1\}$ and $b(s) = \#\{k =1, \ldots, N \mid s_k = -1 \}$. 
\begin{lem}\label{norm_char}
Let $G$ be either $G(V_c^\#)$ or $G_0(W_c^\#)$. Then we have 
\begin{align*}
&(-1)^{q_{G} - q_{H}} \cdot (-\sqrt{-1})^{\#\Delta_B - \#\Delta_{B_H}} \cdot \epsilon(\cV_{G,H}, \psi) \\
&= \begin{cases} (-\sqrt{-1} \cdot \epsilon_\psi)^{b(s)} & (G = G(V_c^\#)), \\  1 & (G = G_0(W_c^\#)). \end{cases}
\end{align*}
\end{lem}

\begin{proof}
First, assume that $G = G_0(W_c^\#)$. In this case, $H = \SO(2a(s), \sgn^{a(s)}) \times \SO(2b(s), \sgn^{b(s)})$. Then, we have
\[
\#\Delta_B - \#\Delta_{B_H} = 2a(s)b(s), \quad \cV_{G,H} = \sgn^m - \sgn^{a(s)} - \sgn^{b(s)}.
\]
 Moreover, since the symmetric spaces attached to even special orthogonal groups have even dimensional, we have
 \[
 q_{G} - q_H \equiv 0 \mod 2.
 \]
 Hence we have 
 \[
 (-1)^{q_{G} - q_H} \cdot (-\sqrt{-1})^{\#\Delta_B - \#\Delta_{B_H}} \cdot \epsilon(\cV_{G,H}, \psi) =1.
  \]
Then, assume that $G=G(V_c^\#)$. In this case, $H = \Sp_{2a(s)} \times \SO(2b(s), \sgn^{b(s)})$. Then, we have
\[\#\Delta_B - \#\Delta_{B_H} = (2a(s) + 1)b(s), \quad \cV_{G,H} = {\rm triv} - \sgn^{b(s)}.\]
Moreover, we have
\[
q_{G} = \frac{1}{2}m(m+1), \quad q_H \equiv \frac{1}{2}a(s)(a(s) + 1) \mod 2.
\]
Hence we have
\begin{align*}
&(-1)^{q_{G} - q_H} \cdot (-\sqrt{-1})^{\#\Delta_B - \#\Delta_{B_H}} \cdot \epsilon(\cV_{G,H}, \psi) \\
&= (-\sqrt{-1})^{2a(s)b(s) + b(s)(b(s)+1))}\cdot (-\sqrt{-1})^{(2a(s) + 1)b(s)}\cdot \epsilon(\cV_{G,H}, \psi) \\
&=(-\sqrt{-1})^{b(s)^2+2b(s)}\cdot \epsilon(\cV_{G,H}, \psi) \\
& = \begin{cases} 1 & \mbox{if $b(s)$ is even}, \\ -\sqrt{-1}\epsilon_\psi & \mbox{if $b(s)$ is odd.} \end{cases}
\end{align*}
Thus, we have Lemma \ref{norm_char}.
\end{proof}

\subsection{Generic representations}\label{gen_rep}

In this subsection, we compute the Harish-Chandra parameters of the generic irreducible representations of $G(V_c^\#)(\R)$ and $G(W_c^\#)(\R)$ in given discrete series $L$-packets.

\begin{prop}\label{Sp_gen}
Putting $\rho_+ = \diag(-1, 1, \ldots, (-1)^m)$ and
\[
h_0 = \frac{1}{\sqrt{2}}\begin{pmatrix} I_m & c\epsilon_\psi\rho_+ J_m \\ c\epsilon_\psi J_m \rho_+ & I_m \end{pmatrix} \in \cA(S_+^\#, T_+^\#),
\]
the irreducible discrete series representation of $G(V_c^\#)(\R)$ having Harish-Chandra parameter $\mu_\phi\circ(\Ad h_0)$ is generic.
\end{prop}

\begin{proof}
It suffices to show that $\Delta_I(\gamma_1, \delta_{h_0}) = (-\sqrt{-1}\epsilon_\psi)^{b(s)}$. Recall that the factor $\Delta_I(-, -)$ is given by the Tate-Nakayama pairing of $(u_+^{-1}\circ \Ad h_0^{-1})^\wedge(s) \in (S_+^\#)^\wedge$ and the cocycle $\lambda(S_+^\#) \in H^1(\Gamma, S_+^\#)$ which is defined in \cite[(2.3)]{LS87}. To compute it, we use some symbols defined in \cite[(2.3)]{LS87}. The cocycle is given by
\[
\lambda(S_+^\#)(\tau) = h_0^{-1} x(\tau_{S_+^\#})n(\omega_{S_+^\#}(\tau)) \tau(h_0) 
\]
for $\tau \in \Gamma$. Here, the factor $x(\tau)$ is the factor defined by using the $a$-data and the $\chi$-data , and $n(\tau)$ is the factor define by using the splitting $\{X_\alpha\}_{\alpha \in \Delta^\circ}$. In our setting, we have
\begin{align*}
n(\omega_{S_+^\#}(\sigma)) &= (-1)^{m-1}c \cdot \begin{pmatrix} & J_m \\ -J_m & \end{pmatrix}, \\
x(\sigma_{S_+^\#}) &= \sqrt{-1}\cdot \begin{pmatrix} J_m \rho_+ J_m & \\ & -\rho_+ \end{pmatrix}.
\end{align*}
Hence, we have
\[
\lambda_+(\sigma) = (-\sqrt{-1}\epsilon_\psi) \cdot I_{2m},
\]
which implies $\Delta_I(\gamma_1, \delta_{h_0}) = (-\sqrt{-1}\epsilon_\psi)^{b(s)}$.
\end{proof}

Recall that we put $t =\lceil n/2 \rceil$. Define $g_1 \in G_0(W_{\sim}^\#)(\C)$ by
\[
f_{k} \cdot g_1
= \begin{cases} f_{k+1} & k \mbox{ is odd}, 1 \leq k \leq 2(n-t), 2t \leq k \leq 2n \\ 
\sqrt{-1} \cdot f_{k+2t-1} & k \mbox{ is even}, 1 \leq k \leq 2(n-t), \\
\sqrt{-1} \cdot f_{k -2t-1} & k \mbox{ is even}, 2t \leq k \leq 2n, \\
f_{k} & 2(n-t) < k < 2t.
\end{cases}
\]
Moreover, put $g_0  = P^{-1}g_1 P \in G(W_c^\#)$.

\begin{prop}\label{SO_gen}
Assume $c =1$. The irreducible representation of $G_0(W_c^\#)(\R)$ having Harish-Chandra parameter $\mu_\phi\circ(\Ad g_0)$ is generic.
\end{prop}

\begin{proof}
It suffices to verify that $\Delta_I(\gamma_1, \delta_{g_0}) = 1$. 
As in the proof of Proposition \ref{Sp_gen}, we use the symbols $\lambda(S_-^\#), x(\tau_{S_-^\#})$, and $n(\omega_{S_-^\#}(\tau))$ defined in \cite[(2.3)]{LS87}. 
We compute $\lambda(S_-^\#)$ separately depending on the parity of $n$. It is useful to put 
\[
a_0 = \diag(1,-1, \ldots, (-1)^{2n-2t-1}) \in \GL_{2n-2t}(\R).
\]

First, assume that $n$ is even.  From our choice of the $a$-data $\{a_\beta\}_{\beta}$, the $\chi$-data $\{\chi_\beta\}_{\beta}$, and the splitting $\{Y_\beta\}_{\beta\in \Delta_-^\circ}$, we obtain
\[
n(\omega_{S_-^\#}(\sigma)) = -Q^{-1}J_{2n}Q, \quad x(\sigma_{S_-^\#}) = Q^{-1} \begin{pmatrix} -a_0 & \\ & a_0 \end{pmatrix} Q.
\]
Hence, we have
\begin{align*}
\lambda(S_-^\#)(\sigma) &= g_0^{-1} x(\sigma_{S_-^\#}) n(\omega_{S_-^\#}(\sigma)) \sigma(g_0) \\
&= -P^{-1}g_1^{-1}P_1P_0 \begin{pmatrix} -a_0 & \\ & a_0 \end{pmatrix} J_{2n} P_0^{-1}P_1^{-1} \sigma(g_1) P \\
&= P^{-1}g_1^{-1}\begin{pmatrix} a_0 & \\ & a_0 \end{pmatrix} \sigma(g_1) P.
\end{align*}
Moreover, since
\[
\sigma(g_1) = \begin{pmatrix} a_0 & \\ & a_0 \end{pmatrix} g_1,
\]
we have $\lambda(S_-^\#)(\sigma) = 1$. 

Then, assume that $n$ is odd. Then, we have
\[
n(\omega_{S_-^\#}(\sigma)) = \begin{pmatrix} & & J_{n-1} \\ & I_2 & \\ J_{n-1} & & \end{pmatrix}, \quad x(\sigma_{S_-^\#}) =  \begin{pmatrix} a_0 & & \\ & I_2 & \\ & & -a_0 \end{pmatrix}.
\]
Hence, we have
\begin{align*}
\lambda(S_-^\#)(\sigma) &= g_0^{-1} x(\sigma_{S_-^\#})n(\omega_{S_-^\#}(\sigma))\sigma(g_0) \\
&=P^{-1}g_1^{-1}P\begin{pmatrix} a_0 & & \\ & I_2 & \\ & & -a_0 \end{pmatrix} \begin{pmatrix} & & J_{n-1} \\ & I_2 & \\ J_{n-1} & & \end{pmatrix} P^{-1}\sigma(g_1)P  \\
&=P^{-1}g_1^{-1}\begin{pmatrix} a_0 & & \\ & I_2 & \\ & & a_0 \end{pmatrix}\sigma(g_1)P \\
&=1.
 \end{align*}
 This completes the proof of Proposition \ref{SO_gen}.
\end{proof}

Then, we introduce some notations.

\begin{df}\label{def_pxi}
If $n=m$, then the restriction of $L$-embedding $\xi$ to $\cT_-$ gives the isomorphism $\xi|_{\cT_-} \colon \cT_- \rightarrow \cT_+$. In this case, we denote by $I_\xi \colon \cT_+ \rightarrow \cT_-$ the inverse of $\xi|_{\cT_-}$. If $n=m + 1$, then we denote by $I_\xi \colon \cT_+ \rightarrow \cT_-$ the restriction of $\xi$ to $\cT_+$. In both cases, we define the homomorphism $\fp_\xi \colon S_-^\# \rightarrow S_+^\#$ so that the following diagram is commutative.
\[
\xymatrix{
X^*(S_+^\#) \ar[rrr]^-{\fp_\xi^*} \ar[d]_{\Ad h_0^{-1}} & & & X^*(S_-^\#) \ar[d]^{\Ad g_0^{^-1}} \\ 
X^*(T_+^\#) \ar[r]_{\fD_+} & X_*(\cT_+)  \ar[r]_{(I_\xi)_*} & X_*(\cT_-) & X^*(T_-^\#) \ar[l]^{\fD_-} 
}
\]
\end{df}

Let $\rho_1$ be an element of $\fS_n$ given by
\[
\rho_1(k) = \begin{cases} (k+1)/2 & (k: \mbox{ odd}), \\ t + k/2 & (k: \mbox{ even}), \end{cases}
\] 
and let $u = (u_1, \ldots, u_n)$ be an element of $\{\pm 1\}^n$ given by
\begin{align}\label{u_k}
u_k = \begin{cases} -\sqrt{-1}\epsilon_\psi & (1 \leq k \leq t), \\ \sqrt{-1}\epsilon_\psi & (t+1 \leq k \leq n).\end{cases}
\end{align}
Then, we have
\begin{align}\label{fpxi}
(\fp_\xi\circ\varphi_P^{-1})((u\cdot\rho_1) \cdot \varsigma_\sim^\#(z_1, \ldots, z_n)) = \varsigma_+^\#(z_1, \ldots, z_{m})
\end{align}
for $z_1, \ldots, z_n \in \C^1$.

\begin{lem}\label{key_Gal}
Let  $w$ and $w'$ be elements of $N(G(W_c^\#), S_-^\#)$ and $N(G(V_c^\#), S_+^\#)$ respectively. If there exists $\rho \in \fS_n$ such that $(\Ad w)(x) = \rho\cdot x$ and $(\Ad w')(\fp_\xi(x)) = \fp_\xi(\rho \cdot x)$ for all $x \in S_-^\#$, then we have $w\sigma(w)^{-1} \in S_-^\#(\C)$, $w'\sigma(w')^{-1} \in S_+^\#(C)$ and $\fp_\xi(w\sigma(w)^{-1}) = w'\sigma(w')^{-1}$. 
 \end{lem}
 
 \begin{proof}
For $\rho \in \fS_n$, there exist $w_\rho' \in N(G(V_c^\#), S_+^\#)$ such that the action of $\Ad w_\rho'$ on $S_+^\#$ commutes with $\rho$ via $\fp_\xi$ if and only if $\rho \in \rho_1\fS_m\rho_1^{-1} \subset \fS_n$  by \eqref{fpxi}. We denote by $w_\rho$ an element of $N(G(W_c^\#), S_-^\#)$ such that the action of $\Ad w_\rho$ on $S_+^\#$ coincides with that of $\rho$. For $\rho, \tau \in \rho_1\fS_m\rho_1^{-1}$, we have 
\begin{align*}
\fp_\xi(w_\rho w_\tau \sigma(w_\tau^{-1} w_\rho^{-1})) &= (\Ad w_\rho')(\fp_\xi((w_\tau\sigma(w_\tau))) \cdot \fp_\xi(w_\rho\sigma(w_\rho^{-1})) \\
&=w_\rho' \sigma(w_\rho')^{-1}\cdot w_\tau' \sigma(w_\tau')^{-1}
\end{align*}
if $w_\rho \sigma(w_\rho )^{-1} \in S_-^\#(C)$, $w_\tau\sigma(w_\tau)^{-1} \in S_-^\#(C)$, $\fp_\xi(w_\rho\sigma(w_\rho)^{-1}) = w_\rho'\sigma(w_\rho')^{-1}$ and $\fp_\xi(w_\tau\sigma(w_\tau)^{-1}) = w_\tau'\sigma(w_\tau')^{-1}$. Hence, it remains to show Lemma \ref{key_Gal} in the case where $\rho$ is a transportation $(\rho_1(k), \rho_1(k+1))$ for some $k=1, \ldots, m-1$, which is contained in $\rho_1\fS_m\rho_1^{-1}$. Then, putting $u' = \rho_1^{-1}(u)$, we have
\begin{align*}
(\Ad w_\rho') (\varsigma_+^\#(z_1, \ldots, z_m)) &= \fp_\xi(\rho\cdot u \cdot \rho_1\cdot \varsigma_-^\#(z_1, \ldots, z_n)) \\
&= \fp_\xi(\rho_1^{-1}\rho\rho_1\cdot \rho_1^{-1}\rho^{-1}\rho_1(u')\cdot u'\cdot\varsigma_-^\#(z_1, \ldots, z_n) )
\end{align*}
for $z_1, \ldots, z_n \in \C^1$. Moreover, we have 
\[
\rho_1^{-1}\rho^{-1}\rho_1(u')\cdot u' = (b_1, \ldots, b_n)
\]
where $b_l = 1$ if $l \not=k, k+1$ and $b_k = b_{k+1} = -1$. Hence, we have the action of $\Ad w_\rho'$ on $S_+^\#$ coincides with that of $\rho_1^{-1}\rho\rho_1\cdot (b_1, \ldots, b_m) \in \fS_m'$. Thus, we have 
\begin{align}\label{fpxi_1}
w_\rho'\sigma(w_\rho')^{-1} = \varsigma_+^\#(b_1, \ldots, b_m).
\end{align}
On the other hand, if we choose $w_\rho \in N(G(W_c^\#), S_-^\#)$ whose action of $S_-^\#$ coincides with $\rho$, then we have 
\begin{align}\label{fpxi_2}
w_\rho^{-1} \sigma(w_\rho) &= \varsigma_-^\#(b_{\rho_1^{-1}(1)}, \ldots, b_{\rho_1^{-1}(n)}) \\ 
&= u\cdot \rho_1\cdot \varsigma_-^\#(b_1, \ldots, b_n).\notag
\end{align}
Therefore, by \eqref{fpxi}, we have $\fp_\xi(w_\rho\sigma(w_\rho)^{-1}) = w_\rho'\sigma(w_\rho')^{-1}$ in this case. Thus, we finish the proof of Lemma \ref{key_Gal}.
\end{proof}

\subsection{
		Parametrizations of the limits of discrete series representations
		}

To describe the set of Harish-Chandra parameters, we define some symbols. Let $p,q,N$ be non-negative integers. 
In the case $e_\H = 1$, put
\begin{align*}
\Delta_c^+ &=  \{ \alpha_k - \alpha_l \mid k < l \}, \\
\Delta_c^- &=\{ \beta_k \pm \beta_l \mid  k < l, \  (2p + 1 - 2k)(2p + 1 - 2l) > 0\}. \end{align*}
In the case $e_\H = -1$, put
\begin{align*}
\Delta_c^+ &= \{ \alpha_k \pm \alpha_l \mid  k < l, \  (2p + 1 - 2k)(2p + 1 - 2l) > 0\} \cup \{ 2\alpha_k \mid 1 \leq k \leq m\}, \\
\Delta_c^- &= \{ \beta_k - \beta_l \mid k < l \}
\end{align*}
We denote by $\mathfrak{P}^+$ (resp. $\mathfrak{P}^-$) the set of the positive systems of $R(G_0(V),S_+)$ (resp. $R(G_0(W), S_-)$) containing $\Delta_c^+$ (resp. $\Delta_c^-$). We denote by $\cX$ the set of the pairs $(\mu, \Psi) \in X^*(S_+) \times \mathfrak{P}^+$ satisfying 
\begin{itemize}
\item $\la \mu, \alpha \ra \geq 0$ for all $\alpha \in \Psi$ and
\item $\la \mu, \alpha \ra > 0$ for all $\alpha \in \Delta_c^+$,
\end{itemize}
by $\cY$ the set of the pairs $(\mu', \Psi') \in X^*(S_-) \times \mathfrak{P}^-$ satisfying 
\begin{itemize}
\item $\la \mu', \beta \ra \geq 0$ for all $\beta \in \Psi'$ and
\item $\la \mu', \beta \ra > 0$ for all $\beta \in \Delta_c^-$. 
\end{itemize}
It is known that for an irreducible limit of discrete series representation $\sigma$ of $G(V)(\R)$,  an element $(\mu_\sigma, \Psi_\sigma)$ of $\cX$ is attached, and for an element of irreducible limits of discrete series representations $\pi$ of $G_0(W)(\R)$, an element $(\mu_\pi, \Psi_\pi)$ of $\cY$ is attached (c.f. \cite{Har66}, \cite[Chapter XII, \S7]{Kna01}). 
If $\mu \in X^*(S_+)$ (resp. $\mu' \in X^*(S_-)$) is nonsingular and positive with respect to $\Delta_c^+$ (resp. $\Delta_c^-$), then the set 
\begin{align*}
&\Psi_\mu = \{ \alpha \in R(G_0(V), S_+) \mid \la \alpha, \mu \ra > 0 \} \\
&(\mbox{resp. } \Psi_{\mu'} = \{ \beta \in R(G_0(W), S_-) \mid \la \beta, \mu' \ra > 0 \})
\end{align*}
is a positive system of $R(G_0(V), S_+)$ (resp. $R(G_0(W), S_-)$), and $(\mu, \Psi_\mu) \in \cX$ (resp. $(\mu', \Psi_{\mu'}) \in \cY$). For such a pair, an irreducible discrete series representation is attached. 
We define $\xi^u$ for $u \in \{\pm\sqrt{-1}\}$ as follows. 
\begin{itemize}
\item Consider the case $e_\H = 1$ and $n = p+q = m$. We define
\begin{align*}
&\xi^{\sqrt{-1}}\colon \Z^m \rightarrow \Z^n, \quad (a_1, \ldots, a_m) \mapsto (-a_m, \ldots, -a_{q+1}, a_1, \ldots, a_q), \\
&\xi^{-\sqrt{-1}}  \colon \Z^m \rightarrow \Z^n, \quad (a_1, \ldots, a_m) \mapsto (a_1, \ldots, a_p, -a_m, \ldots, -a_{p+1}).
\end{align*}
\item In the case $e_\H = 1$ and $n = p + q = m + 1$. We define 
\begin{align*}
&\xi_\btru^{\sqrt{-1}}\colon \Z^m \rightarrow \Z^n, \quad (a_1, \ldots, a_m) \mapsto (-a_m, \ldots, -a_q, a_1, \ldots, a_{q-1},0), \\
&\xi_\btru^{-\sqrt{-1}}  \colon \Z^m \rightarrow \Z^n, \quad (a_1, \ldots, a_m) \mapsto (a_1, \ldots, a_{p-1},0,  -a_m, \ldots, -a_p), \\
&\xi_\btrd^{\sqrt{-1}}\colon \Z^m \rightarrow \Z^n, \quad (a_1, \ldots, a_m) \mapsto (-a_m, \ldots, -a_{q+1}, 0, a_1, \ldots, a_q), \\
&\xi_\btrd^{-\sqrt{-1}}  \colon \Z^m \rightarrow \Z^n, \quad (a_1, \ldots, a_m) \mapsto (a_1, \ldots, a_p,  -a_m, \ldots, -a_{p+1}, 0).
\end{align*}
\item In the case $e_\H = -1$ and $n = m = p+q$. We define 
\begin{align*}
&\xi^{\sqrt{-1}}\colon \Z^m \rightarrow \Z^n, \quad (a_1, \ldots, a_m) \mapsto (a_1, \ldots, a_p, -a_m, \ldots, -a_{p+1}) \\
&\xi^{-\sqrt{-1}}  \colon \Z^m \rightarrow \Z^n, \quad (a_1, \ldots, a_m) \mapsto (a_{p+1}, \ldots, a_m, -a_p, \ldots, -a_1).
\end{align*}
\item In the case $e_\H = -1$ and $n = m+1 = p+q+1$. We define
\begin{align*}
&\xi^{\sqrt{-1}} \colon \Z^m \rightarrow \Z^n, \quad (a_1, \ldots, a_m) \mapsto (a_1, \ldots, a_p, 0, -a_m, \ldots, -a_{p+1}) \\
&\xi^{-\sqrt{-1}}  \colon \Z^m \rightarrow \Z^n, \quad (a_1, \ldots, a_m) \mapsto (a_{p+1}, \ldots, a_m, 0, -a_p, \ldots, -a_1).
\end{align*}
\end{itemize}

For each case, we define $\xi_\bullet^u(\Psi)$ as follows where $\xi_\bullet^u$ denotes either $\xi^u$, $\xi_\btru^u$ or $\xi_\btrd^u$. Take $\nu \in X^*(S_+)$ so that $\nu > 0$ with respect to $\Psi$. Then, $\mu + \nu$ is regular and $\xi_\bullet^u(\mu + \nu) \in X^*(S_-)$ is also regular. Then, we define $\xi_\bullet^u(\Psi) = \Psi_{\xi_\bullet^u(\mu + \nu)}$. One can show that $\xi_\bullet^u(\Psi)$ does not depend on the choice of $\nu$. Then, the local theta correspondence for $(G(V), G(W))$ is described as follows. 

\begin{fact}\label{pre_R}
Let $(\mu, \Psi) \in \cX$.
\begin{enumerate}
\item Assume $e_\H = 1$ and $n= m + 1$. Then, $\theta_\psi(\pi(\mu, \Psi), W) \not=0$ if and only if either $(\xi_\btru^{\epsilon_\psi}(\mu), \xi_\btru^{\epsilon_\psi}(\Psi)) \in \cY$ or $(\xi_\btrd^{\epsilon_\psi}(\mu), \xi_\btrd^{\epsilon_\psi}(\Psi)) \in \cY$. Moreover, $\xi_\btru^{\epsilon_\psi}(\mu)$ (resp. $\xi_\btrd^{\epsilon_\psi}(\mu)$) is the Harish-Chandra parameter of the $G(W)(\R)$-equivalent class of $\theta_\psi(\pi(\mu), W)$ if $\xi_\btru^{\epsilon_\psi}(\mu) \in \cX_{p,q}$ (resp. if $\xi_\btru^{\epsilon_\psi}(\mu)$). \label{pre_R1}
\item Assume either $e_\H = -1$ or $e_\H=1$ with $n=m$. Then $\theta_\psi(\pi(\mu), W) \not=0$ if and only if $(\xi^{\epsilon_\psi}(\mu), \xi^{\epsilon_\psi}(\Psi)) \in \cY$. Moreover, the $G(W)(\R)$-equivalent class of $\theta_\psi(\pi(\mu, \Psi), W)$ has the Harish-Chandra parameter $(\xi^{\epsilon_\psi}(\mu), \xi^{\epsilon_\psi}(\Psi))$ if it is non-zero. \label{pre_R2}
\end{enumerate}
\end{fact}

\begin{rem}\label{point}
These results had been proven by contributions of many researchers \cite{KV78} \cite{Moe89} \cite{Li89} \cite{Pau05} \cite{LPTZ03}. However, some comments are necessary.
\begin{enumerate}[(1)]
\item In \cite{Li89}, Li discussed both cases $e_\H = \pm 1$, and proved Fact \ref{pre_R} in the case where $\mu$ and $\xi_\bullet^{\epsilon_\psi}(\mu)$ are regular ($\xi_\bullet^{\epsilon_\psi}$ denotes either $\xi^{\epsilon_\psi}$, $\xi_\btru^{\epsilon_\psi}$ or $\xi_\btrd^{\epsilon_\psi}$). Moreover, the proof of \cite{Li89} using the characterization of  ``$A_\mathfrak{q}(\lambda)$'' (c.f. \cite[Proposition 6.1]{VZ84}) is still valid for all cases where we discussed in Fact \ref{pre_R}. However, the non-trivial additive character $\psi$ of $\R$ in the definition of the Weil representation is implicit. We address the convention problem in Appendix \ref{app theta HC 1} below. In conclusion, the Weil representation he considered is that associated with a non-trivial additive character $\psi$ satisfying $\epsilon_\psi = \sqrt{-1}$. \label{point0}
\item 
In the case $e_\H = 1$, M{\oe}gline also described the local theta correspondence in terms of Harish-Chandra parameters \cite{Moe89}, which is extended to general case by Paul \cite{Pau05}. 
However, the description differs from Fact \ref{pre_R}. More precisely, she had chosen $\psi$ so that $\epsilon_\psi = -\sqrt{-1}$ to specify the Weil representation, but her description is that obtained by $\xi_\bullet^{\sqrt{-1}}$. This seems to be caused by an error in \cite[I.4]{Moe89} in interpreting the result of Kashiwara-Vergne \cite{KV78} into her setting. We explain the details in \S\ref{notes on psi} below. \label{point_1} 
\label{point_1}
\item 
In the case $e_\H = -1$, Li, Paul, Tan, and Zhu \cite{LPTZ03} extended the result of Li \cite{Li89} to the correspondence between irreducible admissible representations. However, the non-trivial additive character $\psi$ of $\R$ in the definition of the Weil representation is implicit. By tracking the proof, one can conclude that they used the same $\psi$ as in \cite{Li89}.
\end{enumerate}
\end{rem}

For an irreducible limit of discrete series representation $\pi$ of $G_0(V)$ (resp. $\pi'$ of $G_0(W)$) associated with $(\mu, \Psi)$ (resp. $(\mu', \Psi')$) and a positive element $\nu \in X^*(S_+)$ (resp. $\nu' \in X^*(S_-)$)  with respect to $\Psi$ (resp. $\Psi'$), we denote by $S_\nu\cdot \pi$ (resp. $S_{\nu'}\cdot \pi'$) the limit of discrete series representation associated with $(\mu+\nu, \Psi)$ (resp. $(\mu' + \nu', \Psi')$). Discussions on the constructions of such representations and their characters, called the ``coherent continuations'', can be seen in \cite{Zuc77} \cite{SV80}, but we do not use it in this paper. We only use the commutativity of the coherent continuations and the local theta correspondences, which follows from Fact \ref{pre_R} immediately:

\begin{cor}\label{coh_vs_theta}
For an irreducible limit of discrete series representation $\pi$ of $G(V)(\R)$, we have 
\[
\theta_\psi(S_\nu\cdot\pi, W) = S_{\xi^{\epsilon_\psi}(\nu)}\cdot \theta_\psi(\pi, W).
\]
\end{cor}

\subsection{Parabolic inductions}

In this subsection, we discuss the behavior of the Langlands parameter under parabolic inductions. Let $P$ be a parabolic subgroup of $G_0(V)$ defined over $\R$, and let $M$ be its Levi subgroup. Choose $(z, \varphi) \in \RIT_M^\star(V^\#, V)$ (see \S\ref{RIT_Levi}) and put $P^\# = \varphi^{-1}(P)$ and $M^\# = \varphi^{-1}(M)$. We may assume that $P^\#$ contains $B_+^\#$ or $B_-^\#$ (\S\ref{WD}). Hence, by the restriction, we obtain the Whittaker data $\fw^M$ for $M$ from a Whittaker data $\fw$ of $G_0(V)$. We denote by $\Delta'( - , - )_M$ the geometric transfer factor for $M$ associated with $(z, \varphi)\colon M^\# \rightarrow M$. Then, one can verify that the geometric transfer factors $\Delta'( - , - )$ and $\Delta'( - , - )_M$ are ``normalized compatibly'' in the sense of \cite{She08} (c.f. \cite[Appendix B]{Mez16}). 
Moreover, we obtain the following useful property of the Langlands parameters. Let $\phi$ be a tempered $L$-parameter for $M$. We denote by $S_\phi^+(M)$ the inverse image of ${\rm Cent}_{M^\wedge}(\Im \phi)$ in $\overline{M}^\wedge$. Then, identify $S_\phi^+(M)$ with a subgroup of $S_\phi^+(G_0(G))$ in the natural way. 

\begin{cor}\label{temp_key}
Let $\pi_0$ be an irreducible tempered representation of $M(\R)$, and let $\pi$ be an irreducible component of $\Ind_{P(\R)}^{G_0(V)(\R)} \pi_0$. Then, $\pi$ is a tempered representation having the same $L$-parameter as $\pi_0$, and we have
\[
\iota^M[\fw^M, z, \varphi](\pi_0)(s) = \iota[\fw, z, \varphi](\pi)(s)
\]
for $s \in S_\phi^+(M)$.
 \end{cor}
 
\begin{proof}
The temperedness of $\pi$ follows from the direct estimation of the matrix coefficients (c.f. \cite[p.~ 198]{Kna01}). The remaining part follows from the argument of the parabolic descent (c.f. \cite[\S6.3]{Mez16}).
\end{proof}


\section{
		The cases \eqref{vsp I} and \eqref{vsp III} with $F = \R$
		}\label{Arch_SpO}

In this section, we consider the cases \eqref{vsp I} and \eqref{vsp III} with $F =\R$. 
In the case  \eqref{vsp I}, $\H$ is isomorphic to the matrix algebra $\rM_2(\R)$ as an $\R$-algebra. Then, by the Morita equivalence (\S\ref{morita}), we have that $V^\natural$ is the symplectic space and $W_{p,q}^\natural$ is the $2n$-dimensional quadratic space of signature $(2q, 2p)$. In the case \eqref{vsp III}, $G(V)$ and $G(W)$ are quaternionic unitary groups. Recall that $G_0(W)(\R)$ coincides with $G(W)(\R)$ in this case.

\begin{thm}\label{sp-o real}
Let $\pi$ be an irreducible tempered representation of $G(V)(\R)$, and let $\phi$ be its $L$-parameter. Assume that there exists an $L$-parameter $\phi'$ of $G_0(W)$ satisfying \eqref{L-par theta}. Then, the $G(W)(\R)$-equivalent class of $\theta_\psi(\pi, W)$ coincides with $\mathscr{T}_\psi(\pi)$. 
\end{thm}

The proof of Theorem \ref{sp-o real} will be finished at the end of this section. We explain more precisely. In \S\ref{red_real}, we reduce Theorem \eqref{sp-o real} in the case $\pi$ is a discrete series representation by using properties of parabolic inductions. In \ref{choice_real}, we show that Theorem \ref{sp-o real} for an irreducible discrete series representation $\pi$ follows from the existence of certain rigid inner twists $(z_+, \varphi_+) \in \RIT^\star(V^\#, V)$ and $(z_-, \varphi_-) \in \RIT^\star(W^\#, W)$ satisfying some conditions (Proposition \ref{SOSP_key_dia_quat}). Then we prove Proposition \eqref{SOSP_key_dia_quat} separately depending on the cases \eqref{vsp I} and \eqref{vsp III} (\S\ref{prf_I}, \S\ref{prf_III}).

\subsection{Reductions to discrete series representations}\label{red_real}

First, we study the following non-vanishing property of $\cT_\psi(\pi)$.

\begin{lem}\label{countingT}
Let $V$ be a right $m$-dimensional Hermitian space over $\H$, let $\pi$ be an irreducible tempered representation of $G(V)(\R)$, and let $\phi$ be its $L$-parameter.
For a left skew Hermitian space $W$ over $\H$ of dimension $m$ or $m+1$, we write $\mathscr{T}_\psi^W(\pi)$ instead $\mathscr{T}_\psi(\pi)$ to specify $W$. We put $\mathscr{T}_\psi^W(\pi) = 0$ if there do not exist an $L$-parameter $\phi'$ of $G_0(W)$ satisfying \eqref{L-par theta}. Then we have
\begin{enumerate}
\item In the case \eqref{vsp I}, there are precisely four isometry classes of left skew-Hermitian spaces $W$ so that $\dim W = m, m+1$ and $\mathscr{T}_\psi^W(\pi) \not=0$. \label{nonvan_real1}
\item In the case \eqref{vsp III}, for a left skew Hermitian space $W$ over $\H$ of dimention $m$ or $m+1$, we have $\mathscr{T}_\psi^W(\pi) \not=0$ if there exists an $L$-parameter $\phi'$ of $G_0(W)$ satisfying \eqref{L-par theta}. \label{nonvan_real2}
\end{enumerate}
\end{lem}

\begin{proof}
The assertion \eqref{nonvan_real2} follows from the fact that the map $\pi' \mapsto \rho_{\pi'}$ of Conjecture \ref{LLCclg} is a bijection between $\wPi_{\phi'}(G_0(W))$ and ${\rm Irr}(\cS_{\phi'}^+, W)$ in the case \eqref{vsp III}. It remains to prove \eqref{nonvan_real1}. We consider the following two conditions.
\begin{enumerate}[(C1)]
\item The representation $\std \circ \phi$ of $W_\R$ contains the trivial representation. \label{condi TR}
\item The representation $\std \circ \phi$ of $W_\R$ contains the sign representation. \label{condi SG}
\end{enumerate}
We denote by $R_\pm$ (resp. $R_\pm'$) the number of the isometry classes of the skew-Hermitian spaces $W$ so that $\mathscr{T}_\psi^{W}(\pi) \not=0$, $\fd(W) = \pm 1$, and $\dim W = m$ (resp. $\dim W = m+1$). Then the numbers $R_\pm$ and $R_\pm'$ are determined completely whether the conditions (C\ref{condi TR}) and (C\ref{condi SG}) are true or false as Table \ref{table}.
\begin{table}
\caption{}\label{table}
\begin{tabular}{|c|c|c|c|c|c|} \hline
(C1) & (C2) & $R_+$ & $R_-$ & $R_+'$ & $R_-'$ \\ \hline
True & True & $1$ & $1$ & $1$ & $1$ \\ \hline
True & False & $1$ & $0$ & $1$ & $2$ \\ \hline
False & True & $0$ & $1$ & $2$ & $1$ \\ \hline
False & False & $0$ & $0$ & $2$ & $2$ \\ \hline
\end{tabular}
\end{table}
In any case in Table \ref{table}, the sum $R_+ + R_- + R_+' + R_-'$ coincides with $4$. This implies \eqref{nonvan_real1}. 
\end{proof}

\begin{rem}\label{countingP}
It is known that precisely four isometry classes of skew-Hermitian spaces $W$ over $\H$ those satisfy $\theta_\psi(\pi, W) \not= 0$ and $\dim W = m, m+1$. (See \cite[Corollary 23]{Pau05} for more details.) 
\end{rem}

\begin{prop}\label{temp-DS}
If Theorem \ref{sp-o real} holds for all $V$ and for all irreducible discrete series representations, then it holds for all $V$ and for all irreducible tempered representations.
\end{prop}

\begin{proof}
Assume that Theorem \ref{sp-o real} is proved for all irreducible discrete series representations at once.
Then, by the compatibility of local theta correspondences and coherent continuations (Corollary \ref{coh_vs_theta}), we have Theorem \ref{sp-o real} for all limits of discrete series representations.
Consider the case where $\pi$ is an arbitrary irreducible tempered representation of $G(V)$. 

Assume there exists $\phi'$ satisfying \eqref{L-par theta} and that $\mathscr{T}_\psi(\pi)$ is non-zero. It is known that there exist a parabolic subgroup $Q$ of $G_0(W)$ so that the Levi-subgroup $L$ is isomorphic to $G_0(W_\bullet) \times \mathcal{G}_r(\R)$ where $W_\bullet$ is the $(n - r)$-dimensional skew-Hermitian space over $\H$ and $\mathcal{G}_r$ is an inner form of $\GL_r$, an irreducible tempered representation $\tau_1$ of $\mathcal{G}_r(\R)$, and an irreducible limit of discrete series representation $\tau_\bullet$ of $G_0(W_\bullet)$ such that 
\[
\mathscr{T}_\psi(\pi) = \Ind_{Q(\R)}^{G_0(W)(\R)} \tau_\bullet \boxtimes \tau_1, 
\]
the image of $(\Ad t_\bullet^{-1})\circ\phi'$ is contained in ${}^LL$ for some $t_\bullet \in G_0(W)^\wedge$, and the homomorphism 
\[
\cS_{(\Ad t_\bullet^{-1})\circ\phi}^+(L) \rightarrow \cS_\phi^+(G_0(W))
\]
induced by $\Ad t_\bullet$ is surjective. (The existence follows from the work of Shelstad \cite[\S5.4]{She82} and its update in terms of the local Langlands correspondence for rigid inner twists done by Kaletha \cite[\S 5.4]{Kal16}.) Then, we have that $\phi$ is contained in a Levi-subgroup 
\[
(\SO(2m+1-2r, \C) \times \GL_r(\C)) \rtimes W_\R \subset {}^LG_0(V).
\]
Hence, there is a parabolic subgroup $Q$ of $G_0(V)$ so that its Levi-subgroup $L$ is isomorphic to $G_0(V_\bullet) \times \GL_r$ where $W_\bullet$ is $(m-r)$-dimensional Hermitian space over $\H$. This means that there exist  irreducible tempered representations $\pi_\bullet$ and $\pi_1$ of $G_0(V_\bullet)(\R)$ and $\mathcal{G}_r(\R)$ respectively so that 
\[
\pi \subset \Ind_{P(\R)}^{G_0(V)(\R)}\pi_\bullet\boxtimes\pi_1.
\]
Then, by Corollary \ref{temp_key}, we have $\mathscr{T}_\psi(\pi_\bullet) = \tau_\bullet$ which is non-zero. Moreover, using the arguments of L-parameters \cite[\S4.3]{She82}, one can show that $\pi_\bullet$ is a limit of discrete series. Hence, by the assumption of Lemma, we have $\theta_\psi(\pi_\bullet, W_\bullet) = \tau_\bullet$. Then, by the ``induction principle'', we have that $\theta_\psi(\pi, W)$ is non-zero and is a direct summand of $ \Ind_{Q(\R)}^{G_0(W)(\R)} \tau_\bullet\boxtimes \tau_1$, which implies that $\theta_\psi(\pi, W) = \mathscr{T}_\psi(\pi)$.

Finally, by Lemma \ref{countingT} and Remark \ref{countingP}, we have that $\mathscr{T}_\psi(\pi) \not=0$ if and only if $\theta_\psi(\pi, W) \not=0$. This proves Proposition \ref{temp-DS}.
\end{proof}

\subsection{The key proposition}\label{choice_real}

The following proposition is the key to proving Theorem \ref{sp-o real} in the case where $\pi$ is a discrete series representation. Put $\epsilon_1 = \cdots = \epsilon_p = 1$, $\epsilon_{p+1} = \cdots = \epsilon_n = -1$, and 
\[
\underline{\epsilon} = \begin{cases} (1 ,\ldots, 1) & (e_\H = 1), \\ (\epsilon_1, \ldots, \epsilon_n) & (e_\H = -1, \epsilon_\psi = \sqrt{-1}),  \\ (-\epsilon_n, \ldots, -\epsilon_1) & (e_\H = -1, \epsilon_\psi = -\sqrt{-1}).\end{cases}
\]

\begin{prop}\label{SOSP_key_dia_quat}
Let $\xi_\bullet^{\epsilon_\psi}$ denotes $\xi^\epsilon$ (resp. either $\xi_{\btru}^{\epsilon_\psi}$ or $\xi_{\btrd}^{\epsilon_\psi}$) if $n=m$ (resp. $n=m+1$). Then, there exist $(z_+, \varphi_+) \in \RIT^\star(V^\#, V)$ and $(z_-, \varphi_-) \in \RIT^\star(W^\#, W)$ such that
\begin{itemize}
\item $(z_+, \varphi_+) \leftrightarrow (z_-, \varphi_-)$, 
\item $z_+(w) \in S_+^\#(\C), z_-(w) \in S_-^\#(\C) \quad (w \in \cW)$,  
\item $\fp_\xi (z_-(w)) = z_+(w)^{-1} \quad (w \in \cW)$,
\item there exists $\rho \in \fS_n$ such that for $z_1, \ldots, z_n$
\[
\varphi_-(\varsigma_-^\#(z_1, \ldots, z_n)) = (\underline{\epsilon} \cdot \rho) \cdot \varsigma_-(z_1, \ldots, z_n),
\]
\item and the following diagram is commutative.
\[
\xymatrix{
X^*(S_\sim^\#) \ar[rr]^-{(\varphi_P\circ \varphi_-^{-1})^*} & & X^*(S_-) \\
X^*(S_+^\#) \ar[rr]_-{(\varphi_+^{-1})^*} \ar[u]^{(\fp_\xi \circ \varphi_P^{-1})^*} & & X^*(S_+) \ar[u]_{\xi_\bullet^{\epsilon_\psi}}
}
\]
\end{itemize}
\end{prop}

We will prove Proposition \ref{SOSP_key_dia_quat} in \S\ref{prf_I} and \S\ref{prf_III} below. In this subsection, we show that Theorem \ref{sp-o real} for a discrete series representation $\pi$ follows from Proposition \ref{SOSP_key_dia_quat}. 

\begin{proof}[Proof of Theorem \ref{sp-o real}]
Let $\pi$ be an irreducible discrete series representation of $G(V)(\R)$, and let $\phi$ be its $L$-parameter. Take $(z_+, \varphi_+) \in \RIT^\star(V^\#, V)$ and $(z_-, \varphi_-) \in \RIT^\star$ as in Proposition \ref{SOSP_key_dia_quat}. Assume that there exists an L-parameter $\phi'$ of $G_0(W)$ satisfying \eqref{L-par theta} and that $\theta_\psi(\pi, W)\not=0$. 
We may assume that $\phi$ is consistent with $(\cT_+, \cB_+)$ and that $\phi'$ is consistent with $(\cT_-, \cB_-)$ (c.f. \S\ref{HCLpar}). Hence, we obtain $\mu_\phi \in X^*(T_+^\#)$ and $\mu_{\phi'} \in X^*(T_-^\#)$ which are positive with respect to $\cB_+$ and $\cB_-$ respectively. These choices allow us to put $\widehat{g} = \widehat{\varepsilon}^l \ (l=0,1)$ and $\widehat{h} = 1$. Moreover, by replacing $\phi'$ with $(\Ad \widehat{\varepsilon} \circ \phi')$ if necessary, we may assume that $\widehat{g} = 1$. Then, there exist $h \in \mathcal{A}(S_+^\#, T_+^\#)$ and $g \in \mathcal{A}(S_-^\#, T_-^\#)$ such that $\mu_\phi\circ(\Ad h)\circ \varphi_+^{-1}$ and $\mu_{\phi'}\circ(\Ad g)\circ \varphi_-^{-1}$ are the Harish-Chandra parameters of $\pi$ and $\theta_\psi(\pi, W)$ respectively. Consider the following diagram.
\begin{align}\label{regular}
\xymatrix{
X^*(T_-^\#) \ar[rr]^-{((\Ad g_0)\circ \varphi_-^{-1})^*} & & X^*(S_-) \ar[rr]^-{(\varphi_-\circ(\Ad g^{-1}))^*} && X^*(T_-^\#) \\
X^*(T_+^\#) \ar[u]^{I_\xi}\ar[rr]_-{((\Ad h_0)\circ \varphi_+^{-1})^*} & & X^*(S_+) \ar[u]_{\xi_\bullet^{\epsilon_\psi}} \ar[rr]_-{(\varphi_+\circ(\Ad h^{-1}))^*} & & X^*(T_+^\#) \ar[u]_{I_\xi}  
}
\end{align}
Then the image of $\mu_{\phi} \in X^*(T_+^\#)$ in $X^*(T_-^\#)$ is independent from the choices of the routes. Since $\mu_\phi$ and $\mu_{\phi'}$ are regular, we have the diagram \eqref{regular} is commutative.
Hence, the following diagram is also commutative.
\[
\xymatrix{
S_-^\# \ar[rr]^-{\Ad g_0^{-1}g } \ar[d]_{\fp_\xi} & & S_-^\# \ar[d]^{\fp_\xi} \\
S_+ \ar[rr]_-{\Ad h_0^{-1}h}   & &  S_+ 
}
 \]
By the formulation of the Harish-Chandra parameter in this paper, there exists $\gamma \in \fS_n$ such that
\[
((\Ad g)\circ \varphi_-^{-1})^*(a_1, \ldots, a_n) = (\underline{\epsilon}\cdot\gamma)\cdot (a_1, \ldots, a_n)
\]  
for $a_1, \ldots, a_n \in \Z$. Hence, the conditions of Proposition \ref{SOSP_key_dia_quat} imply that
\begin{align*}
(\Ad g)^*(a_1, \ldots, a_n) &= (\varphi_-)^*((\underline{\epsilon}\cdot\gamma)\cdot (a_1, \ldots, a_n)) \\
&= (\underline{\epsilon}\cdot \rho)^{-1}\cdot (\underline{\epsilon}\cdot\gamma)\cdot(a_1, \ldots, a_n) \\
&=\rho^{-1}\gamma\cdot(a_1, \ldots, a_n)
\end{align*}
for $a_1, \ldots, a_n \in \Z$. This shows that $h_0^{-1}h$ and $g_0^{-1}g$ satisfy the conditions of Lemma \ref{key_Gal}. Hence, we have
\begin{align*}
\fp_\xi(\inv_{z_-}(g, g_0)(w)) & = \fp_\xi(g_0^{-1}g \cdot z_-(w)\cdot w(g^{-1}g_0)) \\
&= \fp_\xi((\Ad g_0^{-1}g)(z_-(w)) \cdot (g_0^{-1}g) w(g^{-1}g_0)) \\
&= (\Ad h_0^{-1}h)(z_+(w)^{-1}) \cdot (h_0^{-1}h)w(h^{-1}h_0) \\
&= \inv_{z_+^{-1}}(h, h_0)(w)
\end{align*}
for $w \in \cW$. Therefore, we have
\begin{align*}
\iota_\phi[\fw_-, z_-, \varphi_-](\pi)(I_\xi(s)) &= \langle \inv_{z_-}(g, g_0), (\Ad g_0)^\wedge(I_\xi(s)) \rangle \\
&= \langle \fp_\xi(\inv_{z_-}(g, g_0)), (\Ad h_0)^\wedge(s) \rangle \\
&= \langle \inv_{z_+^{-1}}(h, h_0), (\Ad h_0)^\wedge(s) \rangle\\
&= \overline{\iota_{\xi \circ \phi}[\fw_+, z_+, \varphi_+](\theta_\psi(\pi, W))(s)}.
\end{align*}
Thus, we have $\theta_\psi(\pi, W) = \mathscr{T}_\psi(\pi)$.

Then, by Lemma \ref{countingT} and Remark \ref{countingP}, we have that $\theta_\psi(\pi, W) \not=0$ if and anly if $\mathscr{T}_\psi(\pi) \not=0$, which completes the proof of Theorem \ref{sp-o real}.
\end{proof}

\subsection{The proof of Proposition \ref{SOSP_key_dia_quat} in the case \eqref{vsp I}} \label{prf_I}

Assume that $e_\H =1$.  Put 
\begin{align*}
&e_{11} = \frac{1}{2}(1+ j), &e_{12} = \frac{1}{2}(i - ij),  \\ 
&e_{21} = \frac{1}{2}(- i - ij), & e_{22} = \frac{1}{2}(1 - j).
\end{align*}
Consider the isomorphisms $A_+\colon V_c^\#\otimes\C \rightarrow V_{m,0}\otimes\C$  given by
\[
A_+(e_k^\#) = e_ke_{11}, \ A_+(e_{2m + 1-k}^\#) = e_ke_{21} 
\]
for $k=1, \ldots, m$ and $A_-\colon W_c^\# \otimes\C\rightarrow W_{p,q}\otimes\C$  given by the composition
\[
\xymatrix{W_c^\#\otimes\C\ar[r]^-P& W_{\sim}^\#\otimes\C \ar[r]^-{A_\sim}& W\otimes\C}
\]
where $A_\sim$ is the isometry defined by
\begin{align*}
&A_\sim(f_{2k-1}^\#) = \begin{cases} \sqrt{-1}e_{11} f_k & (k \in I), \\  e_{11}f_k & (k \not\in I)\end{cases} \\
&A_\sim(f_{2k}^\#) = \begin{cases} \sqrt{-1}e_{12}f_k & (k \in I), \\  e_{12}f_k&  (k \not\in I).
\end{cases}
\end{align*}
where  
\[
I = \{ k = 1, \ldots, m \mid (2n+3 - 4k)\cdot (2p + 1 - 2k) > 0 \}.
\]
Moreover, put $z_{0+} = 1 \in Z^1(u \rightarrow \cW, Z \rightarrow G(V_c^\#))$ and denote by $z_{0-} \in Z^1(u\rightarrow \cW, Z \rightarrow G_0(W_c^\#))$ the cocycle satisfying 
\[
z_{0-}(w_1) = \varsigma_-^\#(\eta_1, \ldots, \eta_m)
\]
where $\eta_k = -1$ if $k \in I$ and $\eta_k = 1$ if $k \not\in I$. 
\begin{lem}
We have $(z_{0+},  \fm_V^{^-1}\circ\varphi_{A_+}) \in \RIT^\star(V^\#, V), (z_{0-}, \fm_W^{-1}\circ \varphi_{A_-}) \in \RIT^\star(W^\#, W)$, and $(z_{0+},  \fm_V^{^-1}\circ\varphi_{A_+})\leftrightarrow (z_{0-}, \fm_W^{-1}\circ \varphi_{A_-})$. 
\end{lem}

Put
\[ 
\underline{\epsilon}_\bullet = (\sqrt{-1}\epsilon_\psi\epsilon_1, \ldots, \sqrt{-1}\epsilon_\psi\epsilon_n) \in \{\pm 1\}^n.
\]
Then, there exists $\rho_\bullet \in \fS_n$ so that
\[
\xi_\bullet^{\epsilon_\psi}(a_1, \ldots, a_m) = (\underline{\epsilon}_\bullet\cdot \rho_\bullet)\cdot (a_1, \ldots, a_n)
\]
for $a_1, \ldots, a_m \in \Z$. Here, we put $a_n = 0$ if $n = m + 1$. Choose $\rho_2 \in \fS_n$ so that $ \rho_\bullet^{-1}\cdot\rho_2\cdot\rho_1(n) = n$, and choose $g_2 \in N(S_-, G_0(W))$ so that the action of $\Ad g_2$ on $S_-$ coincides with $\rho_2$. Then, putting $\varphi_\sim = (\Ad g_2) \circ \varphi_{A_\sim}$ and $\rho_3 = \rho_\bullet^{-1}\cdot\rho_2\cdot \rho_1 \in \fS_m$, we have 
\begin{align*}
(\varphi_{\sim}^{-1})^* ((u\cdot \rho_1)\cdot (a_1, \ldots, a_n)
&= (\rho_2\cdot u\cdot \rho_1)\cdot (a_1, \ldots, a_n)) \\
&= (\underline{\epsilon}_\bullet \cdot \rho_\bullet \cdot \rho_\bullet^{-1} \cdot \underline{\epsilon}_\bullet \cdot \rho_2 \cdot u \cdot \rho_1)\cdot (a_1, \ldots, a_n) \\
&=\xi_\bullet^{\epsilon_\psi}((\rho_3 \cdot u')\cdot (a_1, \ldots, a_m))
\end{align*}
for $a_1, \ldots, a_m \in \Z$ and for certain $u ' \in \{\pm 1\}^m$. Here we put $a_n = 0$ if $n = m+1$. Let $h_3$ be an element of $G_0(V)(\C)$ so that the action of $\Ad h_3$ on $S_+$ coincides with $\rho_3\cdot u'$. If we put $u' = (\mu_1, \ldots, \mu_m)$, then we have $h_3^{-1} \sigma(h_3) = \varsigma_+(\mu_1, \ldots, \mu_m)$. Put $\varphi_+ = (\Ad h_3)\circ\varphi_{A_+}$, and denote by $z_+$ the rigid inner form such that $z_+(w_1) = \varsigma_+^\#(\mu_1, \ldots, \mu_m)$. Moreover, we define $z_- \in Z^1(u \rightarrow \cW, Z \rightarrow G_0(W_c^\#))$ by $z_-(w) = \varphi_{A_-}^{-1}(g_2^{-1}w(g_2)) z_{0-}(w)$ for $w \in \cW$. Then we have $(z_+, \varphi_+) \in \RIT^\star(V^\#, V)$, $(z_-, \varphi_-) \in \RIT^\star(W^\#, W)$, and $(z_+, \varphi_+) \leftrightarrow (z_-, \varphi_-)$. 

\begin{lem}
We have, $\fp_\xi(z_-(w_1)) = z_+(w_1)^{-1}$. 
\end{lem}

\begin{proof}
By the construction of $u'$, if we write 
\[
(\rho_1^{-1}\rho_2^{-1})(\underline{\epsilon}_\bullet)\cdot \rho_1^{-1}(u) = (\mu_1', \ldots, \mu_n'),
\]
then we have $\mu_k' = \mu_k$ for $k=1, \ldots, m$. Hence, by \eqref{fpxi}, we have
\begin{align*}
z_+(w_1) &= \fp_\xi((u\cdot \rho_1)\cdot [(\rho_1^{-1}\rho_2^{-1})(\varsigma_-^\#(\sqrt{-1}\epsilon_\psi\epsilon_1, \ldots, \sqrt{-1}\epsilon_\psi\epsilon_n)) \cdot \rho_1^{-1}(\varsigma_-^\#(u_1, \ldots, u_n))])\\
&= \fp_\xi((u\cdot \rho_2^{-1} ) (\varsigma_-^\#(\sqrt{-1}\epsilon_\psi\epsilon_1, \ldots, \sqrt{-1}\epsilon_\psi\epsilon_n)) \cdot \varsigma_-^\#(u_1, \ldots, u_n)) \\
&= \fp_\xi(\rho_2^{-1}\cdot (\varsigma_-^\#(\sqrt{-1}\epsilon_\psi\epsilon_1, \ldots, \sqrt{-1}\epsilon_\psi\epsilon_n)) \cdot \varsigma_-^\#(u_1, \ldots, u_n)).
\end{align*}
On the other hand, we have
\[
g_2^{-1}w_1(g_2) = \varsigma_-(\epsilon_1, \ldots, \epsilon_n) \cdot \varsigma_-(\epsilon_{\rho_2(1)}, \ldots, \epsilon_{\rho_2(n)})^{-1}.
\]
Hence, we have
\begin{align*}
z_+(w_1)^{-1} & =z_+(w_1) \\ &=  \fp_\xi(\varphi_{A_-}^{-1}(g_2^{-1}w_1(g_2)) \cdot \varsigma_-^\#(\sqrt{-1}\epsilon_\psi\epsilon_1, \ldots, \sqrt{-1}\epsilon_\psi\epsilon_n) \cdot \varsigma_-^\#(u_1, \ldots, u_n)) \\
&=\fp_\xi(\varphi_{A_-}^{-1}(g_2^{-1}w_1(g_2)) \cdot z_{0-}(w_1)) \\
&= \fp_\xi(z_-(w_1)).
\end{align*}
\end{proof}
Therefore, we have that $(z_+, \varphi_+)$ and $(z_-, \varphi_-)$ satisfy the all conditions of Proposition \ref{SOSP_key_dia_quat}.

\subsection{The proof of Proposition \ref{SOSP_key_dia_quat} in the case \eqref{vsp III}} \label{prf_III}
		
Put
\begin{align*}
&e_{11} = \frac{1}{2}(1-\sqrt{-1}j), &e_{12} = \frac{1}{2}(i +  \sqrt{-1} ij),  \\ 
&e_{21} = \frac{1}{2}(- i + \sqrt{-1} ij), & e_{22} = \frac{1}{2}(1 + \sqrt{-1}j).
\end{align*}
We may choose the isomorphism $\gamma\colon \rM_2(\C) \rightarrow \H\otimes\C$ given by
\[
\gamma(\begin{pmatrix} x & y \\ z & w \end{pmatrix}) = e_{11} x + e_{12}y + e_{21}z + e_{22}w
\]
for $x,y,z,w \in \C$. Define $A_+ \colon V_c^\#\otimes\C \rightarrow V_{p,q}\otimes\C$ by
\begin{align*}
&A_+(e_k^\#) = \begin{cases}e_ke_{11} & (1 \leq k \leq p), \\ e_ke_{21} & (p+1 \leq k \leq m)\end{cases} \\
&A_+(e_{2m + 1-k}^\#) = \begin{cases}e_ke_{21} & (1 \leq k \leq p), \\ e_ke_{11} & (p+1 \leq k \leq m).
\end{cases}
\end{align*}
Denote by $z_{0+}$ the unique cocycle in $Z^1(u \rightarrow \cW, Z \rightarrow G(V_c^\#))$ satisfying
\[
z_{0+}(w_1) = \varsigma_+^\#(\epsilon_1 \sqrt{-1}, \ldots, \epsilon_m \sqrt{-1})
\]
where $\epsilon_k = 1$ if $1 \leq k \leq p$ and $\epsilon_k = -1$ if $p+1 \leq k \leq m$. On the other hand, we also define $A_-\colon W^\#\otimes\C\rightarrow W\otimes \C$ by the composition
\[
\xymatrix{W^\#\otimes\C\ar[r]^-P & W_{\sim}^\#\otimes\C \ar[r]^-{A_\sim}& W\otimes\C}
\]
where $A_\sim$ is the isometry defined by
\begin{align*}
&A_\sim(f_{2k-1}^\#) 
=\begin{cases} e_{12}\cdot j\cdot f_k & (1 \leq k \leq t), \\ e_{12}  \cdot f_k & (t+1 \leq k \leq n) \end{cases} \\
&A_\sim(f_{2k}^\#)
= \begin{cases} e_{11}\cdot j \cdot f_k  & (1 \leq k \leq t), \\  e_{11}  \cdot f_k & (t+1 \leq k \leq n) \end{cases}
\end{align*}
Denote by $z_{0-}$  the unique cocycle in $Z^1(u \rightarrow \cW, Z \rightarrow G(V_c^\#))$ satisfying
\[
z_{0-}(w_1) = \varsigma_-^\#(-\sqrt{-1}, \ldots, -\sqrt{-1}).
\]
Then, we have the following lemma.

\begin{lem}\label{vsp+inn}
\begin{enumerate}
\item The linear map $A_+$ induces the isometry from $V^\#\otimes\C$ onto $(V\otimes\C)^\natural$. Moreover, we have
\[
w_1(A_+(x)) = A_+(z_{0+}(w_1) \cdot w_1(v))\cdot i^{-1}
\]
for $x \in V_c^\#$. \label{inn1}
\item The linear map $A_-$ induces the isometry from $W^\#\otimes\C$ onto $(W\otimes\C)^\natural$. Moreover, we have
\[
w_1(A_-(y)) = i\cdot A_-(w_1(y)\cdot z_{0-}(w_1)^{-1})
\]
for $y \in W^\#$. \label{inn2}
\end{enumerate}
\end{lem}

\begin{proof}
Since $w_1(e_{11}) = e_{21}\cdot i$ and $w_1(e_{21}) = e_{11}\cdot (-i)$, we have
\[
(w_1(A_+(e_k^\#)), w_1(A_+(e_{2m+1-k}^\#)) =  (A_+(e_k^\#)\cdot i^{-1}, A_+(e_{2m+1-k}^\#)\cdot i^{-1}) \cdot \begin{pmatrix} & \epsilon_k \\ -\epsilon_k & \end{pmatrix}
\]
for $1 \leq k \leq m$. Hence, we have the assertion \eqref{inn1}. Similarly, since $w_1(e_{11}) = -i\cdot e_{12}$ and $w_1(e_{12}) = i\cdot e_{11}$, we have
\[
\begin{pmatrix} w_1(A_\sim(f_{2k-1}^\#)) \\ w_1(A_\sim(f_{2k}^\#)) \end{pmatrix} = \begin{pmatrix} & -1 \\ 1 & \end{pmatrix}^{-1} \begin{pmatrix} i\cdot A_\sim(f_{2k-1}^\#) \\ i\cdot A_\sim(f_{2k}^\#) \end{pmatrix} 
\]
for $k = 1, \ldots, n$. Hence, we have the assertion \eqref{inn2}.
\end{proof}

 Then, by Lemma \ref{vsp+inn} we have the following.
\begin{cor}
We have $(z_{0+}, \varphi_{A_+}) \leftrightarrow (z_{0-}, \varphi_{A_-})$.
\end{cor}

\begin{proof}
Define $\Omega$ by the composition
\[
\xymatrix{
W^\#\otimes\C \ar[rr]^-{A_+\otimes A_-} && (V\otimes\C)^\natural\otimes(W\otimes\C)^\natural \ar[r] & \W\otimes\C.
}
\]
Then, it is an isometry and the following diagram is commutative
\[
\xymatrix{
\Sp(\W^\#) \ar[rr]^-{\varphi_\Omega} & & \Sp(\W) \\ G(V^\#) \times G(W^\#) \ar[u]_{\iota^\#} \ar[rr]_-{(\varphi_{A_+}, \varphi_{A_-})} & & G(V)\times G(W) \ar[u]_\iota
}
\]
Finally, we have
\begin{align*}
w(\Omega(x\otimes y)) &= \Omega(z_+(w) w(x)\otimes w(y) z_-(w)^{-1}) \\
&= \Omega(z_{0+}(w)^{-1}w(x)\otimes w(y) z_{0-}(w)) \\
&=[\Omega\circ\iota(z_{0+}(w), z_{0-}(w))\circ w](x\otimes y)
\end{align*}
for $x \in V^\#\otimes\C$, $y \in W^\#\otimes\C$, and $w \in \cW$. Thus we have
\[
\Omega^{-1}\circ w\circ \Omega \circ w^{-1} = \iota(z_{0+}(w), z_{0-}(w)),
\]
which proves the corollary.
\end{proof}  

Put 
\[
e_0 = (-\sqrt{-1}\epsilon_\psi, \ldots, -\sqrt{-1}\epsilon_\psi) \in \{\pm 1\}^n.
\]
Then, one observes that 
\[
(\varphi_{A_-}^{-1})^*((u \cdot \rho_1)\cdot (a_1, \ldots, a_n)) = (e_0 \cdot  \rho_1)\cdot (a_1, \ldots, a_n)
\]
for $a_1, \ldots, a_n \in \Z$. Here, $u_1, \ldots, u_n$ are defined in \eqref{u_k}. Take $\rho_\bullet \in \fS_n$ so that
\[
\xi_\bullet^{\epsilon_\psi}((a_1, \ldots, a_m)) = (\underline{\epsilon}\cdot \rho_\bullet)\cdot (a_1, \ldots, a_n)
\]
for $a_1, \ldots, a_m \in \Z$. Here, we put $a_n = 0$ if $n = m + 1$. Put $\rho_3 = \rho_\bullet \rho_1^{-1}$ and choose $g_3 \in N(S_-, G(W))$ so that the action of $\Ad g_3$ on $S_-$ coincides with that of $\underline{\epsilon}\cdot \rho_3\cdot u\cdot e_0$. Then, putting $\varphi_- = (\Ad g_3)\circ \varphi_{A_-}$, we have
\[
(\varphi_-^{-1})^*(b_1, \ldots, b_n) = (\underline{\epsilon}\cdot \rho_3)\cdot (b_1, \ldots, b_n)
\]
for $b_1, \ldots, b_n \in \Z$. Moreover, we have
\begin{align*}
(\varphi_-^{-1})^*((u\cdot \rho_1)\cdot (a_1, \ldots, a_n)) &= (\underline{\epsilon}\cdot\rho_3\cdot u \cdot \rho_1)\cdot (a_1, \ldots, a_n) \\ 
&=(\underline{\epsilon}\cdot \rho_\bullet)\cdot [\rho_1^{-1}(u)\cdot (a_1, \ldots, a_n)]. 
\end{align*}
Take $u' \in \{\pm 1\}^m$ so that
\[
\xi_\bullet^{\epsilon_\psi}(u'\cdot(a_1, \ldots, a_m)) = (\underline{\epsilon}\cdot \rho_\bullet)\cdot[\rho_1^{-1}(u)\cdot(a_1, \ldots, a_n)]
\] 
for $a_1, \ldots, a_m \in \Z$. Here we put $a_n = 0$ if $n = m + 1$. Let $h_2$ be an element of $N(S_+, G(V))$ so that the action of $\Ad h_2$ on $S_+$ coincides with that of $u' \cdot (\epsilon_1, \ldots, \epsilon_m)$. Put $z_+ = z_{0+}$ and $\varphi_+ = (\Ad h_2)\circ \varphi_{A_+}$. Moreover, we define $z_- \in Z^1(u \rightarrow \cW, Z \rightarrow G_0(W^\#))$ by $z_-(w) = \varphi_-^{-1}(g_3^{-1}w(g_3)) z_{0-}(w)$ for $w \in \cW$. Then, we have $(z_+, \varphi_+) \in \RIT^\star(V^\#, V)$, $(z_-, \varphi_-) \in \RIT^\star(W^\#, W)$ and $(z_+, \varphi_+)\leftrightarrow (z_-, \varphi_-)$.

\begin{lem}
We have, $\fp_\xi(z_-(w_1)) = z_+(w_1)^{-1}$. 
\end{lem}


\begin{proof}
Since $\underline{\epsilon}\cdot \rho_3\cdot u \cdot e_0 = \rho_3\cdot (\rho_1\cdot\rho_\bullet^{-1})(\underline{\epsilon})\cdot u \cdot e_0$ and 
\[
\rho_\bullet^{-1}(\underline{\epsilon}) = (-\sqrt{-1}\epsilon_\psi\epsilon_1, \ldots, -\sqrt{-1}\epsilon_\psi\epsilon_n),
\]
we have
\[
g_3^{-1}w_1(g_3) =  \varsigma_-( u_1\epsilon_{\rho_1(1)}, \ldots, u_n\epsilon_{\rho_1(n)}).
\]
Observe that
\[
z_{0-}(w_1) \cdot \varsigma_-^\#(u_1, \ldots, u_n) = (u\cdot \rho_1)\cdot \varsigma_-^\#(-\sqrt{-1}, \ldots, -\sqrt{-1})
\]
and 
\begin{align*}
\varsigma_-^\#(\epsilon_{\rho_1(1)}, \ldots, \epsilon_{\rho_1(n)}) 
= (u\cdot\rho_1)\cdot\varsigma_-^\#(\epsilon_1, \ldots, \epsilon_n).
\end{align*}
Hence, we have
\begin{align*}
z_-(w_1) &= z_{0-}(w_1) \cdot \varphi_{A_-}^{-1}(g_3^{-1}w_1(g_3)) \\
&=(u\cdot \rho_1)\cdot \varsigma_-^\#(-\epsilon_1\sqrt{-1}, \ldots, -\epsilon_n\sqrt{-1}).
\end{align*}
According to \eqref{fpxi}, this implies
\begin{align*}
\fp_\xi(z_-(w_1)) &= \varsigma_+^\#(-\epsilon_1\sqrt{-1}, \ldots, -\epsilon_m\sqrt{-1}) \\
&= z_+(w_1)^{-1}.
\end{align*}
Therefore, we have that $(z_+, \varphi_+)$ and $(z_-, \varphi_-)$ satisfy the all conditions of Proposition \ref{SOSP_key_dia_quat}.
\end{proof}


\section{
		The case \eqref{vsp III} with $m=n=1$ 
		}\label{nonArch_m=n=1}

Let $F$ be a non-Archimedean local field, and let $D$ be the unique division quaternion algebra over $F$. When $m = n = 1$, Ikematsu \cite{Ike19} has described the local theta correspondence in terms of character relations. In this section, we verify the conjecture \ref{Pre_PC} is consistent with it.

\subsection{Settings}

Assume that $V = D$ with the Hermitian form $(x, y) = x^*y$ for $x, y \in D$, and $W=D$ with the skew-Hermitian form $\langle x, y \rangle = x \alpha_0 y^*$ where $\alpha_0$ is a non-zero trace-zero element of $D$. We put $\alpha = -\alpha_0/2$. Then, it is known that there exists a non-zero trace-zero element $\beta \in D$ such that $\alpha\beta + \beta\alpha = 0$. Then $1, \alpha, \beta, \alpha\beta$ consist a basis of $D$ over $F$. We put $a = \alpha^2, b = \beta^2 \in F^\times$, and put $E = F(\sqrt{a})$. Then, $W_c^\#$ is the $F$-vector space $F_2$ of row vectors of degree $2$ equipped with the quadratic form $(x,y) \mapsto 2cx^2 - 2ac y^2$. We may assume that $c=1$, and we write $V^\#$ (resp. $W^\#$) instead of $V_1^\#$ (resp. $W_1^\#$). 
We fix the identification $\gamma\colon \rM_2(\overline{F}) \rightarrow D \otimes \overline{F}$ by
\begin{align*}
&\gamma(e_{11}) = \frac{1}{2b}(b + \sqrt{b}\cdot \beta), &\gamma(e_{12}) = \frac{1}{2b}(b\cdot \alpha - \sqrt{b}\cdot \alpha\beta), \\ &\gamma(e_{21}) = \frac{1}{2ab}(b\cdot \alpha + \sqrt{b}\cdot\alpha\beta), &\gamma(e_{22}) = \frac{1}{2b}(b - \sqrt{b}\cdot\beta).
\end{align*}

\begin{lem}\label{na-lem1}
Fix $w_1 \in \cW$ so that its image in $\Gamma$ is not contained in $\Gamma_{F(\sqrt{b})}$. Then, there exists $z \in Z^1(u \rightarrow \cW, \{\pm1\} \rightarrow E^1)$ such that
\[
z(w_1) = \begin{pmatrix} 0 & -\sqrt{-a}^{-1} \\ \sqrt{-a} & 0 \end{pmatrix}.
\]
\end{lem}

\begin{proof}
Let $\rho$ be the non-trivial homomorphism of $\Gamma/\Gamma_{F(\sqrt{b})}$ onto $Z$. We define a cocycle $c_1 \in Z^1(\Gamma/\Gamma_{F(\sqrt{b})}, E^1/\{\pm1\})$ by 
\[
c_1(\sigma^k) = \begin{pmatrix} 0 & -\sqrt{-a}^{-1} \\ \sqrt{-a} & 0 \end{pmatrix}^k \mod \{\pm1\}
\]
for $k=0,1$. Then, there exists $z_1 \in Z^1(u \rightarrow \cW, \{\pm1\} \rightarrow E^1)$ whose image in $Z^1(\Gamma, E^1/\{\pm1\})$ coincides with $c_1$. Then, the image of $z_1\cdot \rho \in Z^1(u\rightarrow \cW, \{\pm 1\} \rightarrow E^1)$ is also $c_1$. Moreover, we have $(z_1 \cdot \rho)(w_1) = -z_1(w_1)$, which proves the lemma.
\end{proof}

We define $B_+\colon V^\#\otimes\overline{F} \rightarrow (V\otimes\overline{F})^\natural$ by $B_+(e_1) = e_{11}$ and $B_+(e_2) = e_{21}$. We also define $B_-\colon W^\#\otimes\overline{F} \rightarrow (W\otimes\overline{F})^\natural$ by $B_-(f_1)= e_{11}$ and $B_-(f_2) = e_{12}$. Consider the four embeddings $\varsigma_+^\# \colon E^1 \rightarrow G(V^\#), \varsigma_-^\# \colon E^1 \rightarrow G(W^\#), \varsigma_+ \colon E^1 \rightarrow G(V)$, and $\varsigma_-\colon E^1 \rightarrow G(W)$ given by
\begin{align*}
&\varsigma_+^\#(x + \sqrt{a} y) =\varsigma_-^\#(x + \sqrt{a}y) = \begin{pmatrix} x & y \\ ay & x \end{pmatrix},\\
&\varsigma_+(x + \sqrt{a}y) = \varsigma_-(x + \sqrt{a}y) = x + y \alpha 
\end{align*}
for $x,y \in \overline{F}$ satisfying $x^2 - ay^2 = 1$. Put $S_+ = {\rm Im} \iota_+$ and $S_- = {\rm Im} \iota_-$. Note that in \cite{Ike19}, $E^1$ is identified with the maximal torus $S_+$ (resp. $S_-$) of $G(V)$ (resp. $G(W)$) by the embedding $\gamma \mapsto \varsigma_+(\gamma)^{-1}$ (resp. $\gamma \mapsto \varsigma_-(\gamma)^{-1}$) for $\gamma \in E^1$. We put $z_+ = \varsigma_+^\#\circ z \in Z^1(u \rightarrow \cW, Z \rightarrow G(V^\#))$ and $z_- = \varsigma_-^\#\circ z \in Z^1(u\rightarrow \cW, Z \rightarrow G_0(W^\#))$. 

\begin{lem}\label{lift_na}
Take $w_1 \in \cW$ as in Lemma \ref{na-lem1}. Then, we have
\begin{align*}
w_1(B_+(x)) &= B_+(z_+(w_1)x)\cdot (\alpha\sqrt{-a}^{-1}) \\
w_1(B_-(y)) &= (-\alpha\sqrt{-a}^{-1})\cdot B_-(y \cdot z_-(w_1)^{-1})
\end{align*}
for $x \in V^\#$ and $y \in W^\#$.
\end{lem}

\begin{proof}
We have
\begin{align*}
(w_1(B_+(e_1)), w_1(B_+(e_2))) &= (e_{22}, e_{12}\cdot a^{-1}) =(e_{21}\alpha, e_{11}\alpha a^{-1}) \\
&=(e_{11}\cdot \alpha\sqrt{-a}^{-1}, e_{21}\cdot \alpha\sqrt{-a}^{-1})\cdot z_+(w_1),
\end{align*}
which implies the first assertion of Lemma \ref{lift_na}. Besides, we have
\[
\begin{pmatrix}w_1(B_-(f_1)) \\ w_1(B_-(f_2))\end{pmatrix} = \begin{pmatrix} e_{22} \\ a e_{21} \end{pmatrix} = \begin{pmatrix} \alpha a^{-1} e_{12} \\ \alpha e_{11} \end{pmatrix} = z_-(w_1)^{-1} \begin{pmatrix} (-\alpha\sqrt{-a}^{-1}) e_{11} \\ (-\alpha\sqrt{-1}^{-1}) e_{12}\end{pmatrix},
\]
which implies the second assertion of Lemma \ref{lift_na}. Hence, we complete the proof.
\end{proof}

Now, we put $\varphi_+ = \varphi_{B_+}$ and $\varphi_- = \varphi_{B_-}$. 

\begin{cor}
We have $(z_+, \varphi_{+}) \leftrightarrow (z_-, \varphi_{-})$.
\end{cor}

\begin{proof}
Put $\Omega = B_+\otimes B_- \colon \W^\# \rightarrow \W^\natural = \W$. Then it is obvious that the diagram \eqref{corr_rig} is commutative. Moreover, we have 
\begin{align*}
(\Omega^{-1}\circ w \circ \Omega \circ w^{-1})(x\otimes y) &= \Omega^{-1}(\Omega(z_+(w)^{-1}\sigma(\sigma^{-1}(x \otimes y))z_-(w))) \\
&= z_+(w)^{-1}x\otimes yz_-(w) 
  \end{align*}
for $w \in \cW$, $x \in V^\#$, and $y \in W^\#$. This implies that $\Omega^{-1}\circ w \circ \Omega \circ w^{-1} = \iota^\#(z_+(w), z_-(w))$. 
\end{proof}

\subsection{Descriptions of the local theta correspondence}

In \cite{Ike19} the local theta correspondence in this case is described as follows. 
Let $\eta$ be an irreducible representation of $G_0(W)(F)$, which is a character since $G_0(W)(F)$ is Abelian. We denote by $\phi'$ the L-parameter of $\eta$, and by $\phi$ the L-parameter of $G(V)$ given by the composition $\xi\circ\phi'$. Then, it is known that 
\[
\wPi_\phi(G(V)) = 
\begin{cases}
\varnothing & (\eta = 1), \\
\{\tau_+\} & (\eta \not=1, \eta^2 =1), \\
\{ \tau_+, \tau_- \} & (\eta^2 \not=1).
\end{cases}
\]
Here, in the case $\eta^2 \not=1$, the representations $\tau_+$ and $\tau_-$ are specified by the character relation
\begin{align}\label{tau_pm}
&\Tr_{\tau_+}(\delta) - \Tr_{\tau_-}(\delta) \\
&=\lambda(E/F, \psi)\cdot \omega_{E/F}(\frac{\delta^{-1} - \delta}{\alpha_0})\cdot \frac{(\eta\circ\varsigma_-\circ\varsigma_+^{-1})(\delta) - (\eta\circ\varsigma_-\circ\varsigma_+^{-1})(\delta)^{-1}}{|\gamma - \gamma^{-1}|_E^{1/2}} \notag
\end{align}
for $\delta \in S_+$.

\begin{fact}\label{Ike_theta_corr}
We have
\[
\theta_\psi(\eta, V) = \begin{cases} 0 & (\eta = 1) \\ \tau_+ & (\eta \not=1).\end{cases} 
\]
\end{fact}
To obtain the Langlands parameter of $\tau_+$, we compute the geometric transfer factor $\Delta'[\fw_+, z_+, \varphi_{+}](-, -)$ which will be abbreviated into $\Delta'(-, -)$. We may assume that the endoscopic data is $\cE(\dot{s})$ where $\dot{s}$ is a pre-image in $\overline{G(V)}^\wedge$ of $-1 \in G(V)^\wedge$.  Let $\gamma \in E^1$ and let $\delta \in G(V)(F)$ so that $\gamma \not=\pm1$ and $\gamma$ is a norm of $\delta$. Then, we have
\begin{align}\label{formula_transfer_na}
\Delta'(\gamma, \delta) &= \lambda(E/F, \psi)\cdot \omega_{E/F}(\frac{\gamma^{-1} - \gamma}{-2\sqrt{a}})\cdot|\gamma - \gamma^{-1}|_E^{-\frac{1}{2}}\\ 
& \notag \qquad \times \la \inv_{z_+}(u_+^\#(\gamma), \delta), \widehat{u_+}^{-1}(\dot{s}) \ra.
\end{align}
Before computing $\iota[\fw_+, z_+, \varphi_{+}](\tau_-)(\xi(\dot{s}))$, we observe a property of $\Delta'(-,-)$.

\begin{lem}\label{tr_fac_minus}
Let $\gamma \in E^1$. Then, we have 
\[
\Delta'(\gamma^{-1}, \delta) = - \Delta'(\gamma, \delta).
\]
\end{lem}

\begin{proof}
By the expression \eqref{formula_transfer_na}, we have
\begin{align}\label{n=m=1, quot}
\frac{\Delta'(\gamma^{-1}, \delta)}{\Delta'(\gamma, \delta)} = \omega_{E/F}(-1) \cdot \la \frac{\inv_{z_+}(u_+^\#(\gamma^{-1}), \delta)}{\inv_{z_+}(u_+^\#(\gamma), \delta)}, \widehat{u_+}^{-1}(\dot{s}) \ra.
\end{align}
Put $z_+' = \inv_{z_+}(u_+^\#(\gamma), \delta)$. Then we have $z_+'(w) = \pm z_+(w)$ for $w \in \cW$. Let $g \in G(V^\#)(\overline{F})$ satisfying $\varphi_{+}(g u_+^\#(\gamma) g^{-1}) = \delta$. Take $h \in G(V^\#)(\overline{F})$ so that $\varphi_{+}(h) = \sqrt{-b}^{-1}\beta \in G(V)(\overline{F})$. Then, we have
\[
\varphi_+(ghu_+^\#(\gamma^{-1})h^{-1}g^{-1}) = \delta.
\]
Thus, we have
\begin{align*}
\varphi_+(z_+'(w) w(h) z_+'(w)^{-1}) &= w(\sqrt{-b}^{-1}\beta) \\
&=\sqrt{-b}^{1-w}\cdot \varphi_+(h),
\end{align*}
which means that 
\[
\inv_{z_+}(u_+^\#(\gamma^{-1}), \delta)(w) = \sqrt{-b}^{1-w}\cdot z_+'(w)
\]
for $w \in \cW$. The image of the cocycle $w \mapsto \sqrt{-b}^{1-w}$ is trivial in $H^1(\Gamma, E^1)$ if and only if $-b \in N_{E/F}(E^\times)$. Since the image of $\widehat{u_+}^{-1}(\dot{s})$ in $H^0(\Gamma, \widehat{E^1}) = \{\pm 1\}$ is $-1$, we have
the Tate-Nakayama pairing term in \eqref{n=m=1, quot} coincides with $\omega_{E/F}(-b)$. Hence we have
\[
\frac{\Delta'(\gamma^{-1}, \delta)}{\Delta'(\gamma, \delta)} = \omega_{E/F}(b) = -1,
\]
which proves Lemma \ref{tr_fac_minus}.
\end{proof}

\begin{thm}\label{m=n=1_conj}
Assume that $F$ is non-Archimedean. Then, Conjecture \ref{conj_main} holds in the case \eqref{vsp III} with $\epsilon = 1$,  $n=m=1$.
\end{thm}

\begin{proof}
The numbers $\iota_\phi[\fw_+, z_+, \varphi_+](\tau_\pm)(\xi(\dot{s}))$ are characterized as coefficients of the spectral decomposition of the stable distribution $f \mapsto e(G(V)) {\rm Tr}_{\eta\circ\varsigma_-}(f^{\phi, \cE(\dot{s})})$, which is computed as follows. Since the Kottwitz's sign $e(G(V))$ of $G(V)$ is $-1$, we have
\begin{align*}
&e(G(V))\cdot{\rm Tr}_{\eta\circ\varsigma_-}(f^{\phi, \cE(\dot{s})}) \\
&= -\int_{E^1}(\Delta'(\gamma, \varsigma_+(\gamma))O_{\varsigma_+(\gamma)}(f) + \Delta'(\gamma, \varsigma_+(\gamma)^{-1})O_{\varsigma_+(\gamma)^{-1}}(f)) \cdot (\eta\circ\varsigma_-)(\gamma) \: d\gamma \\
&= -\int_{E^1}(\Delta'(\gamma, \varsigma_+(\gamma)) (\eta\circ\varsigma_-)(\gamma) + \Delta'(\gamma^{-1}, \varsigma_+(\gamma))(\eta\circ\varsigma_-)(\gamma^{-1})) \cdot O_{\varsigma_+(\gamma)}(f) \: d\gamma \\
&= -\la z_+, \widehat{\varsigma_+}^{-1}(\xi(\dot{s})) \ra \int_{E^1} \lambda(E/F, \psi)\omega_{E/F}(\frac{\gamma^{-1} - \gamma}{-2\sqrt{a}})\frac{(\eta\circ\varsigma_-)(\gamma) - (\eta\circ\varsigma_-)(\gamma^{-1})}{|\gamma - \gamma^{-1}|^{1/2}}\cdot O_{\varsigma_+(\gamma)}(f) \: d\gamma \\
&=-\la z_+, \widehat{\varsigma_+}^{-1}(\xi(\dot{s}))\ra \int_{S_+}\lambda(E/F, \psi)\omega_{E/F}(\frac{\delta^{-1} - \delta}{\alpha_0})\frac{(\eta\circ\varsigma_-\circ\varsigma_+^{-1})(\delta) - (\eta\circ\varsigma_-\circ\varsigma_+^{-1})(\delta^{-1})}{|\gamma - \gamma^{-1}|^{1/2}}\cdot O_\delta(f) \: d\delta \\
&=\la z_+, \widehat{\varsigma_+}^{-1}(\xi(\dot{s})) \ra \int_{G(V)(F)} f(g)(\Tr_{\tau_-}(g) - \Tr_{\tau_+}(g)) \: dg
\end{align*}
for $f \in C_c(G(V)(F))$. Hence we have
\begin{align*}
\iota_\phi[\fw_+, z_+, \varphi_+](\tau_\pm)(\xi(\dot{s})) &= \mp \la z_+, \widehat{\varsigma_+}^{-1}(\xi(\dot{s})) \ra \\
&=\mp \la z_-, \widehat{\varsigma_-}^{-1}(\dot{s}) \ra \\
&=\mp \iota_\phi[\fw_-, z_-, \varphi_-](\eta)(\dot{s}).
\end{align*}
This proves Theorem \ref{m=n=1_conj}.
\end{proof}


\section{Appendix: Centers of even Spin groups} \label{cospin}

In this appendix, we prove an elementary result that describes the action of certain outer automorphism on the centers of complex even Spin groups. Let $X$ be a complex vector space of $\dim X =2r$, let $( \ , \ )\colon X\times X \rightarrow \C$ be a non-degenerate symmetric bilinear form over $\C$, let $x_1, \ldots, x_r, y_r, \ldots, y_1$ be a basis of $X$ so that the  representation matrix of $( \ , \ )$ is the anti-diagonal matrix $J_{2r}$, let
\[
{\rm Cl}(X) = {\rm T}(X)/ I
\]
be the Clifford algebra. Here ${\rm T}(X)$ denotes the tensor algebra of $X$, and $I$ denotes the two-sided ideal generated by $x\otimes y - (x,y)$ for all $x,y \in X$. The isomorphism $X^{\otimes k} \rightarrow X^{\otimes k}$ given by $a_1\cdots a_k \mapsto a_k \cdots a_1$ induces the linear map $*\colon {\rm Cl}(X) \rightarrow {\rm Cl}(X)$. Then, we put $N(a) =aa^* \in \C$ for $a \in {\rm Cl}(X)$. We denote by $I_{X}$ the identity automorphism of $X$. Then, the isomorphism $(-I_X)^{\otimes k}\colon X^{\otimes k} \rightarrow X^{\otimes k}$ induces the isomorphism $\gamma\colon {\rm Cl}(X) \rightarrow {\rm Cl}(X)$ of algebras. Define 
\[
{\rm GSpin}(X) = \{ g \in {\rm Cl}(X)\mid (X \rightarrow X\colon x \mapsto \gamma(g)xg^{-1}) \in {\rm SO}(X) \}
\]  
and ${\rm Spin}(X) = \{g \in {\rm GSpin}(X) \mid N(g) = 1\}$. We denote by $\widetilde{Z}_0$ the kernel of the natural surjection ${\rm Spin}(X) \rightarrow {\rm SO}(X)$, by $\widetilde{Z}$ the inverse image of $\{\pm I_X\}$ in ${\rm Spin}(X)$. Then, $\widetilde{Z}$ coincides with the center of ${\rm Spin}(X)$ whenever $r > 1$.

\begin{prop}\label{Spin main}
Denote by $\theta$ the image of $g_0 \in {\rm O}(X)\setminus {\rm SO}(X)$ in the group of the outer automorphisms ${\rm Out}({\rm Spin}(X)) = {\rm Out}({\rm SO}(X))$. Then, $\theta$ acts on $\widetilde{Z}_0$ trivially, and acts on $\widetilde{Z}$ non-trivially.
\end{prop}

\begin{proof}
In the proof, we identify ${\rm SO}(X)$ with a subgroup of $\GL_{2r}(\C)$ via the basis $x_1, \ldots, x_r, y_r, \ldots, y_1$. We fix 
\[
g_0 = \begin{pmatrix} I_{r-1} & & \\ & J_2 & \\ & & I_{r-1} \end{pmatrix} \in {\rm O}(X).
\]
Let $T$ be the maximal torus consisting of the diagonal matrices in ${\rm SO}(X)$, and let $\widetilde{T}$ be an abstract complex torus defined by 
\[
\{(a_1, \ldots, a_n; b_1, \ldots, b_n)\mid a_k, b_k \in \C^\times, \prod_{k=1}^n a_kb_k = 1 \}/\sim.
\]
Here, the relation $\sim$ is defined by $(a_1, \ldots, a_n; b_1, \ldots, b_n) \sim (a_1',\ldots,a_n'; b_1, \ldots, b_n')$ if and only if there exist $x_1, \ldots, x_n \in \C^\times$ so that 
\begin{itemize}
\item the product $x_1\cdot \cdots \cdot x_n$ is $1$, and
\item we have $a_k' = x_k a_k$ and $b_k' = x_k b_k$ for all $k$.
\end{itemize}
Then, one can show that the homomorphism $t \colon \widetilde{T} \rightarrow {\rm Spin}_{2n}(\C)$ given by
\[
t((a_1, \ldots, a_n; b_1, \ldots, b_n)) = \prod_{k=1}^n \frac{1}{2}(a_k x_k\cdot y_k + b_k y_k\cdot x_k)
\]
is an isomorphism from $\widetilde{T}$ onto the preimage of $T$ in ${\rm Spin}(X)$. The composition $\widetilde{t}\colon \widetilde{T} \rightarrow T$ of $t$ and the covering map ${\rm Spin}(X) \rightarrow {\rm SO}(X)$ is given by
\[
\widetilde{t}(a_1, \ldots, a_n; b_1, \ldots, b_n) = \diag(a_1b_1^{-1}, \ldots, a_rb_r^{-1}, b_ra_r^{-1}, \ldots, b_1a_1^{-1})
\]
for $(a_1, \ldots, a_r, b_1, \ldots, b_r) \in \widetilde{T}$. Then, we have a bijection $u\colon \ker \widetilde{t}\rightarrow \{\pm 1\}$ given by
\[
u((a_1, \ldots, a_r; b_r, \ldots, b_1)) = a_1\cdot\cdots\cdot a_r
\]
for $(a_1, \ldots, a_r; b_r, \ldots, b_1) \in \ker \widetilde{t}$. 
Consider the element 
\[
c = (\sqrt{-1}, \ldots, \sqrt{-1}; -\sqrt{-1}, \ldots, -\sqrt{-1})
\]
of $\widetilde{Z}$. Then, we have $\theta(c)c^{-1} \in \ker \widetilde{t}$ and
\[
u(\theta(c)c^{-1}) = u((1, \ldots, 1, -1; -1, 1,\ldots, 1)) = -1.
\]
This shows $\theta(c)\not= c$, which proves the proposition.  
\end{proof}

Let $A$ be a finite subgroup of ${\rm SO}(X)$ containing $\{ \pm I_X \})$, let $B$ be the inverse image of $A$ in ${\rm Spin}(X)$, and let $\zeta$ be a character of $\widetilde{Z}$. Then, we denote by ${\rm Irr}(B, \zeta)$ the set of irreducible representations of $B$ whose restiction to $\widetilde{Z}$ meets with $\zeta$.
\begin{cor}\label{spin_fact}
Let $\zeta, \zeta'$ be the two different characters of $\widetilde{Z}$ whose restrictions to $\widetilde{Z}_0$ is non-trivial. Then, we have $\zeta\circ\theta^{-1} = \zeta'$ and $\theta$ induces the bijection form ${\rm Irr}(B, \zeta)$ onto ${\rm Irr}(B, \zeta')$. 
\end{cor}


\section{Appendix: Annotation on Fact \ref{Pre_PC}}\label{app_op}

Fact \ref{Pre_PC} is proved by Atobe \cite{Ato18} in the case \eqref{vsp I}, and by Gan and Ichino \cite{GI16} in the case \eqref{vsp II}. However, they use a bit different convention of the local theta correspondence as explained in \S\ref{conv_theta} below. In this appendix, we discuss some basic properties of the operation ``${\op}$'', and we address their convention to ours.

\subsection{The operation ``$\op$'' and representation matrices}\label{rel_ct}

Let $V$ be a right $\epsilon$-Hermitian space over $D$, and let $x_1, \ldots, x_m$ be a basis of $V$. 
Then, $x_1, \ldots, x_m$ is also a basis of $V^{\op}$. We denote by $R$ the representation matrix of the $\epsilon$-Hermitian form of $V$ with respect to $x_1, \ldots, x_m$. The following lemma will be useful for explicit computations.

\begin{lem}\label{op1}
Let $g \in G(V)(F)$. Denote by $A$ the representation matrix of $g\colon V\rightarrow V$ with respect to $x_1, \ldots, x_m$. Then, the representation matrix of $\fs_V(g)\colon V^{\op} \rightarrow V^{\op}$ with respect to $x_1, \ldots, x_m$ is $RAR^{-1}$.
\end{lem}

\begin{proof}
Recall that $\fs_V(g)x = g^{-1}x$ for $x \in V = V^{\op}$. Hence,  putting $(b_{kl})_{kl} = A^{-1}$, we have
\[
g^{-1}x_k = x_1\cdot b_{1k}+\cdots + x_m\cdot b_{mk} =b_{1k}^* \cdot x_1 + \cdots + b_{mk}^*\cdot x_m 
\] 
for $k=1, \ldots, m$. This implies that the representation matrix of $\fs_V(g)$ with respect to $x_1, \ldots, x_m$ is ${}^tA^{*-1}$ that equals to $RAR^{-1}$.
\end{proof}

\subsection{Another setting}\label{conv_theta}

Let $E$ be a quadratic extension field of $F$ or $F$ itself. In some literature (for example \cite{Ato18}, \cite{AG17}, \cite{GI16}), they considered the reductive dual pairs constructed by a right $\epsilon$-Hermitian space and a right $(-\epsilon)$-Hermitian space, that is, the actions of the unitary groups are taken from the left side. 

Let $V$ be an $m$-dimensional right $E$-vector space equipped with a \emph{left-linear} $\epsilon$-Hermitian form $( \ , \ )$, that is, the $F$-bilinear map $( \ , \ ) \colon V \times V \rightarrow E$ satisfying 
\[
(xa, y) = a(x,y), \ (y,x) = \epsilon \cdot (x,y)^*
\]
for $a \in E$ and $x,y \in V$. Let $W$ be a $(-\epsilon)$-Hermitian space equipped with a right $(-\epsilon)$-Hermitian form $\langle \ , \ \rangle$.  Then, the form $\la\la \ , \ \ra\ra'$ on $V\otimes W$ given by
\[
\la\la x_1\otimes y_1, x_2\otimes y_2 \ra\ra' = {\rm Tr}_{E/F}((x_1,x_2)\cdot \la y_1, y_2 \ra)
\] 
is symplectic. We regard $V\otimes W$ as a \emph{right} $F$-vector space. Hence, we have the natural homomorphism $\iota_{V,W}'\colon G(V) \times G(W) \rightarrow \Sp(V\otimes W)$ given by 
\[
(x\otimes y)\cdot \iota_{V,W}'(h,g) = hx\otimes gy
\]
for $x \in V, y \in W, h \in G(V)$, and $g \in G(W)$. 

\subsection{Weil representations}

We keep the setting of \S\ref{conv_theta}. To discuss the Weil representation, we introduce some auxiliary spaces.
Consider a right-linear $\epsilon$-Hermitian form $( \ , \ )^\varrho$ on $V$ given by 
\[
(x,  y )^\varrho = (x,y)^* \quad (x,y \in V^\varrho).
\] 
We denote by $V^\varrho$ the right $\epsilon$-Hermitian space $V$ equipped with the form $( \ , \ )^\varrho$. Then, we have $G(V) = G(V^\varrho)$. Choose an involution $*\colon W \rightarrow W$ so that $(xa)^* = x^* a^*$ for $x \in W$ and $a \in E$, and consider a right-linear $(-\epsilon)$-Hermitian form $\langle \ , \rangle_\varrho$ on $W$ given by
\[
\langle x,y \rangle_\varrho = \langle x^*, y^* \rangle \quad (x,y \in W).
\]
We denote by $W_\varrho$ the $(-\epsilon)$-Hermitian space $W$ equipped with the form $\la \ , \ \ra_\varrho$. Then, we have an isomorphism $\varrho\colon G(W) \rightarrow G(W_\varrho)$ given by $\varrho(g) = *\circ g \circ *$.
On $V^\varrho\otimes W_\varrho$, we consider the symplectic form $\la\la \ , \ \ra\ra$ given by
\[
\la\la x_1\otimes y_1, x_2\otimes y_2 \ra\ra = {\rm Tr}_{E/F}((x_1, x_2)^\varrho \cdot \la y_1, y_2 \ra_\varrho^*)
\]
as in \S\ref{groups}. Then, we have the natural isometry $V^\varrho\otimes W_\varrho \rightarrow  V \otimes W \colon x\otimes y \mapsto x\otimes y^*$. Thus, we denote by $\W$ the both spaces $V\otimes W$ and $V^\varrho\otimes W_\varrho$. Then, the following diagram is commutative.
 \[
 \xymatrix{
G(V) \times G(W) \ar[d]_{{\rm Id}\times \varrho} \ar[rr]^-{\iota_{V,W}'}  & & \Sp(\W) \ar@{=}[d] \\
G(V^\varrho) \times G(W_\varrho) \ar[rr]_-{\iota_{V^\varrho, W^\varrho}} & & \Sp(\W)
 }
 \]
Moreover, the following diagram is also commutative.
\[
\xymatrix{
G(V^\varrho) \times G(W_\varrho) \ar[rr]^-{\iota_{V^\varrho,W_\varrho}}\ar[d]_{\Id \times \fs_W} && \Sp(\W) \ar[d]_-{\fs_\W} \\ G(V^\varrho) \times G(W_\varrho^{\op}) \ar[rr]_-{\iota_{V^\varrho,W_\varrho^{\op}}} && \Sp(\W^{\op}) 
}
\]
Therefore, we construct the Weil representation of $G(V) \times G(W)$ from that of $G(V^\varrho)\times G(W_\varrho)$. Take a polarization $\W = \X + \Y$. Then, we have the isomorphism 
\[
(\fs_\W, \Id) \colon \Mp(c_{\psi, \Y}, \W) \rightarrow \Mp(c_{\psi, \Y^{\op}}, \W^{\op})
\]
where we write $\Y^{\op}$ instead of $\Y$ to emphasize that we regard it as a subspace of $\W^{\op}$.
Taking characters $\chi_V, \chi_W$ as in \S\ref{LTC}, we define the lifting $\widetilde{\iota}_{V,W,\chi_V, \chi_W}' \colon G(V) \times G(W) \rightarrow \Mp(c_{\psi, \Y}, \W)$ of the embedding $\iota_{V,W}'$ by the composition
\[
(\fs_\W, \Id)^{-1} \circ \widetilde{\iota}_{V^\varrho,W_\varrho^{\op},\chi_V^{-1}, \chi_W} \circ (\Id \times (\fs_W\circ\varrho)).
\]
Hence, we obtain the Weil representation $\omega_{\psi, V,W}^{\chi_V, \chi_W} $ of $G(V) \times G(W)$ given by
\[
\omega_{\psi, \Y^{\op}}\circ \widetilde{\iota}_{V^\varrho,W_\varrho^{\op},\chi_V^{-1}, \chi_W} \circ (\Id \times (\fs_W\circ\varrho)).
\]

\begin{rem}
This construction is consistent with \cite{Ato18}, \cite{AG17}, and \cite{GI16}:
\begin{itemize}
\item In the case \eqref{vsp I}, one can show \cite[Proposition 7.3]{Ato18}, \cite[Proposition 4.10 (4)]{AG17} using the Weil representation $\omega_{\psi, V,W}^{\chi_V, \chi_W}$, which are crucial parts of calculations to determine the behavior of the characters of S-group under local theta correspondences. 
\item In the case \eqref{vsp II}, one can verify that the Weil representation $\omega_{\psi, V,W}^{\chi_V, \chi_W}$ satisfies the twelve formulae in \cite[pp.~758]{GI16}. Note that the auxiliary trace zero element $\delta \in E$ is chosen as $(-\daleth)$. 
\end{itemize}
\end{rem}

\subsection{Langlands parameters}

We keep the setting of \S\ref{conv_theta}. We describe the behavior of the Langlands parameter under the operations ``${\op}$'' and $\varrho$. 

Denote by ${}^\#W_c$ the right $(-\epsilon)$-Hermitian space so that $({}^\#W_c)_\varrho = (W_\varrho)_c^\#$. Let $f$ be a bijective isometry over $E\otimes\overline{F}$ from ${}^\#W_c\otimes\overline{F}$ onto $W\otimes \overline{F}$. As explained in the introduction, we denote by $\varphi_f$ the isomorphism from $G({}^\#W_c)$ onto $G(W)$ so that $\varphi_f(g)f(x) = f(gx)$ for $g \in G({}^\#W_c)(\overline{F})$ and $x \in {}^\#W_c\otimes\overline{F}$, and by $t_f$ the cocycle in $Z^1(\Gamma, G({}^\#W_c))$ given by $t_f(\sigma) = f^{-1}\circ \sigma \circ f \circ \sigma$ for $\sigma \in \Gamma$. Then, we have $(t_f, \varphi_f)$ is a pure inner twist. 

We denote by $f_\varrho$ the composition $*\circ f\circ *$. Then, $f_\varrho \colon (W_\varrho)_c^\# \otimes \overline{F} \rightarrow W_\varrho \otimes \overline{F}$ is linear and isometric, which induces an isometry $f_\varrho^{\op}$ from $(W_\varrho)_c^{\#\op} \otimes \overline{F}$ onto $W_\varrho^{\op} \otimes \overline{F}$. We define the isometry $f'$ from $(W_\varrho^{\op})_c^\# \otimes \overline{F}$ onto $W_\varrho^{\op} \otimes \overline{F}$ by the composition
\begin{align*}
\xymatrix{(W_\varrho^{\op})_c^\#\otimes\overline{F} \ar[r]^-{\gamma} & (W_\varrho)_c^{\#\op}\otimes \overline{F} \ar[r]^-{f_\varrho^{\op}} & W_\varrho^{\op}\otimes\overline{F}}
\end{align*}
where $\gamma$ denotes the isometry given by $\gamma(x) = {}^tx^*$ for $x \in (W_\varrho^{\op})_c^\# \otimes \overline{F}$. We denote by $\varphi_{f'}$ the isomorphism from $G((W_\varrho^{\op})_c^\#)$ onto $G(W_\varrho^{\op})$ so that $f'(x)\varphi_{f'}(g) = f'(xg)$ for $x \in (W_\varrho^{\op})_c^\# \otimes \overline{F}$ and $g \in G((W_\varrho^{\op})_c^\#)(\overline{F})$, and by $t_{f'}$ the cocycle in $Z^1(\Gamma, G((W_\varrho^{\op})_c^\#))$ given by $t_{f'}(\sigma) = \sigma\circ {f'}^{-1}\circ\sigma^{-1}\circ f'$ for $\sigma \in \Gamma$. Then, we have $(t_{f'}, \varphi_{f'}) \in \RIT^\star((W_\varrho^{\op})_c^\#, W_\varrho^{\op})$. 

We define the L-group of $G_0({}^\#W_c)$ via the identification $G({}^\#W_c) = G(({}^\#W_c)^\varrho)$. Then, the isomorphism $\widehat{\varrho}\colon G_0({}^\#W_c)^\wedge \rightarrow G_0((W_\varrho^{\op})_c^\#)^\wedge$ induced by the composition
\begin{align}\label{compop}
\xymatrix{
G_0({}^\#W_c) \ar[r]^-{\varrho}  & G_0((W_\varrho)_c^\#) \ar[r]^-{\fs_{(W_\varrho)_c^\#}} &  G_0((W_\varrho)_c^{\#\op}) \ar[r]^-{\varphi_\gamma} & G_0((W_\varrho^{\op})_c^\#) 
}
\end{align}
is given by
\[
\widehat{\varrho}(g) = \begin{cases} g & (E = F), \\ {}^tg^{-1} & ([E:F] = 2) \end{cases} \qquad (g \in G_0({}^\#W_c)^\wedge). 
\]

\begin{prop}\label{LLC-op}
Let $\phi$ be a tempered $L$-parameter for both $G_0({}^\#W_c)$, and let $\pi$ be a tempered irreducible representation of $G_0(W)(F)$ having the $L$-parameter $\phi$. Then, $\pi\circ \fs_{W_\varrho}\circ \varrho$ has the L-parameter $\widehat{\varrho}\circ \phi$, and we have
\[
\iota_{\widehat{\varrho}\circ\phi}[\fw_-, t_{f'}, \varphi_{f'}](\pi\circ\fs_{W_\varrho}\circ\varrho)(\widehat{\varrho}(\dot{s})) = \iota_{\phi}[\fw, t_f, \varphi_f](\pi)(\dot{s})
\]
where $\fw_-$ is the Whittaker data of $G((W_\varrho^{\op})_c^\#)$ defined in \S\ref{WD}, and $\fw$ is the Whittaker data of $G({}^\#W_c)$  associated with $\psi$ (resp. $x \mapsto \psi_{1/2}({\rm Tr}(x\cdot (-\daleth)))$) when $-\epsilon = -1$ (resp. $-\epsilon = 1$).
\end{prop}

\begin{proof}
First, we have the following diagram is commutative.
\[
\xymatrix{
G({}^\#W_c) \ar[rr]^{\varphi_f}\ar[d] & & G(W) \ar[d]^{\fs_{W_\varrho}\circ\varrho} \\ G((W_\varrho^{\op})_c^\#) \ar[rr]_{\varphi_{f'}} & & G(W_\varrho^{\op})
}
\]
Here, the left column map is the isomorphism \eqref{compop}. 
Second, for $x \in {}^\#W_c\otimes\overline{F}$ and $\sigma \in \Gamma$, we have
\begin{align*}
\gamma(x)\cdot (\varphi_\gamma\circ\fs_{(W_\varrho)_c^\#}\circ\varrho)(t_f(\sigma)) &= \gamma( x \cdot (\fs_{(W_\varrho)_c^\#}\circ\varrho)(t_f(\sigma))) \\
&= \gamma(\varrho(t_f(\sigma))^{-1} \cdot x) \\
&=t_{f'}(\sigma)^{-1} ( \gamma(x) ) \\
&= \gamma(x) \cdot t_{f'}(\sigma).
\end{align*}
Thus, the 1-cocycle $t_f$ corresponds to $t_{f'}$ by the isomorphism \eqref{compop}. 
Moreover, one can show that $\fw$ is transferred from $\fw_-$ by the isomorphism \eqref{compop}.
Finally, it remains to show that $\pi\circ\fs_W\circ \varrho$ has the L-parameter $\widehat{\varrho}\circ\phi$. In the case $E = F$, this is obvious. Hence, we assume $[E:F] =2$. It suffices to show that 
\begin{align}\label{eq_planc}
\mu(s, (\pi\circ\fs_{W_\varrho}\circ\varrho) \boxtimes \tau) = \mu(s, \phi_\pi^\vee \boxtimes \phi_\tau)
\end{align}
for all irreducible square-integrable representations $\tau$ of $\GL_k(E)$ for $k=1, \ldots, n$ when $\pi$ is square-integrable. By the definition of the Plancherel measures of L-parameters (\S\ref{Lpkt}), we have
\begin{align}\label{planc_rhs}
\mu(s, \phi_\pi^\vee\boxtimes\phi_\tau) = \mu(-s, \phi_\pi\boxtimes\phi_\tau^\vee).
\end{align}
Let $X_0, Y_0$ be a $k$-dimensional right $F$-vector space, let $x_1,\ldots, x_k$ be a basis of $X_0$, let $y_1, \ldots, y_k$ be a basis of $Y_0$. Put $X = X_0 \otimes E$ and $Y = Y_0 \otimes E$. We define a left-linear $(-\epsilon)$-Hermitian form $\la \ , \ \ra'$ on $X \oplus Y$ so that 
\[
\la x_r, y_s \ra' = \delta_{r,s} \quad (1 \leq r,s \leq k), 
\]
and put $W' = W \bot (X \oplus Y)$. Denote by $P$ the maximal parabolic subgroup of $G(W')$ preserving $X$, and by $Q$ the parabolic subgroup $(\fs_{W_\varrho'}\circ \varrho)(P)$ of $G(W_\varrho^{' \op})$. Then, identifying $\GL(X)$ with $\GL_k(E)$ via the basis $x_1, \ldots, x_k$, we have
\[
\Ind_{Q}^{G(W_\varrho^{' \op})} (\pi\circ \fs_{W_\varrho} \circ \varrho) \boxtimes \tau = (\Ind_{P}^{G(W')} \pi\boxtimes \tau^\vee)\circ \fs_{W_\varrho'}\circ \varrho,
\]
which implies that
\begin{align}\label{planc_lhs}
\mu(s, (\pi\circ\fs_{W_\varrho}\circ\varrho)\boxtimes \tau) = \mu(-s, \pi\boxtimes \tau^\vee).
\end{align}
By \eqref{planc_rhs} and \eqref{planc_lhs}, we have \eqref{eq_planc}. This completes the proof of Proposition \ref{LLC-op}.
\end{proof}


\section{Appendix: Annotation on Fact \ref{pre_R} (i)}\label{app theta HC 1}

As in \S\ref{LTC}, the mainstream notation of Weil representation (or the oscillator representation) would depend on a non-trivial additive character $\psi$ of $F$. However, in \cite{LPTZ03} and \cite{Li89}, the non-trivial additive character in the definition of the oscillator representation is implicit (see Remark \ref{point} \eqref{point0}). In \S\S\ref{eH=1_cpt}--\ref{eH=-1_cpt}, we summarize a computation in the case one of the reductive groups consisting of the dual pair is anisotropic with fixing specific non-trivial additive character $\psi$ of $\R$.

Let $V_{p,q}$ (resp. $W_{p,q}$) be the Hermitian space (resp. skew-Hermitian space) over $\H$ defined in \S\ref{Arch_computation}. Recall that $X^*(S_+)$ and $X^*(S_-)$ are identified with $\Z^m$ and $\Z^n$. For a non-trivial additive character $\psi\colon \R \rightarrow \C^1$, we denote by $d_\psi$ the complex number satisfying $\psi(x) = e^{d_\psi x}$ for $x \in \R$, and put $\epsilon_\psi = d_\psi / |d_\psi|$. 

\subsection{The case \eqref{vsp I} with $W$ anisotropic}\label{eH=1_cpt}

The local theta correspondence for the dual pair $G(V_{m,0})\times G(-W_{p,0})$ with $e_\H = 1$ has been described by Kashiwara and Vergne \cite{KV78}. More precisely, they studied the representation $L_{k}$ of ${\rm Mp}(n) \times {\rm O}(k)$ on the space $L^2(M_{n,k})$. In the modern terminologies, at least when $k$ is even, the representation $L_k$ coincides with the restriction of the Weil-representation $\omega_{\psi}$ with $\epsilon_\psi = -\sqrt{-1}$, which can be verified by using the discussion of \S\ref{rel_ct} and by the formula of the projective representation $L_{k}$ of $\Sp(n)$ (see \cite[II.1.3]{KV78}). 

Now, we state a part of their results in the setting of our paper. Recall that identified $X^*(S_+)$ and $X^*(S_-)$ with $\Z^m$ and $\Z^n$. We put $K_+ = G(V_{m,0})(\R) \cap \GL_m(\R(i))$ where $\R(i)$ denotes the sub-field of $\H$ spanned by $\R$ and $i$. This is a maximal compact subgroup of $G(V_{m,0})(\R)$ containing $S_+$. Note that the signature of the quadratic space $-W_{p,0}^\natural$ is $(2p, 0)$ (see \S\ref{morita}). Then, they essentially proved the following:

\begin{fact}\label{KVfact}
Let  $\sigma$ be an irreducible representation of $G(-W_{n,0})(\R)$ having the highest weight $(\nu_1, \ldots, \nu_k, 0, \ldots, 0)$ where $0 \leq k \leq n$ so that $\nu_k \not=0$. Denote by $\mu(\sigma)$ the signature of $\sigma$ (in the sense of \cite[(6.10)]{KV78}). Then, for a non-trivial additive character $\psi$ of $\R$ with $\epsilon_\psi = -\sqrt{-1}$, $\Theta_\psi(\sigma)$ is non-zero if and only if either 
\begin{itemize}
\item $\mu(\sigma) = +$ and $\nu_k = 0$ for $k > m$, or
\item $\mu(\sigma) = -$, $n < m$, and $\nu_j \not=0$ for $j \leq 2(n-m)$.
\end{itemize}
Moreover, if $\Theta_\psi(\sigma)$ is non-zero, then it is irreducible and the $K$-type of the minimal degree has the highest weight
\begin{align}\label{KVcorr}
(0, \ldots, 0, -1, \ldots, -1, -\nu_k, \ldots, -\nu_1) - (n, \ldots, n).
\end{align}
where $0$ appears in $(m-k) - (1- \mu(\sigma))(n-k)$ times in the first term. 
\end{fact}

We denote by $\tau_n(\sigma)$ the irreducible representation of $K_+$ having the highest weight \eqref{KVcorr}. Note that if we use $-W_{0,n}$ instead of $-W_{n,0}$, then Fact \ref{KVfact} still hold by only replacing \eqref{KVcorr} with 
\begin{align}\label{KVcorr2}
(\nu_1, \ldots, \nu_k, 1, \ldots, 1, 0, \ldots, 0) + (n, \ldots, n).
\end{align}
We denote by $\tau_n'(\sigma)$ the irreducible representation of $K_+$ having the highest weight \eqref{KVcorr2}. 

\subsection{The case \eqref{vsp III} with $q=0$}\label{eH=-1_cpt}

In the case $e_\H = -1$, it seems to be necessary to compute the $K$-type correspondence in the space of joint harmonics for the dual pair $G(V_{p,0})\times G(W_{n,0})$. First, we recall the Fock model of the Weil representation following \cite{KK07} quickly. Let $\X, \Y$ be isotropic subspaces so that $\W = \X + \Y$, let $e_1, \ldots, e_N$ be a basis of $\X$ over $F$, let $e_1', \ldots, e_N'$ be a basis so that $\la\la e_k, e_l' \ra\ra =\delta_{k,l}$. Then, we denote by $\mathbb{K}$ the complex subspace of $\W\otimes\C$ spanned by $e_k - \sqrt{-1} e_k'$ for $k = 1, \ldots, N$, and by $\mathbb{L}$ the complex subspace of $\W\otimes\C$ spanned by $e_k' - \sqrt{-1} e_k$ for $k = 1, \ldots, N$. We consider the quantum algebra $\Omega_\psi(\W\otimes \C)$, which is given by
\[
T(\W\otimes\C)/ I(\{ w\otimes w' - w'\otimes w - d_\psi\la \la w, w' \ra \ra \mid w,w' \in \W\otimes \C\})
\]
where $T(\W\otimes\C)$ is the tensor algebra of $\W\otimes\C$ and $I(A)$ denotes the two-sided ideal generated by a given subset $A$ of $T(\W\otimes\C)$. Then, the quotient $\Omega_\psi(\W\otimes\C)/\Omega_\psi(\W\otimes\C)\K$ is naturally isomorphic to the symmetric algebra $\Sym(\L)$ of $\L$. Then there exists a Lie algebra homomorphism $\cF_\psi^{(\W)}\colon \fsp(\W) \rightarrow \Omega_\psi(\W\otimes\C)$ so that 
\begin{align}
\label{fock_char1} &\cF_\psi^{(\W)}(\fsp(\W)) \subset \Omega_\psi(\W\otimes\C)^{(2)}, \\
\label{fock_char2} &\cF_\psi^{(\W)}(X)\otimes w - w \otimes \cF_\psi^{(\W)}(X) = w\cdot X
\end{align}
for $w \in \W\otimes\C$ and $X \in \fsp(\W)$ . One can show that the Lie-algebra homomorphism $\cF_\psi^{(\W)}$ is determined uniquely by the conditions \eqref{fock_char1} and \eqref{fock_char2}. We write $\cF_\psi$ instead of $\cF_\psi^{(\W)}$ if there is no fear of confusion. By this embedding, we have the action $\fsp(\W)$ on $\Sym(\L)$, which is called the Fock model of the Weil representation and is referred to as $r_\psi$ in \cite{KK07}. We identify $\W$ with $\rM_{m,n}(\H)$ by the isomorphism given by $x\otimes y \mapsto (x_k \cdot y_l)_{k,l}$. Then, the symplectic form $\la\la \ , \ \ra\ra$ on $\rM_{m,n}(\H)$ is given by
\[
\la\la X, Y \ra\ra = \Tr_{\H/\R}(X \cdot i \cdot {}^tY^*) 
\] 
for $X, Y \in \rM_{p,n}(\H)$. Thus, the subspaces
\[
\X = \rM_{p,n}(\R(j)) \mbox{ and } \Y = \rM_{p,n}(\R(j))\cdot i
\]
are isotropic. Obviously, we have $\W = \X + \Y$. We denote by $e_{a,b}(x)$ the matrix whose $(a,b)$-component is $x$ and the other components are $0$. Take the basis $\{e_{a,b}(1), e_{a,b}(j)\}_{1 \leq a \leq p, 1 \leq b \leq n}$ for $\X$ and $\{e_{a,b}(i), -e_{a,b}(ij)\}_{1 \leq a \leq p, 1 \leq b \leq p}$ for $\Y$.  Then, the basis $\{e_{a,b}(1), e_{a,b}(j) \}_{a,b} \cup\{ e_{a,b}(i), -e_{a,b}(ij)\}_{a,b}$ consists a Witt basis of $\W$ in the sense of \cite[\S2]{KK07}.  Moreover, we put 
\begin{align*}
&\be_{a,b} = \frac{1}{2}(e_{a,b}(1) - \epsilon_\psi \cdot e_{a,b}(i)), \\
&\bf_{a,b} = \frac{1}{2}(e_{a,b}(j) + \epsilon_\psi \cdot e_{a,b}(ij)), \\
&\be_{a,b}' = \frac{1}{2}(-\epsilon_\psi \cdot e_{a,b}(1) + e_{a,b}(i)), \\
&\bf_{a,b}' = \frac{1}{2}(-\epsilon_\psi \cdot e_{a,b}(j) -  e_{a,b}(ij))
\end{align*}
for $a = 1, \ldots, p$ and $b = 1, \ldots, n$. We denote by $\K$ the subspace of $\W\otimes_\R\C$ spanned by $\{\be_{a,b} , \bf_{a,b}\}_{a,b}$, and by $\L$ the subspace of $\W\otimes_\R\C$ spanned by $\{\be_{a,b}', \bf_{a,b}'\}_{a,b}$. 
We write down the formulas of $\cF_\psi(d\iota(X))$ when $X$ is in the image of the differential $d\iota$ of $\iota$.
Put \[
\sigma_{a,b}(x) = \frac{1}{2}(e_{ab}(x) - e_{ba}(x^*))
\]
for $x \in \H$, and put
\begin{align*}
h_{a,b} &= \epsilon_\psi \sigma_{a,b}(1) + \sigma_{a,b}(i), \\
x_{a,b} &= \epsilon_\psi \sigma_{a,b}(j) +  \sigma_{a,b}(ij), \\
y_{a,b} &= \epsilon_\psi \sigma_{a,b}(j) - \sigma_{a,b}(ij).
\end{align*}
Then, they spans the Lie algebra $\fg(V)\otimes\C$ as a vector space over $\C$, and we have
\begin{align}
\cF_\psi(d\iota(h_{ab})) &= \epsilon_\psi\cdot\sum_{c=1}^n \left( w_{ac}\frac{\partial}{\partial w_{bc}} - z_{bc}\frac{\partial}{\partial z_{ac}}\right), \label{fock_h1} \\
\cF_\psi(d\iota(x_{ab})) &= \epsilon_\psi\cdot \sum_{c=1}^n\left( w_{bc}\frac{\partial}{\partial z_{ac}}  - w_{ac}\frac{\partial}{\partial z_{bc}} \right), \label{fock_x} \\
\cF_\psi(d\iota(y_{ab})) &= \epsilon_\psi\cdot\sum_{c=1}^n\left(z_{ac}\frac{\partial}{\partial w_{bc}} - z_{bc}\frac{\partial}{\partial w_{ac}}\right) \label{fock_y}
\end{align}
for $1 \leq a, b \leq p$ with $a\not=b$. On the other hand, put
\[
s_{ab}(x) = \frac{1}{2} (e_{ab}(x) - e_{ba}(ix^*i^{-1})) 
\]
for $x \in D$, and put
\begin{align*}
k_{ab} &= s_{a,b}(i) + \epsilon_\psi s_{a,b}(1), \\
p_{a,b} &= s_{a,b}(j) - \epsilon_\psi \cdot s_{a,b}(ij), \\
\overline{p}_{a,b} &= s_{a,b}(j) + \epsilon_\psi \cdot s_{a,b}(ij).
\end{align*}
Then they spans the Lie algebra $\fg(W)\otimes\C$ as a vector space over $\C$, and we have
\begin{align}
\cF_\psi(d\iota(k_{ab})) &=\epsilon_\psi\cdot \sum_{c=1}^m \left(z_{ca}\frac{\partial}{\partial z_{cb}} + w_{ca}\frac{\partial}{\partial w_{cb}}\right) + \epsilon_\psi\cdot \delta_{a,b} \cdot m, \label{fock_p1}\\
\cF_\psi(d\iota(p_{ab})) &= \frac{1}{|d_\psi|} \sum_{c=1}^m (z_{ca}w_{cb} - w_{ca}z_{cb}) \label{Fock tW}\\
\cF_\psi(d\iota(\overline{p}_{ab})) &=  |d_\psi| \sum_{c=1}^m \left(\frac{\partial^2}{\partial z_{cb} \partial w_{ca}} - \frac{\partial^2}{\partial z_{ca}\partial w_{cb}}  \right)  \label{fock_p2}
\end{align}
for $1 \leq a, b \leq n$. 
Then, we consider special vectors given by as follows. Let $\underline{r} = (r_1, \ldots, r_p) \in \Z^p$.
We define
\begin{align}
&v(\underline{r}) = \prod_{k=1}^p \det\begin{pmatrix} w_{1,1} & \cdots & w_{1,k} \\ \vdots & \ddots & \vdots \\ w_{k,1} & \cdots & w_{k,k} \end{pmatrix}^{r_k}, \label{max_vct1}
\end{align}
By using the formulae \eqref{fock_h1}-\eqref{fock_p2}, we have the following.
\begin{prop}\label{q=0_desc}
Assume that $\epsilon_\psi = \sqrt{-1}$. 
\begin{enumerate}[(1)]
\item The polynomial $v(\underline{r})$ is contained in the space of joint harmonics.
\item The polynomial $v(\underline{r})$ is a maximal vector with respect to both $\Delta_c^+$ and $\Delta_c^-$. 
\item The action of ${\rm Lie}(S_+) \times {\rm Lie}(S_-)$ on $v(\underline{r})$ is given by the character
\[
\sum_{k=1}^p (r_k +  \cdots + r_p)\alpha_k + \sum_{l = 1}^n (p + r_l + \cdots + r_n)\beta_l.
\] 
Here, we put $r_t = 0$ if $t > p$.
\end{enumerate}
\end{prop}

\subsection{The correspondence of limits of discrete series}\label{conclusion} 

Assume that $\epsilon_\psi = \sqrt{-1}$. Then, we have:
\begin{prop}\label{DS_corr}
Put $(V,W) = (V_{m,0}, W_{p,q})$ if $e_\H = 1$ and $(V,W) = (V_{p,q}, W_{n,0})$ if $e_\H=-1$.  
Let $\sigma$ be an irreducible limit of discrete series representation of $G(V)(\R)$ having the Harish-Chandra parameter $(\mu_\sigma, \Psi_\sigma)$. Then, $\theta_\psi(\sigma, W)$ is non-zero if and only if  $\xi_\bullet^{\sqrt{-1}}(\mu_\sigma, \Psi_\sigma) \in \cY$. Moreover, if $\theta_\psi(\sigma, W) \not=0$, then its Harish-Chandra parameter is $\xi_\bullet^{\sqrt{-1}}(\mu_\sigma, \Psi_\sigma)$.
\end{prop}
This proposition implies that the non-trivial additive character defining the Weil representation in \cite{Li89} is $\psi$ with $\epsilon_{\psi} = \sqrt{-1}$. The proof goes the same line with \cite{Li89}. However, we write the proof for the readers since we discuss a bit extended version.

The strategy of the proof is the use of the characterization of the module ``$A_{\mathfrak{q}}(\lambda)$'' in terms of infinitesimal characters and $K$-types \cite[Proposition 6.1]{VZ84} (see also \cite[Proposition 6.1]{Li89}). We put $(\mu, \Psi) = \xi_\bullet^{\sqrt{-1}}(\mu_\sigma, \Psi_\sigma)$. Denote by $\chi[\mu]$ the infinitesimal character obtained by $\eta$ via the Harish-Chandra isomorphism. Denote by $\mathfrak{z}^{G(W)}$ the algebra of the $G(W)$-fixed points of the center $\mathfrak{z}$ of the universal enveloping algebra of $\fg(W)$. Then, the restriction of an infinitesimal character of an irreducible component of $\theta_\psi(\sigma, W)|_{G_0(W)}$ to $\mathfrak{z}^{G(W)}$ is determined uniquely from $\theta_\psi(\sigma, W)$ if it is non-zero, which we denote by $\chi_{\theta_\psi(\sigma, W)}$. Then, by \cite[Theorem 1.13]{Prz96}, we obtain
\begin{align}\label{infin_corr}
\chi[\mu]|_{\mathfrak{z}^{G(W)}} = \chi_{\theta_\psi(\sigma, W)}.
\end{align}

Then, we analyze the $K$-types correspondence in the space of the joint harmonics \cite{How89}. We denote by $\underline{1}_k$ the element $(1, \ldots, 1)$ of $\Z^k$ for a positive integer $k$.
\begin{lem}\label{JHlem}
Assume that $\theta_\psi(\sigma, W) \not=0$.
\begin{enumerate}
\item The lowest $K$-type of $\sigma$ is given by $\mu_\sigma + \rho(\Psi_\sigma) - 2 \rho(\Delta_c^+)$. \label{LKTfml}
\item The $K$-type $\mu_\sigma + \rho(\Psi_\sigma) - 2 \rho(\Delta_c^+)$ occurs in the space of joint harmonics.\label{LKT is min}
\item  The space of joint harmonics contains 
\[
(\mu_\sigma +  \rho(\Psi_\sigma) - 2 \rho(\Delta_c^+))\boxtimes  \xi_{\bullet0}^{\sqrt{-1}}(\mu_\sigma + \rho(\Psi_\sigma) - 2 \rho(\Delta_c^+))
\]
as representation of $K_+ \times K_-$. Here, we put 
\[
\xi_{\bullet0}^{\sqrt{-1}}(a) =  \begin{cases} \xi_\bullet^{\sqrt{-1}}(a - (p-q)\cdot \underline{1}_m) & (e_\H = 1), \\  \xi_\bullet^{\sqrt{-1}}(a) + (p-q) \cdot \underline{1}_n & (e_\H = -1)
\end{cases}
\]
for $a \in \Z^m$. \label{JHlem3}
\item We have
\[
\xi_{\bullet0}^{\sqrt{-1}}(\mu_\sigma + \rho(\Psi_\sigma) - 2 \rho(\Delta_c^+)) = \mu + \rho(\Psi) - 2\rho(\Delta_c^-).
\] \label{JHlem4}
\end{enumerate}
\end{lem}

\begin{proof}
The proof of the assertion \eqref{LKTfml} is contained in \cite[\S7]{Vog79}. Then, by the formula of the degree of the $K$-types (c.f. \cite[Lemma 1.4.5]{Pau98}, \cite[Lemma 3.4]{LPTZ03}), we have $\mu_\sigma + \rho(\Psi_\sigma) - 2\rho(\Delta_c^+)$ has the minimal degree. This proves \eqref{LKT is min}. 

We prove \eqref{JHlem3}. We only discuss in the $e_\H=-1$ case for simplicity. The parallel proof goes for $e_\H=1$ cases except that some replacements of symbols are necessary because not $V_{p,0}, V_{0,q}$ but $W_{p,0}, W_{0,q}$ are anisotropic. We denote by $\W_{p,0}$ the tensor product $V_{p,0} \otimes W_{n,0}$, and by $\W_{0,q}$ the tensor product $V_{0,q}\otimes W_{n,0}$.  We denote by $\L$ a maximal subspace of $\W\otimes \C$ so that the Hermitian form $(x,y) \mapsto -\sqrt{-1}\la\la x, \overline{y} \ra\ra$ on $\L$ is negatively defined and nondegenerate. Then, $\L$ decomposes into $\L_{p,0}\oplus \L_{0,q}$ along with $\W\otimes \C = (\W_{p,0}\otimes \C) \oplus (\W_{0,q}\otimes\C)$. As in \S\ref{eH=-1_cpt}, we can take a basis $\{z_{ab}, w_{ab} \mid 1\leq a \leq p, 1 \leq b \leq n\}$ of $\L_{p,0}$, and a basis $\{z_{ab}, w_{ab}\mid p+1 \leq a \leq m, 1\leq b \leq n\}$ of $\L_{0,q}$. Denote by $\mathscr{D}_{p,0}$ (resp. $\mathscr{D}_{0,q}$) the set of the minors of either of the matrices
\[
(z_{ab})_{1 \leq a \leq p, 1 \leq b \leq n}, \ (w_{ab})_{1 \leq a \leq p, 1 \leq b \leq n} \quad (\mbox{resp. } (z_{ab})_{p+1 \leq a \leq m, 1 \leq b \leq n} \ (w_{ab})_{p+1 \leq a \leq m, 1 \leq b \leq n} ).
\]
For example, the polynomial $v(\underline{r})$ of \eqref{max_vct1} is contained in $\mathscr{D}_{p,0}$.
Let $v_{p,0} \in \Sym(\L_{p,0})$ (resp. $v_{0,q} \in \Sym(\L_{0,q}$) be a product of polynomials in $\mathscr{D}_{p,0}$ (resp. $\mathscr{D}_{0,q}$), and let $v_0$ be the polynomial in $\Sym(\L) = \Sym(\L_{p,0})\otimes \Sym(\L_{0,q})$ given by $v_{p,0}\otimes v_{0,q}$. Then, we can verify that $v_0$ lies in the space of joint harmonics as follows. For a Lie sub-algebra $\mathfrak{l}$ of $\fsp(\W)$, we denote by $\mathfrak{l}^{(2)}$ the set of $X \in \mathfrak{l}$ whose image $\cF_\psi(X)$ is belonging to the $\C$-subspace of $\Omega_\psi(\W\otimes\C)$ spanned by 
\[
\frac{\partial^2}{\partial z_{ab}\partial z_{cd}}, \quad \frac{\partial^2}{\partial z_{ab}\partial w_{cd}}, \quad \frac{\partial^2}{\partial w_{ab}\partial w_{cd}}
\]
for various $a,b,c,d$. 
Denote by $M_V$ the centralizer of $K_-$ in $\Sp(\W)$, and by $\mathfrak{m}_V$ its Lie algebra. For $X \in \mathfrak{m}_V$ and $x \in \H$, one can show that 
\[
X \cdot E_{ab}(x) = \sum_{c=1}^m E_{cb}(x_c)
\]
for some $x_1, \ldots, x_m \in \H$. Hence, an element of $\cF_\psi(\mathfrak{m}_V^{(2)})$ is of the form
\[
\sum_{a,b,c}t_{a,b,c}\cdot \frac{\partial^2}{\partial z_{ac}\partial z_{bc}} + u_{a,b,c} \cdot \frac{\partial^2}{\partial z_{ac}\partial w_{bc}} + v_{a,b,c}\cdot \frac{\partial^2}{\partial w_{ac}\partial w_{bc}}
\]
where $t_{a,b,c}, u_{a,b,c}, v_{a,b,c} \in \C$. This implies that $\cF_\psi(\mathfrak{m}_V^{(2)}) \cdot v_0 = 0$. Similary, one can show that $\cF_\psi(\mathfrak{m}_W)\cdot v_0=0$ where we denote by $M_W$ the centralizer of $K_+$ in $\Sp(\W)$, and by $\mathfrak{m}_W$ its Lie algebra. Hence, $v_0$ lies in the space of joint harmonics. By combining this with Proposition \ref{q=0_desc}, we have \eqref{JHlem3}.

Finally, we prove \eqref{JHlem4}. Assume $n=m$. By the definition of $\pi$, we have 
\[
\begin{cases}
\xi^{\sqrt{-1}}(\Psi_\sigma) = \Psi \cup \{2\beta_1, \ldots, 2\beta_n\} & (e_\H = 1), \\
\xi^{\sqrt{-1}}(\Psi_\sigma \setminus \{ 2\alpha_1, \ldots, 2\alpha_m \}) = \Psi & (e_\H =-1),
\end{cases}
\]
which implies $\xi^{\sqrt{-1}}(\rho(\Psi_\sigma)) = \rho(\Psi) + \underline{\epsilon}$. Here, $\underline{\epsilon} \in \Z^m$ is defined in \S\ref{choice_real}. Moreover, we have
\[
\xi^{\sqrt{-1}}(2\rho(\Delta_c^+) + e_\H\cdot (p-q) \cdot \underline{\epsilon}) = 2\rho(\Delta_c^-) + \underline{\epsilon}.
\]
Hence, we have \eqref{JHlem4}. Then, assume $e_\H = 1$ with $n= m + 1$ which equals to $p+q$. In this case, we have
\begin{align*}
&\xi_\btru^{\sqrt{-1}}(\Psi_\sigma) = \Psi \cup \{2\beta_k\}_{k\not=p} \setminus \{\beta_k \pm \beta_p\}_{k\not= p}, \\
&\xi_\btrd^{\sqrt{-1}}(\Psi_\sigma) =  \Psi \cup \{2\beta_k\}_{k\not=n} \setminus \{\beta_k \pm \beta_n\}_{k\not= n},
\end{align*}
which implies $\xi_\bullet^{\sqrt{-1}}(\rho(\Psi_\sigma)) = \rho(\Psi)$ for $\bullet = \btru$ or $\btrd$. Moreover, we have
\[
\xi_\bullet^{\sqrt{-1}}(2\rho(\Delta_c^+) + (p-q)\cdot \underline{1}_m) = 2\rho(\Delta_c^-) 
\]
where $\underline{1}_m = (1, \ldots, 1) \in \Z^m$. Hence, we have \eqref{JHlem4}. Then, we assume $n=m+1$ with $e_\H = -1$. In this case, we have
\[
\xi^{\sqrt{-1}}(\Psi_\sigma \setminus \{2\alpha_k\}_{k=1}^m) = \Psi \setminus \{\frac{p+1 -k}{|p+1-k|}\beta_k \pm \beta_{p+1}\}_{k\not= p+1} 
\]
which implies $\xi^{\sqrt{-1}}(\rho(\Psi_\sigma)) = \rho(\Psi)$. Moreover, we have
\[
\xi^{\sqrt{-1}}(2\rho(\Delta_c^+)) = 2\rho(\Delta_c^-) + (p-q)\cdot \underline{1}_n.
\]
Hence, we have \eqref{JHlem4} in all cases, and we complete the proof of Proposition \ref{JHlem}. 
\end{proof}

By Lemma \ref{JHlem} and \cite[Proposition 6.1]{VZ84}, we have the following:
\begin{cor}
If $\theta_\psi(\sigma, W) \not=0$, then we have $\xi_\bullet^{\sqrt{-1}}(\mu_\sigma, \Psi)\in \cY$ and the  Harish-Chandra parameter of $\theta_\psi(\sigma, W)$ is given by $\xi_\bullet^{\sqrt{-1}}(\mu_\sigma, \Psi)$.
\end{cor}

It remains to show that if $\xi_\bullet^{\sqrt{-1}}(\mu_\sigma, \Psi)\in \cY$ then $\theta_\psi(\sigma, W)$ is non-zero. We only discuss the $e_\H = -1$ case for simplicity. The parallel proof goes for $e_\H=1$ cases except that some replacements of symbols are necessary because not $V_{p,0}, V_{0,q}$ but $W_{p,0}, W_{0,q}$ are anisotropic. By the assumption, we can take an irreducible limit of discrete series representation $\pi$ of $G_0(W)(\R)$ so that its Harish-Chandra parameter is $(\mu, \Psi)$.  Let $\tau_1$ be an irreducible representation of $G(V_{p,0})(\R)$, and let $\tau_2$ be an irreducible representation of $G(V_{0,q})(\R)$ so that $\tau_1\boxtimes\tau_2$ is the lowest $K$-type of $\sigma$. Then, by Proposition \ref{q=0_desc}, we have both $\Theta_\psi(\tau_1, W)$ and $\Theta_\psi(\tau_2, W)$ are non-zero. Moreover, one can show the assertion \eqref{JHlem} of Lemma \ref{JHlem} although we do not assume $\theta_\psi(\sigma, W) \not=0$. This implies that the tensor product representation $\Theta_\psi(\tau_1, W) \otimes \Theta_\psi(\tau_2, W)$ of $G(W)(\R)$ has a $K$-type whose highest weight is $\mu + \rho(\Psi) - 2\rho(\Delta_c^-)$, and that every irreducible summand of  $\Theta_\psi(\tau_1, W) \otimes \Theta_\psi(\tau_2, W)$ has the infinitesimal character $\chi[\eta]$. Hence, by \cite[Proposition 6.1]{VZ84}, we have 
 \begin{align}\label{argli}
{\rm Hom}^{G(W)}(\Theta_\psi(\tau_1, V)\otimes \Theta_\psi(\tau_2, V), \pi) \not=0.
\end{align}
Since the left-hand side of \eqref{argli} coincides with ${\rm Hom}^{G(W_{0,p})\times G(W_{0,q})}(\Theta_\psi(\pi, V), \tau_1\boxtimes\tau_2)$, we have $\Theta_\psi(\pi, V)$ is non-zero. However, using \cite[Proposition 6.1]{VZ84} again, we have $\theta_\psi(\pi, V)$ is nothing other than $\sigma$. This implies that $\theta_\psi(\sigma, W)$ is non-zero.
Therefore, we finish the proof of Proposition \ref{DS_corr}.


\section{Appendix: Annotation on Fact \ref{pre_R} (ii)}\label{notes on psi}

The local theta correspondence for the dual pair $(G(V_{m,0}), G(-W_{p,q}))$ with $e_\H =1$ and either $p+q = m$ or $m+1$ has been also described by M{\oe}gline \cite{Moe89}. We remark again that $-W_{p,q}$ is a free left module over $\H$ and the signature of $-W_{p,q}^\natural$ is $(2p,2q)$ (see \S\ref{morita}). Moreover, Paul \cite{Pau05} extended it to all symplectic-orthogonal dual pairs of equal or almost equal ranks. However, there is an error in \cite[\S I.4]{Moe89} when quoting the result of \cite{KV78}. The author expects that \cite{Moe89} and \cite{Pau05} are valid if we change the choice of the non-trivial additive character $\psi$ of $\R$ so that $\epsilon_\psi = \sqrt{-1}$, but he have not verified it strictly. 
In the following, we will discuss this further. 

Recall that we put $K_+ = G(V_{m,0})(\R) \cap \GL(\R(i))$ and  $K_- = G(-W_{p,0})(\R) \times G(-W_{0,q})(\R)$. Then, $K_+$ (resp. $K_-$) is a maximal compact subgroup of $G(V_{m,0})(\R)$ (resp. $G(-W_{p,q})(\R)$) containing $S_+(\R)$ (resp. $S_-(\R)$). Let $\sigma_1$ be an irreducible representation of $G(-W_{p,0})(\R)$, let $\sigma_2$ be an irreducible representation of $G(-W_{0,q})(\R)$, let $(a_1, \ldots, a_k, 0, \ldots, 0) \in \Z^p$  be the highest weight of $\tau_1$, and let $(b_1, \ldots, b_l, 0, \ldots, 0) \in \Z^q$ be the highest weight of $\tau_2$. Denote by $\mathscr{K}$ the space of joint harmonics in the Fock model of the Weil representation $\omega_{\psi, \Y}$ of $\Mp(V_{m,0}\otimes (-W_{p,q}))$ where $\Y$ is a maximal isotropic subspace of $V_{m,0}\otimes(-W_{p,q})$ (see Remark \ref{natural_Weil}). M{\oe}gline, taking $\psi$ so that $\epsilon_\psi = -\sqrt{-1}$, asserted the following (\cite[pp.~9]{Moe89}).
\begin{enumerate}[(i)]
\item Let $\tau$ be an irreducible representation of $K_+$. If $\tau\boxtimes(\sigma_1\boxtimes \sigma_2))$ appears in $\mathscr{K}$, then we have $\tau \subset \tau_p'(\sigma_1)\otimes\tau_q(\sigma_2)$. \label{Moe1.4}
\end{enumerate}
However, this is not consistent with Fact \ref{KVfact}. One can verify this in the simplest case $q=0$. 
To resolve this error, we replace the choice of the additive character $\psi$: one can show that the assertion \eqref{Moe1.4} is true if we take $\psi$ so that $\epsilon_\psi = \sqrt{-1}$. Since the argument of the latter part (pp.~10--11) of \cite[\S I.4]{Moe89} do not use $\psi$, we have \cite[Corollary, \S I.4]{Moe89} by replacing $\psi$ so that $\epsilon_\psi = \sqrt{-1}$.


\section{Appendix: Annotation on Fact \ref{mez_fml}} \label{spec_correction}

In the case $F=\R$, assuming the twisted version of Hypothesis \ref{Flem}, Mezo proved the twisted endoscopic character relation by constructing the spectral transfer factors, that is, the coefficients of the trace distributions associated with the irreducible representations in given $L$-packet \cite{Mez13}\cite{Mez16}. In this paper, we use the construction to obtain Langlands parameters from Harish-Chandra parameters. However, there is a sign error in the computations expanding $\Delta_{II}$. In this appendix, we point out the sign error (\S\ref{sign_error}), summarize the updates of the transfer factors (\S\ref{update_tf}), and prove Fact \ref{mez_fml}. 

\subsection{A sign error} \label{sign_error}

Let $G$ be an arbitrary connected reductive group over $\R$. 
We use the notations and terminologies of \cite{Mez13}. In particular,  we choose $a$- and $\chi$- data in the same way as in \cite{Mez13}. We put $\gamma' = \eta_1(x)\gamma_1$ and $\delta' = x\delta$. In \cite[(76)]{Mez13}, the second factor $\Delta_{II}(\gamma', \delta')$ is computed by
\[
\sqrt{-1}^{\dim \mathfrak{u}_{(G^*)^{\theta^*}} - \dim \mathfrak{u}_H}\cdot \frac{\det(1- \Ad \gamma'; \overline{u}_H) \cdot |\det(1 - \Ad {\delta'}^*\theta^*; \overline{u}_{G^*})| }{\det(1 - \Ad {\delta'}^*\theta^*; \overline{u}_{G^*})\cdot |\det(1- \Ad \gamma'; \overline{u}_H)|}.
\]
However, it should be replaced with 
\begin{align}\label{DeltaI}
&(-\sqrt{-1})^{\dim \mathfrak{u}_{(G^*)^{\theta^*}} - \dim \mathfrak{u}_H}\cdot \frac{\det(1- \Ad \gamma'; \overline{u}_H) \cdot |\det(1 - \Ad {\delta'}^*\theta^*; \overline{u}_{G^*})| }{\det(1 - \Ad {\delta'}^*\theta^*; \overline{u}_{G^*})\cdot |\det(1- \Ad \gamma'; \overline{u}_H)|}  \\
&\times  \prod_{\alpha_{res} < 0} \chi_\alpha (N(\alpha({\delta'}^*))). \notag
\end{align}

\subsection{A note on geometric transfer factors}\label{update_tf}
We only discuss the theory of standard endoscopies (i.e. $\theta = \Id$). In this case Langlands and Shelstad \cite{LS87} gave a definition of the relative or absolute geometric transfer factor $\Delta_0$ by
\[
\Delta = \Delta_I\Delta_{II}\Delta_{III}\Delta_{IV}.
\]
Here, $\Delta_I, \Delta_{II}, \Delta_{IV}$ are the factors defined in \cite{LS87}, and we put $\Delta_{III} = \Delta_{III_1}\Delta_{III_2}$ for simplicity. The twisted version was also defined in \cite{KS99}. In \cite{Mez13}, Mezo used this definition. However, some errors were pointed out by Waldspurger, and Kottwitz and Shelstad updated the definition of the geometric transfer factors \cite{KS12}. One of the modified definitions is  
\[
\Delta_I^{-1}\Delta_{II}\Delta_{III}^{-1}\Delta_{IV}
\]
which we will denote by $\Delta_0'$. Here, the definition of $\Delta_I$ is also modified in \cite{KS12}. 
Kaletha's transfer factor $\Delta'$ which we use in this paper to define the Langlands parameter is an appropriate normalization of $\Delta_0'$.

Now, we observe the quotient $\Delta_0'/\Delta$ when $G$ is a quasi-split connected reductive group over $\R$. The modified version of $\Delta_I$ in \cite{KS12} coincides with the original $\Delta_I$ in \cite{LS87} and \cite{KS99} if the base field is $\R$. Moreover, we have $\Delta_I^{-1} = \Delta_I$ in this case. Hence, we have
\begin{align} \label{Delta'}
(\Delta_0' / \Delta)(\gamma, \delta) &= (\Delta_{III}(\gamma, \delta))^{-2} \notag\\
&=\langle (\delta^*, \gamma_1), a_{T'} \rangle^2 \notag\\ 
&= \prod_{\alpha_{res} < 0} \chi_\alpha(N(\alpha({\delta}^*)))^{-1}. 
\end{align}

\subsection{The proof of Fact \ref{mez_fml}}

In this subsection, we assume that $G=G_0(V)$. Put $G^\# = G_0(V_c^\#)$ and take $(z, \varphi) \in \RIT^*(V_c^\#, V)$. According to the character identity \cite[(60)]{Mez13}, the value of the parameter $\iota_\psi[\fa, z, \varphi](\pi)(s)$ is the product of the Kottwitz sign $e(G)$ and the spectral transfer factor $\Delta_{spec}(\phi_{H_1}, \pi)$ that is computed explicitly from the geometric transfer factors \cite[pp.~59]{Mez13}. 
Now, we consider the setting of \ref{mez_fml}. In particular, $\theta = \Id$. Since the center of $G$ is anisotropic, we have $n_\theta = 1$ (for the definition of $n_\theta$, see \cite[pp.~56]{Mez13}) and we have \cite[(115)]{Mez13} is $1$.  Since the Kottwitz sign $e(G)$ is given by $(-1)^{q_G - q_{G^\#}}$ (\cite{Kot83}), we have 
\[
\frac{\sgn(H)}{(-1)^{q^{\sqrt{-1}\mu}}} = (-1)^{q_H - q_G} = e(G)\cdot (-1)^{(q_H - q_{G^\#})}
\]
where $q_H, q_{G^\#}$, and $q_G$ are the symbols as in \S\ref{HCLpar}. Finally, we have $\dim \mathfrak{u}_H = \# \Delta_{B_H}$, and $\dim \mathfrak{u}_G = \#\Delta_B$. Therefore, by taking \S\S\ref{sign_error}--\ref{update_tf} into account, we have   
\begin{align}\label{new_fml}
\iota_\phi[\fa, z, \varphi](\pi)(s) &= (-1)^{q_H - q_{G^\#}}\cdot(-\sqrt{-1})^{\dim \mathfrak{u}_{G} - \dim \mathfrak{u}_H} \notag \\ 
&\quad \times \epsilon_L(G, H ; \psi)\cdot \Delta_I(\gamma_1, \delta_g)\cdot \la \inv_z^\varphi(\delta_g, \delta_h), (\Ad g)^\wedge(s) \ra.
\end{align}
We remark here that the products of the values of the $\chi$-data appearing in \eqref{DeltaI} and \eqref{Delta'} cancel each other.


\bibliographystyle{alpha}
\bibliography{lrc.bib}

\end{document}